\documentclass[10pt,a4paper]{article}
\title{{Existence and Asymptotic Behavior of Solutions to a Semilinear Hyperbolic-Parabolic Model of Chemotaxis  }}
\date{}
\author{Cristiana Di Russo$^{1}$}

\usepackage[T1]{fontenc}
\usepackage[english]{babel}                                
\usepackage[utf8]{inputenc}        

\usepackage{amsfonts}
\usepackage{graphicx}
\usepackage{amsthm}
\usepackage{fancyhdr}
\usepackage[square,numbers]{natbib}
\usepackage{amsmath,amssymb,amscd,amsthm}
\usepackage{multicol,longtable,tabularx}
\usepackage{enumerate}
\usepackage{latexsym}
\usepackage{float}
\usepackage{amstext,graphicx,color}
\usepackage[english]{babel}
\usepackage{ifthen}
\usepackage{verbatim}
\usepackage{calc}
\usepackage{amsmath,amssymb,amsthm,mathrsfs}

\usepackage{multimedia,color,stmaryrd}

\allowdisplaybreaks

\newcommand{\bb}{\frac{1}{\beta}}

\newcommand{\R}{{\mathbb R}}

\newtheorem{theorem}{Theorem}[section]

\newtheorem{lemma}[theorem]{Lemma}
\newtheorem{defn}[theorem]{Definition}
\newtheorem{proposition}[theorem]{Proposition}
\newtheorem{rmk}[theorem]{Remark}

\begin{document}

\setcounter{page}{1}

\maketitle
\vspace{-0.8cm}
\begin{center} 
{ $^1$ Laboratoire MAPMO UMR CNRS 7349, Universit\'e d'Orl\'eans, F\'ed\'eration Denis Poisson,\\ UFR Sciences, B\^atiment de math\'ematiques,
B.P. 6759 -F-45067 Orl\'eans cedex 2,
France \vspace{0.1cm}}
\end{center}
\vspace{0.3cm}
\begin{abstract}
We consider a general hyperbolic-parabolic model of chemotaxis in the multidimensional case. For this system we show the global existence of smooth solutions to the Cauchy problem and we determine their asymptotic behavior.  Since this model does not enter in the classical framework of dissipative problems, we analyze it combining the features of the hyperbolic and the parabolic parts and using detailed decay estimates of the Green function. 
\end{abstract}

\begin{small} \textbf{Keywords} : Chemotaxis, hyperbolic-parabolic systems, dissipativity, asymptotic behavior.

\end{small}

\pagestyle{myheadings}
\thispagestyle{plain}
\markboth{C. DI RUSSO}{A Hyperbolic-Parabolic Model of Chemotaxis}

\section{Introduction}
Chemotaxis, the movement of cells in response to a chemical substance, is decisive in many biological processes and determines how cells arrange and organize themselves. For example, the formation of cells aggregations (amoebae, bacteria, etc)
occurs during the response of the species populations to the change in the environment of the chemical concentrations.
In multicellular organisms instead, chemotaxis of cell populations plays a crucial role throughout
the life cycle: during embryonic development it is involved in organizing
cell positioning, e.g. during gastrulation  and patterning of the
nervous system; in the adult life, it directs immune cell migration to sites of inflammation and fibroblasts into wounded regions to initiate healing. These same
mechanisms are used during cancer growth, allowing tumor cells to invade the
surrounding environment or stimulating new blood vessel growth  \cite{Mu2}.\\
The movement of bacteria under the effect of chemotaxis has been a widely studied topic in Mathematics in the last decades, and numerous models have been proposed. Moreover it is possible to describe this biological phenomenon at different scales. For example, by considering the population density as a whole, it is possible to obtain macroscopic models of partial differential equations. One of the most celebrated model of this class is the one proposed by Patlak in 1953 \cite{Pa} and subsequently by Keller and Segel in 1970 \cite{KeSe3}.

In the Patlak-Keller-Segel (PKS) system, the evolution of density of bacteria is
described by a parabolic equation, and the density of chemoattractant is generally driven by a
parabolic or an elliptic equation.
The behavior of this reaction-advection-diffusion system is now quite well-known: in the
one-dimensional case, the solution is always global in time. In several space dimensions, in the parabolic elliptic case, if
initial data are small enough in some norms, the solution will be global in time and rapidly decaying in time; while on the
opposite, it will explode in finite time at least for some large initial data.

The simplicity, the analytical tractability, and the capacity to replicate some of the
key behaviors of chemotactic populations are the main reasons of the success of this model of chemotaxis. In particular, the ability to display
auto-aggregation, has led to its prominence as a mechanism for self-organization of
biological systems.\\
Moreover, there exists a lot of variations of PKS model to describe  biological processes in which chemotaxis is involved. They differ in the functional forms of the three main mechanisms involved: the sensing of the chemoattactant, which has an effect on the oriented movement of the species,
the production of the chemoattractant by a mobile species or an external source, and the degradation of the chemoattractant by a mobile species or an external effect.\\
However, the approach of PKS model is not always sufficiently precise to describe the
biological phenomena \cite{FiLaPe}. As a matter of fact, the diffusion can lead to fast dissipation or explosive behaviors
and prevents us to observe intermediate organized structures. Moreover it is not able to reproduce
the ``run and tumble'' behavior, the movement along straight lines, the sudden stop and the change of direction, typical of bacteria like E.Coli.

The main reason is that this approach describe processes on a long time scale, while for short time range one gets better a description from models with finite characteristic speed.\\
Kinetic transport equations describe quite well the movement of a single organism. For example the ``run and tumble''  can be described by the velocity-jump process \cite{HiOt,Ste}.

At an intermediate scale between diffusion and kinetic models we can find hyperbolic models.
This class of models can be derived as a fluid limit of transport equations but with a different
scaling, namely the hydrodynamic scaling $t\rightarrow \epsilon t$, $x\rightarrow \epsilon x$ \cite{ChaMarPeSc}.\\
Starting from a transport equation for the chemosensitive movements, in \cite{Hi5} Hillen shows a kinetic derivation of hyperbolic models by the moment closure method, thus obtaining the Cattaneo model for chemosensitive movement.
Using the first two moments he obtains the following hyperbolic-parabolic model:
\begin{equation}
\label{sistema_I}
\left\{
\begin{array}{l}
\partial_{t} u +\nabla \cdot v = 0, \\\\
\partial_{t} v + \gamma^2\nabla u= -v+h(\phi,\nabla \phi)g(u),\\\\
\partial_{t} \phi =\Delta \phi +a u-b\phi.
\end{array}
\right.
\end{equation}
where $x\in \mathbb{R}^n$, $t\geq0 $, $u$ is the population density, $v$ are the fluxes, $\phi$ is the concentration of chemical species, and the source terms $g,h$ are smooth functions.

We start our analytical study by considering the semilinear hyperbolic-parabolic system 
\begin{equation}
\label{sistema_GUMA}
\left\{
\begin{array}{l}
\partial_{t} u +\nabla \cdot v = 0, \\\\
\partial_{t} v + \gamma^2\nabla u= -b(\phi,\nabla\phi)v+h(\phi,\nabla \phi)g(u),\\\\
\partial_{t} \phi =\Delta \phi +f(u,\phi),
\end{array}
\right.
\end{equation}
which generalize the one proposed by Dolak and Hillen in  \cite{DoHi}. The parameter $\gamma$ is the characteristic speed of propagation of the cells and the source terms $b,h,g$, and $f$ are smooth functions.

The coupling of hyperbolic and parabolic equations has been widely studied by Kawashima and Shizuta \cite{Ka,KaShi2,ShiKa}. Under the smallness assumption on the initial data and the dissipation condition on the linearized system, they were able to prove global (in time) existence and asymptotic stability of smooth solutions to the initial value problem for a general class of symmetric hyperbolic-parabolic systems.\\
System \eqref{sistema_I} does not enter in this framework. As a matter of fact, due to the presence of the source term $a u$, the dissipative condition fails.\\
With reference to the one dimensional case, a first result of local and global existence for weak solutions, under the assumption of turning rate's boundness, was proved in \cite{HiSte}. Subsequently Guarguaglini et al. in \cite{GuMaNaRi} have proved more general results for this model under weaker hypotheses, by showing a general result of global stability
of zero constant states for the Cauchy problem and of general constant state for the
Neumann problem. These results have been obtained using the linearized
operators, and the accurate analysis of their nonlinear perturbations.

In order to obtain our global existence result for the multidimensional case we follow this approach. The basic idea is to consider the hyperbolic and parabolic equation ``separately'', and to take advantage of their respective properties. 

Thanks to the Green function of the heat equation $\Gamma^p$, and the Duhamel's formula, we know that, the solution to the parabolic equation is:
$$
\phi(x,t)=(e^{-bt}\Gamma^p(t)\ast \phi_0)(x)+\int_0^t e^{-b(t-s)}\Gamma^p(t-s)\ast  (a u(s))ds.
$$
On the other hand for the damped wave equation,
\begin{equation}
\label{sistema_II}
\left\{
\begin{array}{l}
\partial_{t} u +\nabla \cdot v = 0, \\\\
\partial_{t} v + \gamma^2\nabla u= -v,
\end{array}
\right.
\end{equation}
we have, by the theory of dissipative systems \cite{SiThoWa}, that the presence of the dissipative term $-v$ enforces a faster decay of the solution.
This implies that we can write the solution of the hyperbolic part of system \eqref{sistema_I}, as
$$
w(x,t)=(\Gamma^h(t)\ast w_0)(x)+\int_0^t \Gamma^h(t-s)\ast H(\phi,\nabla\phi,w)(s)ds,
$$
where $w=(u,v)$, $H(\phi,\nabla\phi,w)=[0,h(\phi,\nabla\phi)g(u)]^t$ and $\Gamma^h$ is the Green function of the damped wave equation \eqref{sistema_II}.

Our strategy has been to use the decomposition of the Green function of dissipative hyperbolic systems done by Bianchini at al. \cite{BiHaNa} and its precise
decay rates.\\ Indeed  in \cite{BiHaNa} the authors proposed a detailed description of the multidimensional Green function for a class of partially dissipative systems.
They analyzed the behavior of the Green function for the linearized problem, decomposing it into two main terms. The first
term is the diffusive one, and consists of heat kernel, while the faster term consists of the hyperbolic part. Moreover they gave a more precise description of the behavior of the diffusive part,
which is decomposed into four blocks, which decay with different decay rates. They showed that solutions have canonical projections
on two different components: the conservative part and the dissipative part. The
first one, which formally corresponds to the conservative part of the equations, decays in time like the heat kernel, since it corresponds to the diffusive
part of the Green function. On the other side, the dissipative part is strongly influenced
by the dissipation and decays at a rate $t^{-\frac{1}{2}}$ faster than the conservative one.\\
By these refined estimates we were able to prove global existence of smooth solutions for small initial data, and to determine at the same time their asymptotic behavior.

We are able to show decay rates of the $L^\infty$-norm of solution of order $O(t^{-\frac{n}{2}})$, faster than the one obtained in \cite{GuMaNaRi} which was $O(t^{-\frac{n}{4}})$. \\
Moreover we show the global existence, and we determinate the asymptotic behavior of solutions, also for perturbation of non-zero constant stationary states in the case of simpler source terms. In order to prove this result, we need to adapt the decay estimates of the Green function to compensate the lack of polynomial decay of linear term in the hyperbolic equations.

The parabolic and hyperbolic models of chemotaxis are expected to have the same behavior for long time. We investigate this aspect analytically and
 we show that the difference between the solution of PKS model and the hyperbolic one decays with a rate of $O(t^{-\frac{n}{2}})$ in $L^2$, so $t^{-\frac{n}{4}}$ faster than the decay of solutions themselves if $n\leq3$, otherwise we get a decay $t^{-\frac{1}{2}}$ faster than the decay of solutions.\\

The article is organized as follows: in the first section, we review some properties of partially dissipative hyperbolic systems, we recall the results obtained by Bianchini et al. in \cite{BiHaNa} about the asymptotic behavior of their smooth solutions, and the local existence in time for smooth solutions to system \eqref{sistema_GUMA} to the Cauchy problem.
Subsequently, in Section 3, we are able to prove the  global existence result thanks to the refined decay estimates of the Green Kernel of hyperbolic equations. \\
In Section 4 we study the case of perturbation of non-zero constant stationary state.
For large time hyperbolic and parabolic model are expected to have the same behavior. Then,
in the last section,

we examine the difference between solutions to the hyperbolic-parabolic system \eqref{sistema_GUMA} and to the related PKS model, showing that this difference decays with a faster rate.

\section{Background}
\subsection{Partially Dissipative Hyperbolic Systems}\label{PrimaSezione}
In this first section we recall some properties of hyperbolic dissipative systems.\\
Let us focus our attention on the following multidimensional system of balance laws
\begin{equation}
\label{sistema_Hyp}
\left\{
\begin{array}{l}
\partial_{t} u +\nabla \cdot v = 0, \\\\
\partial_{t} v + \gamma^2\nabla u= -\beta v,
\end{array}
\right.
\end{equation}
where $u: \mathbb{R}^n\times \mathbb{R}^+ \rightarrow \mathbb{R}^+$, $v=(v_1,v_2,\ldots, v_n):\mathbb{R}^n\times \mathbb{R}^+ \rightarrow \mathbb{R}^n$,
with initial conditions 
$$
u(x,0)=u_0(x), \quad v(x,0)=v_0(x).
$$
We can observe that since \eqref{sistema_Hyp} is equivalent to damped wave equation, the behavior of the solutions to the Cauchy problem for this system  is quite well known \cite{Ka}.
Moreover system \eqref{sistema_Hyp} belongs to the class of dissipative hyperbolic systems.

It is possible to rewrite system \eqref{sistema_Hyp} in a compact form as
\begin{equation}\label{dissi_compa}
\partial_t w+\sum_{j=1}^n A_j\partial_{x_j}w= g(w),
\end{equation}
where  $w=(u,v)\in \Omega \subseteq \mathbb{R}\times \mathbb{R}^n$, and 
\begin{equation*}
A_j=\left(\begin{array}{lr}
0 & e_j\\
 \gamma^2 e_j^t & 0  
\end{array}\right),
\end{equation*}
with $(A_j)_{11} \in \mathbb{R}^{1\times 1}$, $(A_j)_{12} \in \mathbb{R}^{n\times 1}$, $(A_j)_{21} \in \mathbb{R}^{1\times n}$, and $(A_j)_{22} \in \mathbb{R}^{n\times n}$ and $e_j$ is the canonical $j-$th vector of $\mathbb{R}^n$. Here we denote the source term by 

\begin{equation*}
\quad g(w)=
\left(
\begin{array}{c}
0\\
q(w)
\end{array}
\right)=\left(
\begin{array}{c}
0\\
-\beta v
\end{array}
\right), \quad \textrm{ with } q(w)\in \mathbb{R}^n.
\end{equation*}

The initial condition reads
\begin{equation}\label{dissi_ini}
w(x,0)=w_0(x).
\end{equation}

By the introduction of new variables $W=(W_1,W_2)$, with
$$
W_1=u, \quad W_2=\frac{v}{\gamma^2},
$$ 
and a symmetric positive definite matrix $A_0$, defined as 
\begin{equation}\label{matrici}
A_0=\left(\begin{array}{cc}
 I & 0\\
 0 & \gamma^2 I 
\end{array}\right),
\end{equation}
it is possible to symmetrize system \eqref{dissi_compa}.
Selecting  $W$ as new variable, our system reads 
$$
A_0(W)\partial_t W + \sum_{j=1}^n \bar{A}_j\partial_{x_j} W = G(\Phi(W)).
$$
where 
\begin{equation*}\label{matrici}
\bar{A}_j:=A_j A_0(W)=\left(\begin{array}{cc}
0 & \gamma^2 e_j\\
\gamma^2 e_j^t & 0  
\end{array}\right),
\end{equation*}
and $G(W)=g(\Phi(W))=(0, Q(W))^t$.
Let us notice that,  for every $j=1,\ldots, n$, the matrix $\bar{A}_j$ is symmetric.

In order to continue the analysis of smooth solutions for dissipative hyperbolic system let us introduce the condition of Shizuta and Kawashima (SK) \cite{ShiKa} for hyperbolic systems. 
\begin{defn}\label{SK}
System \eqref{dissi_compa} verifies condition (SK), if every eigenvector of $\sum\limits_{j=1}^n A_j\xi_j$ is not in the null space of $Dg(0)$ for every $\xi\in \mathbb{R}^{n +1}\backslash\{0\}$.
\end{defn}

We can observe that system \eqref{dissi_compa} verifies the Kawashima condition since, given an equilibrium state 
\[
\left( 
\begin{array}{c}
u\\
0
\end{array}
\right) \in \mathbb{R}^{n+1}, 
\]
then the generic vector
\[
\left( 
\begin{array}{c}
X\\
0
\end{array}
\right) \in \mathbb{R}^{n+1}, 
\]
is eigenvector of $\sum\limits_{j=1}^nA_j\xi_j$, if and only if $X=0$.

With reference to the existence of smooth solutions to system \eqref{dissi_compa}, we recall the following result, which is a special case of the results in \cite{BiHaNa}. 
\begin{theorem}\label{existence}
Let us consider the Cauchy problem \eqref{dissi_compa}-\eqref{dissi_ini}. Let $s \geq 0$. For every $w_0\in H^s(\R^n)$, there is a unique global solution $w$ to \eqref{dissi_compa}-\eqref{dissi_ini} which verifies
$$
w\in C^0([0,\infty);H^s(\mathbb{R}^{n}))\cap C^1([0,\infty);H^{s-1}({\mathbb{R}^{n}})), 
$$
and such that,
$$
\sup_{0\leq t<+\infty} \|w(t)\|_{H^s}^2+\int_0^{+\infty}\|v(\tau)\|_{H^s}^2 d\tau \leq C\|w_0\|_{H^s}^2,
$$
where $C$ is a positive constant.
\end{theorem}

The refined estimates of the Green Kernel of system \eqref{dissi_compa} proposed by Bianchini et al. \cite{BiHaNa}, holds for linearized dissipative system in the Conservative-Dissipative form.
Then, we rewrite system \eqref{dissi_compa} in this particular form, which will be useful in our study.\\
Let us consider a linear  system with constant coefficients
\begin{equation}\label{cons2}
w_t+\sum_{j=1}^n A_j w_{x_j}=Bw,
\end{equation}
where $w=(w_1,w_2)\in \mathbb{R}\times \mathbb{R}^{n}$. 

\begin{defn}\label{Def_ConsDiss}
System \eqref{cons2} is in Conservative-Dissipative form (C-D form) if it is symmetric, i.e. $A^t_{j}=A_j$ for all $j=1,\ldots,n$, and 
there exists a negative definite matrix $D\in \mathbb{R}^{n\times n}$, such that
$$
B=\left(
\begin{array}{cc}
 0 &  0\\
0 & D
\end{array}
\right).
$$
\end{defn}
In this case $w_1$ is called the conservative variable, while $w_2$ is the dissipative one.\\
Under suitable assumptions every symmetrizable dissipative system can be rewritten in the C-D form.
Let us observe that system \eqref{dissi_compa} can be easily written in the Conservative-Dissipative form by a change of variable.

Set 
$$
M=\left(
\begin{array}{cc}
I & 0 \\
0 & \gamma^{-1} 
\end{array}
\right),
$$
and define the matrices of the C-D form
$$
\tilde{A}_j=\left(
\begin{array}{cc}
0 &\gamma e_j\\ \
\gamma e_j^t &0
\end{array}
\right), \quad \tilde{B}= \left(
\begin{array}{cc}
0 & 0 \\
0 & -\beta I
\end{array}
\right). 
$$ 
Setting
$$
\left(
\begin{array}{c}
\tilde{w}_1\\
\tilde{w}_2
\end{array}
\right)=M\left(
\begin{array}{c}
u\\
v
\end{array}
\right)= \left(
\begin{array}{c}
u\\
\frac{v}{\gamma}
\end{array}
\right)
$$
and reporting in \eqref{sistema_Hyp}, we obtain the conservative-dissipative form for this system
\begin{equation*}
\label{C-D}
\left\{
\begin{array}{l}
\partial_{t} \tilde{w}_1 +\gamma \nabla \cdot  \tilde{w}_2 = 0, \\\\
\partial_{t} \tilde{w}_2 + \gamma \nabla \tilde{w}_1= -\beta \tilde{w}_2.
\end{array}
\right.
\end{equation*}
We will consider by now the Conservative-Dissipative form of system \eqref{sistema_Hyp} written as:
\begin{equation}
\label{C-D_uv}
\left\{
\begin{array}{l}
\partial_{t} u +\gamma \nabla \cdot  v = 0, \\\\
\partial_{t} v + \gamma \nabla u= -\beta v.
\end{array}
\right.
\end{equation}

\subsubsection{The Multidimensional Green Function}\label{MultiGreen}
We present now the results on the study of the Green Kernel $\Gamma^h(x,t)$ of multidimensional dissipative hyperbolic systems done by Bianchini et al. in \cite{BiHaNa}. In their work the authors analyzed the behavior of the Green function for linearized problems, which has been decomposed into two main terms. The first term, the diffusive one, consists of heat kernels, while the faster term consists of the hyperbolic part.\\
In general, the form of the Green function is not explicit, but it is possible to deal with its Fourier transform. The separation of the Green kernel into various parts is done at the level of a solution operator $\Gamma(t)$ acting on $L^1\cap L^2(\mathbb{R}^n,\mathbb{R}^{n+1})$.

They proved the following theorem, \cite{BiHaNa}:
\begin{theorem}\label{Theo_BHN}
Consider the linear PDE in the conservative-dissipative form
\begin{equation*}\label{theo_CD}
\partial_t w+\sum_{j=1}^n A_j \partial_{x_j} w= Bw,
\end{equation*}
where $A_j$, $B$ satisfy the assumption (SK), and let $Q_0=R_0L_0$, $Q_{-}=I-Q_0=R_{-}L_{-}$ be the eigenprojectors on the null 
space and the negative definite part of $B$ with  $L_0=R_0^T=[I_{n_1} \; 0 ]$ and $L_-=R_-^T=[0 \; I_{n_2}]$.

Then, for any function $w_0\in   L^1 \cap L^2 (\mathbb{R}^n,\mathbb{R}^{n+1})$  the solution of the linear dissipative system can be decomposed as
$$
w(t)=\Gamma^h(t)\ast w_0=K(t)w_0+\mathcal{K}(t) w_0,
$$  
where  for any multi index $\beta$ and for every $p\in [2,+\infty]$, the following estimates hold.\\
$K(t)$ estimates:
\begin{eqnarray*}\label{k_estimates}
\|L_0 D^\beta K(t) w_0\|_{L^p}&\leq& C(|\beta|)\min\{1, t^{-\frac{n}{2}(1-\frac{1}{p})-\frac{|\beta|}{2}}\}\|L_0 w_0\|_{L^1}\\
&+&C(|\beta|)\min\{1, t^{-\frac{n}{2}(1-\frac{1}{p})-\frac{1}{2}-\frac{|\beta|}{2}}\}\|L_{-} w_0\|_{L^1},\nonumber\\\nonumber\\
\|L_{-} D^\beta K(t) w_0\|_{L^p}&\leq& C(|\beta|)\min\{1, t^{-\frac{n}{2}(1-\frac{1}{p})-\frac{1}{2}-\frac{|\beta|}{2}}\}\|L_0 w_0\|_{L^1}\\
&+&C(|\beta|)\min\{1, t^{-\frac{n}{2}(1-\frac{1}{p})-1-\frac{|\beta|}{2}}\}\|L_{-} w_0\|_{L^1}\nonumber.
\end{eqnarray*}
$\mathcal{K}(t)$ estimates:
\begin{equation*}\label{K_estimates}
\|D^{\beta}\mathcal{K}(t) w_0\|_{L^2}\leq C e^{-ct}\|D^{\beta}w_0\|_{L^2}.
\end{equation*}
\end{theorem}

\subsection{Local Existence of Smooth Solutions}\label{Locale}
Since our aim is to prove the global existence of smooth solutions with small initial data to the complete hyperbolic-parabolic system \eqref{sistema_GUMA}, a sharp results of local existence of solutions in essential for our proof. Let us consider 
a more general semilinear hyperbolic-parabolic system
\begin{equation}
\label{sistema_generaleC}
\left\{
\begin{array}{l}
\partial_{t} u +\gamma \nabla \cdot v = F_1(u,v,\phi,\nabla \phi), \\\\
\partial_{t} v + \gamma \nabla u= F_2(u,v,\phi,\nabla \phi),\\\\
\partial_{t} \phi =\Delta \phi +F_3(u,v,\phi,\nabla \phi),
\end{array}
\right.
\end{equation}
where $u,\phi: \mathbb{R}^n\times \mathbb{R}^+ \rightarrow \mathbb{R}^+$, $v=(v_1,v_2,\ldots, v_n):\mathbb{R}^n\times \mathbb{R}^+ \rightarrow \mathbb{R}^n$, $F_1, F_3: \mathbb{R}^n\times \mathbb{R}^+ \rightarrow \mathbb{R}$ and $F_2: (F_2^1,\ldots, F_2^n):\mathbb{R}^n\times \mathbb{R}^+ \rightarrow \mathbb{R}^n$, with $F_i(0)=0$.
We complement the system with the initial conditions
\begin{equation}\label{dati_cauchyL}
u(x,0)=u_0(x), \quad v(x,0)=v_0(x), \quad \phi(x,0)=\phi_0(x),
\end{equation}
and with the regularity assumptions   
\begin{equation}\label{Ass_dati_cauchy}
u_0, v_0 \in H^s(\mathbb{R}^n), \quad \phi_0 \in H^{s+1}(\mathbb{R}^n).
\end{equation}

With reference to the local existence of smooth solutions to system (\ref{sistema_generaleC}), we recall the following result:
\begin{theorem}\label{locale}
There exists $t^*>0$, only depending on initial data, such that, under the assumptions that, for $i=1,2,3$, $F_i$ are locally Lipschitz maps,  problem \eqref{sistema_generaleC}-\eqref{dati_cauchyL}-\eqref{Ass_dati_cauchy}, has a unique local solution 
$$ 
w=(u,v)\in C([0,t^*),H^s(\mathbb{R}^n)), \quad  \phi\in C([0,t^*),H^{s+1}(\mathbb{R}^n)).
$$
\end{theorem}
This theorem can be proved with a standard fixed point method \cite{Di2}, and it is a special case of Theorem 2.9 in \cite{Ka}.

\section{The Cauchy Problem}
At the beginning of this section we recall some results which will be useful to establish the existence of global solutions to the more specific problem
\begin{equation}
\label{sistema3}
\left\{
\begin{array}{l}
\partial_{t} u +\gamma\nabla \cdot v = 0, \\\\
\partial_{t} v + \gamma\nabla u= -b(\phi,\nabla\phi)v+h(\phi,\nabla \phi)g(u),\\\\
\partial_{t} \phi =\Delta \phi +f(u,\phi),
\end{array}
\right.
\end{equation}
with the initial conditions 
\begin{equation}\label{dati_cauchy}
u(x,0)=u_0(x), \quad v(x,0)=v_0(x), \quad \phi(x,0)=\phi_0(x),
\end{equation}
and the regularity assumptions
\begin{equation}\label{Ass_Dati_2}
u_0,v_0\in H^s(\R^n)\cap L^1(\R^n), \quad \phi_0\in H^{s+1}(\R^n).
\end{equation}
In order to prove our results we make some assumptions on the functions $b,f,g,h$ on the right hand side in system \eqref{sistema3}.

\begin{description}
\item[($H_b$):] $b\in C^{s+1}(\mathbb{R}^{n+1})$ and
$$
b(z,w)=\beta+\bar{b}(z,w),\\
$$
where $\beta>0$, and for all fixed $K>0$
$$
|\bar{b}(z,w)|\leq B_k(|z|+|w|) \quad \textrm{ for all } z,w \in [-K,K],
$$
where $B_k$ is a suitable constant depending on $K$.
\item[($H_h$):]  $h\in C^{s+1}(\mathbb{R}^{n+1})$ and $h(0,0)=0$. In particular for all fixed $K>0$
with
$$
|h(z,w)|\leq H_k(|z|+|w|) \quad \textrm{ for all } z,w \in [-K,K],
$$
where $H_k$ is a suitable constant depending on $K$.
\item[($H_g$):]  $g\in C^{s+1}(\mathbb{R})$ and $g(0)=0$. For all fixed $K>0$
with
$$
|g(z)|\leq G_k|z| \quad \textrm{ for all } z\in [-K,K],
$$
where $G_k$ is a suitable constant depending on $K$.
\end{description}

Let us notice that under the assumptions $(H_h, H_g)$ this general sensitivity function, $h(\phi,\nabla\phi)g(u)$, covers different possible relations between species and chemical substance present in chemotaxis models as reported in \cite{HiPa}.

\begin{description}
\item [($H_f$):] $f\in C^{s+1}(\mathbb{R}^2)$ and
$$
f(z,w)=az-bw+\bar{f}(z,w),\\
$$
where $a,b>0$, and for all fixed $K>0$,
$$
|\bar{f}(z,w)|\leq F_k(|z|^2+|w|^2) \quad \textrm{ for all } z,\, w\in [-K,K],
$$
where $F_k$ is a suitable constant depending on $K$.
\end{description}
By these assumptions, we are led to consider the system
\begin{equation}
\label{sistema_H}
\left\{
\begin{array}{l}
\partial_{t} u +\gamma \nabla \cdot v = 0,\\ \\
\partial_{t} v + \gamma \nabla u= -\beta v-\bar{b}(\phi,\nabla\phi)v+h(\phi,\nabla \phi)g(u),\\\\
\partial_{t} \phi =\Delta \phi +au-b\phi +\bar{f}(u,\phi).
\end{array}
\right.
\end{equation}
It is possible to rewrite the above system as 
\begin{equation}
\label{sistema_H_com}
\left\{
\begin{array}{l}
\partial_{t} w +\sum\limits_{j=1}^n A_j\partial_{x_j} w = B w+\bar{B}(\phi,\nabla\phi)w+H(\phi,\nabla \phi,w),\\\\
\partial_{t} \phi =\Delta \phi +au-b\phi +\bar{f}(u,\phi),
\end{array}
\right.
\end{equation}
where
$$
A_j=\left(
\begin{array}{cc}
0 &\gamma e_j\\ \
\gamma e_j^t &0
\end{array}
\right),\quad B= \left(
\begin{array}{cc}
0 & 0 \\
0 & -\beta 
\end{array}
\right), 
$$
and 
$$
\bar{B}(\phi,\nabla\phi)=\left(
\begin{array}{c}
0 \\
-\bar{b}(\phi,\nabla\phi)
\end{array}
\right),\quad H(\phi,\nabla\phi,w)= \left(
\begin{array}{c}
0 \\
h(\phi,\nabla\phi)g(w)
\end{array}
\right). 
$$
Thanks to the regularity of source terms, the local Lipschitz condition yields. Then we can apply Theorem \ref{locale} and deduce the local existence of solution to \eqref{sistema_H}. \\
Before proceeding in our study of global existence of solutions we recall some well-known inequalities in the Sobolev spaces \cite{Ta3}.
\begin{proposition}\label{Propo1}
Let $u,v \in H^{s}(\mathbb{R}^n)\cap L^\infty(\mathbb{R}^n), \quad s>0, \quad |\beta|\leq s$, then
\begin{eqnarray*}\label{Taylor}
\|D^\beta(uv)\|_{L^2}\leq c(\|u\|_{L^\infty}\|D^\beta v\|_{L^2} +\|v\|_{L^\infty}\|D^\beta u\|_{L^2}).
\end{eqnarray*}
If $u,v \in H^{s+|\beta|}(\mathbb{R}^n)$,
$$
\|D^\beta(uv)\|_{H^s}\leq c(\|u\|_{L^\infty}\|D^\beta v\|_{H^s} +\|v\|_{L^\infty}\|D^\beta u\|_{H^s}),
$$
if $\beta=0$, then
$$
\|uv\|_{L^2}\leq \|u\|_{L^2}\|v\|_{L^\infty}.
$$
\end{proposition}

\begin{proposition}\label{F}
Let $F$ be smooth and assume $F(0)=0$. Then, for $u \in H^{s}(\R^n)\cap L^\infty(\mathbb{R}^n)$
\begin{eqnarray*}\label{Taylor_bis}
\|F(u)\|_{H^s}\leq C_s(\|u\|_{L^\infty})(1+\|u\|_{H^s}).
\end{eqnarray*}
\end{proposition}
\begin{proposition}\label{Prop_h}
Let $u\in H^s(\mathbb{R}^n)\cap L^{\infty}(\mathbb{R}^n)$ $(s\geq 1)$ such that there exists $\gamma_0>0$ that for $(x,t)\in \mathbb{R}^n\times [0,+\infty)$,
$$
|u(x,t)|\leq \gamma_0.
$$
Then for every smooth function $h$
\begin{equation*}\label{prop}
\|D^\beta h(u)\|_{L^2}\leq C_{\beta}\|h'\|_{C^{|\beta|-1}(|u|\leq \gamma_0)}\|u\|_{L^\infty}^{|\beta|-1}\|D^\beta u\|_{L^2},
\end{equation*} 
with $\beta\neq 0$, $|\beta|\leq s$.
\end{proposition}

\subsection{Continuation Principle}
Now we are going to prove the existence of global solutions to problem \eqref{sistema_H}-\eqref{dati_cauchy}-\eqref{Ass_Dati_2} using the following Continuation Principle.
 
\begin{proposition}\label{continuation}
Let $T<+\infty$ be the maximal time of existence for a local solution $(w,\phi)$ to system \eqref{sistema_H}-\eqref{dati_cauchy}-\eqref{Ass_Dati_2}.
Then
$$
\limsup_{t\rightarrow T^-} \|w(t)\|_{H^s}+\|\phi(t)\|_{H^{s+1}}=+\infty.
$$
\end{proposition}

\begin{proof}
Let $(w,\phi)$ be a given local smooth solution on a maximal time interval $(0,T_{max})$. \\
Let $T>T_{max}$ and assume there exists an a priori bound
$$
R:=\sup_{(0,T)}\max \left\{ \|\phi\|_{H^{s+1}},\|w\|_{H^s}\right\}.
$$
Let $t_R>0$ be the maximal time of existence of solutions to the Cauchy problem, with 
$\|w_0\|_{H^s}$, $\|\phi_0\|_{H^{s+1}}\leq R$.
Then, there exists $\bar{t}\in \left(T-\frac{t_R}{2},T\right)$ such that, we can consider the functions $w(x,\bar{t})\in H^s(\mathbb{R}^n)$ and $\phi(x,\bar{t})\in H^{s+1}(\mathbb{R}^n)$ as initial data for a new Cauchy problem, with maximal time of existence $\overline{T}=\bar{t}+t_R>T_{max}$, and we find a contradiction.
\end{proof}

From the previous result, it is enough to estabilish an a priori $H^s$, $H^{s+1}$ bound to give the global existence.
Beside we can notice that to prove the global existence result, it is enough to prove the boundness of $L^\infty$-norm of functions $(w,\phi)$, as showed by the following Lemma.
\begin{lemma}\label{boundinf}
Let $(w,\phi) \in C([0,t),(H^s(\R^n))\times C([0,t),H^{s+1}(\R^n)) $ a solution of (\ref{sistema_H}) for $\, 0\leq t\leq T$, where $\|w(t)\|_{L^\infty}$, $\| \phi(t)\|_{W^{1,\infty}}\leq K$, then there will exist a constant $C_k$ such that,
$$
\|w(t)\|_{H^s}+\|\phi(t)\|_{H^{s+1}}\leq c(\|w_0\|_{H^s}+\|\phi_0\|_{H^{s+1}})e^{C_k (t)}, \qquad 0\leq t\leq T. 
$$
\end{lemma}

\begin{proof}
Let $\|w(t)\|_{L^\infty},\| \phi(t)\|_{W^{1,\infty}}\leq K$, then we want to prove that $H^s$ norms of these functions are bounded.\\
Thanks to the Duhamel's formula we can write the solution of the hyperbolic part as
$$
w(x,t)=(\Gamma^h(t)\ast w_0)(x)+\int_0^t \Gamma^h(t-s)\ast (\bar{B}(\phi,\nabla\phi)(s)w(s)+H(\phi,\nabla\phi,w)(s))ds,
$$
where $\Gamma^h$ is the Green function of system \eqref{C-D_uv}.
Then 
\begin{eqnarray*}
\|w(t)\|_{H^s} &\leq& C\|w_0\|_{H^s}+ \int_0^t C\|\bar{B}(\phi,\nabla\phi)(s)w(s)\|_{H^s}+ \|H(\phi,\nabla\phi,w)(s)\|_{H^s}ds,
\end{eqnarray*}

by Proposition \ref{Propo1} we deduce
\begin{align*}
\|w(t)\|_{H^s} \leq& C\|w_0\|_{H^s}+C\int_0^t (\|\bar{b}(\phi,\nabla\phi)(s)\|_{L^\infty}\|w(s)\|_{H^s}+\|w(s)\|_{L^\infty}\|\bar{b}(\phi,\nabla\phi)(s)\|_{H^s}ds\\
+&C\int_0^t\|h(\phi,\nabla\phi)(s)\|_{L^\infty}\|g(w)\|_{H^s} +\|h(\phi,\nabla\phi)(s)\|_{H^s}\|g(w)\|_{L^\infty}ds.
\end{align*}
Let us observe that, by assumptions $(H_b)$,$(H_f)$,$(H_g)$,$(H_h)$ and Proposition \ref{Prop_h}, we have
\begin{eqnarray*}
\|g(w)\|_{H^s}&\leq& %C \|g(w)\|_{L^2}+C_{\bar{g}'}\|w\|_{L^\infty}^{s-1}\|D^s w\|_{L^2}\leq 
CG_k \|w\|_{L^2}+C_{\bar{g}'}\|w\|_{L^\infty}^{s-1}\| w\|_{H^s}.\nonumber\\
\|h(\phi,\nabla\phi)\|_{H^s}%&\leq&H_k(\|\phi\|_{L^2}+\|\nabla\phi\|_{L^2})+C_{\bar{b}'}\|(\phi, \nabla\phi)\|_{L^\infty}^{s-1}\sum_{|\alpha|=s}\|D^s (\phi,\nabla\phi)\|_{L^2}\nonumber\\
&\leq&C[H_k(\|\phi\|_{L^2}+\|\nabla\phi\|_{L^2})+ C_{\bar{h}'} K^{s-1}\|\phi\|_{H^{s+1}}].\nonumber\\
\|\bar{b}(\phi,\nabla\phi)\|_{H^s}&\leq& C[B_k(\|\phi\|_{L^2}+\|\nabla\phi\|_{L^2})+ C_{\bar{b}'} K^{s-1}\|\phi\|_{H^{s+1}}].\nonumber\\
\|\bar{f}(u,\phi)\|_{H^s}%&\leq &c\|\bar{f}(u,\phi)\|_{L^2}+\sum_{|\alpha|=s}\|D_x^\alpha \bar{f}(u,\phi)\|_{L^2}\nonumber\\ 
&\leq& C[F_k(\|u\|_{L^2}\|u\|_{L^\infty}+\|\phi\|_{L^2}\|\phi\|_{L^\infty})+C_{\bar{f}'} K^{s-1}(\|\phi\|_{H^{s}}+\|u\|_{H^s})].\nonumber
\end{eqnarray*}

By previous inequalities we get 
\begin{align*}
\|w(t)\|_{H^s}\leq&  c(\|w_0\|_{H^s}+\int_0^t (B_k+H_kG_k)(\|\phi(s)\|_{L^\infty}+\|\nabla\phi(s)\|_{L^\infty})\|w(s)\|_{H^s}ds\\
+&\int_0^t (B_k+H_kG_k)\|w(s)\|_{L^\infty}(\|\phi(s)\|_{L^2}+\|\nabla\phi(s)\|_{L^2})ds\\
+&\int_0^t \|w(s)\|_{L^\infty}(C_{\bar{b}'}+C_{\bar{h}'})K^{s-1}(\|\phi\|_{H^s}+\|\nabla\phi\|_{H^s})ds.
\end{align*}
The last relation can be written as:
\begin{eqnarray}\label{stimaW2}
\|w(t)\|_{H^s}\leq C\left(\|w_0\|_{H^s}+\int_0^t M_k(\|\phi\|_{H^{s+1}}+\|w(s)\|_{H^s})ds\right),
\end{eqnarray}
where the constant $M_k$ depends on $K$ and $C_{b'}$, $C_{h'}$.\\
Let us consider now the solution of the parabolic equation, that thanks to Duhamel's formula we can write as
$$
\phi(x,t)=(e^{-bt}\Gamma^p(t)\ast \phi_0)(x)+\int_0^t e^{-b(t-s)}\Gamma^p(t-s)\ast ( a u(s)+\bar{f}(u,\phi)(s))ds.
$$
Then, we can estimate the $H^{s+1}$-norm of $\phi$ as follows
\begin{eqnarray*}
\|\phi(t)\|_{H^{s+1}} %&\leq& \|e^{-bt}\Gamma^p(t)\ast \phi_0\|_{H^{s+1}}+\int_0^t \|e^{-b(t-s)}\Gamma^p(t-s)\ast a u(s)+\bar{f}(u,\phi)(s) ds\|_{H^{s+1}}\\
&\leq& C\|\phi_0\|_{H^{s+1}}+\int_0^t a\|w(s)\|_{H^s} + (1+(t-s)^{-\frac{1}{2}})\|\bar{f}(u,\phi)(s)\|_{H^s} ds\\
&\leq&C\|\phi_0\|_{H^{s+1}}+D_k\int_0^t (1+(t-s)^{-\frac{1}{2}})\left(\|w(s)\|_{H^s}+\|\phi(s)\|_{H^{s+1}}\right)ds,
\end{eqnarray*}
where the constant $D_k$ depends on $K$ and $C_{f'}$. If we sum the last inequality and \eqref{stimaW2}  we obtain
\begin{eqnarray*}
\|w(t)\|_{H^s}+\|\phi(t)\|_{H^{s+1}}&\leq& C(\|\phi_0\|_{H^{s+1}}+\|w_0\|_{H^s})\\
&+&\int_0^t (1+(t-s)^{-\frac{1}{2}})(D_k+M_k)(\|w(s)\|_{H^s}+\|\phi(s)\|_{H^{s+1}})ds.
\end{eqnarray*}

Applying Gronwall's Lemma we easily deduce
\begin{equation}\label{Stimafinale}
\|w(t)\|_{H^s}+\|\phi(t)\|_{H^{s+1}}\leq \tilde{c}(\|w_0\|_{H^s}+\|\phi_0\|_{H^{s+1}})e^{(D_k+M_k)(t+\sqrt{t})}.
\end{equation}
\end{proof}

%%%%%%%%%%%%%%%%%%%%%%%%%%%%%%%%%ESISTENZA-GLOBALE%%%%%%%%%%%%%%%%%%%%%%%%%%%%%%%%%%%%%%%%%%%%%%%%%%%%%%%%%%%%%%%%%%%%%%%%%%%%%%%%%%%%%%%%%%%%%%%%%%%%%%%%%%%%%%%%%%%%%%%%%%%%%%%%%%%%%%%%%%%%%%%%%%%%%%%%%%%%%%%%%%%%%%%%%%%%%%%%%%%%%
\subsection{Global Existence and Asymptotic Behavior of Smooth Solutions}\label{Globale}
In this section our aim is to prove the boundness of solutions to system \eqref{sistema_H} for every time $t$.
Once that this result will be obtained, we could easily prove the global existence of solutions by Lemma \ref{boundinf} and  Continuation Principle \ref{continuation}. 
The estimates are built up on sharp decay estimates, obtained by Theorem \ref{Theo_BHN} for the Green function of the hyperbolic operator and the known decay of the heat kernel.\\
Let us observe that by this approach, we get simultaneously the boundness of %$L^\infty$-
norm of solutions and also their decay rates.  
Given $\delta>0$, let us define for a given function $g$ the functionals
\begin{equation*}\label{funzM}
M^\delta_g(t)=\sup_{(0,t)}(\max\{1,s^\delta\}\|g(s)\|_{L^2}),
\end{equation*}

\begin{equation*}\label{funzN}
N^\delta_g(t)=\sup_{(0,t)}(\max\{1,s^\delta\}\|g(s)\|_{L^\infty}).
\end{equation*}
Moreover let us denote by $D_x^s$ any space derivative $D_x^\alpha$, such that $|\alpha|=s $.\\
Before starting our proof, let us recall an useful lemma  \cite{BiHaNa}:
\begin{lemma}\label{Lemma5.2}
For any $ \gamma, \delta\geq 0$, $t\geq 2$
$$
\nu:=\min\left\{\gamma, \delta, \gamma +\delta-1 \right\},
$$
it holds
\begin{equation*}
\int_0^t \min\left\{1,(t-s)^{-\gamma}\right\}\min\left\{1,s^{-\delta}\right\}ds\leq C\cdot 
\left\{ \begin{array}{ll}
\min\left\{1,t^{-\nu}\right\} , & \gamma, \delta\neq 1,\\
 \min\left\{1,t^{-\nu}(1+\ln t)\right\},& \gamma\leq 1, \delta=1 \textrm{ or } \, \gamma= 1, \delta\leq1,\\
\min\left\{1,t^{-1}\right\}, & \gamma>1, \delta=1 \textrm{ or }\, \gamma= 1, \delta>1,
\end{array}
\right.
\end{equation*}
\begin{equation*}
\int_0^t \min\left\{1,s^{-\delta}\right\}ds\leq C\cdot 
\left\{ \begin{array}{ll}
1, & \delta> 1,\\
\ln t,&  \delta=1,\\
t^{1-\delta}, & 0\leq \delta<1, 
\end{array}
\right.
\end{equation*}
\begin{equation*}
\int_0^t e^{-c(t-s)}\min\left\{1,s^{-\delta}\right\}ds\leq C \min\left\{1,s^{-\delta}\right\}, \quad \gamma\geq 0.
\end{equation*}
\end{lemma}

\subsubsection{Decay Estimates for the Chemoattractant}
We can collect the estimate referred to the function $\phi$ in the following proposition.

\begin{proposition}\label{phi_estimates}
Let $(u,v,\phi)$ be the solution of system \eqref{sistema_H}-\eqref{dati_cauchy}-\eqref{Ass_Dati_2}, under the assumptions $(H_b)$, $(H_f)$, $(H_g)$, $(H_h)$. Let $K,T>0$ such that for $t\in (0,T)$, $\|u(t)\|_{L^\infty}$, $\|\phi(t)\|_{W^{1,\infty}}\leq K$. Then for $t\in (0,T)$,  
\begin{equation}\label{Prop_Nphix}
\begin{array}{lcl}
N_{D_x^1 \phi}^{\frac{n}{2}}(t)&\leq& C\left(\|D_x\phi_0\|_{L^\infty}+(1+F_kK)N_u^{\frac{n}{2}}(t)+F_kKN_\phi^{\frac{n}{2}}(t)\right),\\\nonumber\\
M_{D_x^{s+1}\phi}^{\tilde{\delta}}(t)&\leq& C\left(\|D_x^{s+1}\phi_0\|_{L^2}+(1+C_k) M_{D_x^s u}^{\tilde{\delta}}(t)+C_k M_{D_x^s\phi}^{\tilde{\delta}}(t)\right),
\end{array}
\end{equation}
where $\tilde{\delta}=\min \left\{\frac{n}{4}+\frac{1}{2}+\frac{s}{2},\frac{n}{2}\right\}$, and the constant $C_k$ depends on $K$ and $C_f'$.
Moreover, if $K$ is sufficiently small, then we have
\begin{equation}\label{Prop_Nphi}
\begin{array}{lcl}
N_\phi^{\frac{n}{2}}(t)&\leq& C\left(\|\phi_0\|_{L^\infty}+(1+F_kK)N_u^{\frac{n}{2}}(t)\right),\\\\
M_\phi^{\frac{n}{4}}(t)&\leq& C\left(\|\phi_0\|_{L^2}+(1+F_kK)M_{u}^{\frac{n}{4}}(t)\right),\\\nonumber\\
M_{D_x^{s}\phi}^{\tilde{\delta}}(t)&\leq& C\left(\|D_x^{s}\phi_0\|_{L^2}+(1+C_k) M_{D_x^s u}^{\tilde{\delta}}(t)\right)\nonumber.
\end{array}
\end{equation}
\end{proposition}

\begin{proof}
Fix $K>0$ large enough and let $T>1$. Take a solution to system \eqref{sistema_H} such that $\|u\|_{L^\infty(\mathbb{R}^n\times(0,T))}$, $\|\phi\|_{W^{1,\infty}(\mathbb{R}^n\times(0,T))}\leq \frac{K}{2}$, this is possible provided that the initial data are suitably small.
Thanks to the Duhamel's formula it is possible to write the function $\phi$ as
\begin{equation}\label{solphi2}
\phi(x,t)=(e^{-bt}\Gamma^p(t)\ast \phi_0)(x)+\int_0^t e^{-b(t-s)}\Gamma^p(t-s)\ast (\alpha u(s)+\bar{f}(u,\phi)(s)) ds.
\end{equation}
Now we proceed in estimating the function in the different norms. Let us start with the $L^\infty$-norm
\paragraph{$L^\infty$-estimate for $\phi$\\}
By the previous equations, we have
\begin{eqnarray*}
\|\phi(t)\|_{L^\infty}&\leq& e^{-bt}\|\Gamma^p(t)\ast\phi_0\|_{L^\infty}+\int_0^t e^{-b(t-s)}\|\Gamma^p(t-s)\ast (\alpha u(s)+\bar{f}(u,\phi)(s))\|_{L^\infty} ds\\
%&\leq& e^{-bt}\|\phi_0\|_{L^\infty}+ \int_0^t e^{-b(t-s)}\|\Gamma^p(t-s)\|_{L^1}\| (\alpha u(s)+\bar{f}(u,\phi)(s))\|_{L^\infty}ds\\
&\leq&  Ce^{-bt}\|\phi_0\|_{L^\infty}+ \int_0^t Ce^{-b(t-s)} ((\alpha+KF_k)\|u(s)\|_{L^\infty}+KF_k\|\phi(s)\|_{L^\infty})ds.
\end{eqnarray*}
Let us multiply by $\min \{1, s^{-\frac{n}{2}}\}\max\{1, s^{\frac{n}{2}}\}=1$, which yields,
\begin{eqnarray*}
\|\phi(t)\|_{L^\infty}&\leq& C \left(e^{-bt}\|\phi_0\|_{L^\infty}+(1+KF_k) N_u^{\frac{n}{2}}(t)\int_0^t e^{-b(t-s)} \min\{1,s^{-\frac{n}{2}}\}ds\right.\\
&+& \left.K F_k N_\phi^{\frac{n}{2}}(t) \int_0^t e^{-b(t-s)}\min\{1,s^{-\frac{n}{2}}\}ds\right).
\end{eqnarray*}
Thanks to Lemma \ref{Lemma5.2}, we easily deduce
\begin{equation}\label{phiLinf2}
\|\phi(t)\|_{L^\infty}\leq C\left(e^{-bt}\|\phi_0\|_{L^\infty}+ (1+F_kL)\min\{1,t^{-\frac{n}{2}}\}N_u^{\frac{n}{2}}(t)+F_kL \min\{1,t^{-\frac{n}{2}}\}N_\phi^{\frac{n}{2}}(t)\right).
\end{equation}
%or if we estimate $\|\cdot\|_{L^\infty}\leq \|\cdot\|_{H^s}$,
%\begin{equation}\label{phiLinf3}
%\|\phi(t)\|_{L^\infty}\leq ce^{-bt}\|\phi_0\|_{L^\infty}+ (1+F_kL)\min\{1,t^{-\delta_1}\}M_u^{\delta_1}(t)+F_kL \min\{1,t^{-\frac{n}{2}}\}M_\phi^{\delta_2}(t).
%\end{equation}

\paragraph{$L^\infty$-estimate for $D_x^1\phi$\\}
Proceeding in a similar way, we get
\begin{eqnarray*}
\|D_x^1\phi(t)\|_{L^\infty}&\leq& \|D_x^1(e^{-bt}\Gamma^p(t)\ast\phi_0)\|_{L^\infty}+
\int_0^t \|D_x^1(e^{-b(t-s)}\Gamma^p(t-s)\ast(\alpha u(s)+\bar{f}(u,\phi)(s))\|_{L^\infty} ds\\
%&\leq& e^{-bt}\|D_x^1(\Gamma^p(t)\ast\phi_0)\|_{L^\infty}+\int_0^t e^{-b(t-s)}\|D_x^1(\Gamma^p(t-s)\ast (\alpha u(s)+\bar{f}(u,\phi)(s))\|_{L^\infty} ds\\
%&\leq& e^{-bt}\|\Gamma^p(t)\|_{L^1}\|D_x^1\phi_0\|_{L^\infty}+
%\int_0^t e^{-b(t-s)}\|D_x^1\Gamma^p(t-s)\|_{L^1}\| \alpha u(s)+\bar{f}(u,\phi)(s)\|_{L^\infty} ds\\
&\leq& Ce^{-bt}\|D_x^1\phi_0\|_{L^\infty}+\int_0^t C e^{-b(t-s)}(t-s)^{-\frac{1}{2}} ((\alpha+KF_k)\|u(s)\|_{L^\infty}+ KF_k\|\phi(s)\|_{L^\infty})ds\\
&\leq& C \left[ e^{-bt}\|D_x^1\phi_0\|_{L^\infty}+(1+KF_k)N_u^{\frac{n}{2}}(t)\int_0^t e^{-b(t-s)}(t-s)^{-\frac{1}{2}}\min\{1,s^{-\frac{n}{2}}\}ds\right.  \\
&+& \left.KF_k N_\phi^{\frac{n}{2}}(t)\int_0^t e^{-b(t-s)}(t-s)^{-\frac{1}{2}}\min\{1,s^{-\frac{n}{2}}\}ds \right].
\end{eqnarray*}

Thus the estimates of the $L^\infty$-norm of the first derivative is given by
\begin{equation}\label{phixLinf}
\|D_x^1\phi(t)\|_{L^\infty}\leq C \left(e^{-bt}\|D_x^1\phi_0\|_{L^\infty}+ (1+KF_k)N_{u}^{\frac{n}{2}}(t)\min\{ 1,|t-1|^{-\frac{n}{2}}\}+KF_k N_{\phi}^{\frac{n}{2}}(t)\min\{ 1,|t-1|^{-\frac{n}{2}}\}\right). 
\end{equation}

From the last inequality and \eqref{phiLinf2} follows that the functionals
$N_\phi^{\frac{n}{2}}$ and $N_{D_x^1 \phi}^{\frac{n}{2}}$, can be estimated as
\begin{eqnarray*}
N_\phi^{\frac{n}{2}}(t)\leq  C\left(\|\phi_0\|_{L^\infty}+(1+F_kK)N_u^{\frac{n}{2}}(t)+F_kKN_\phi^{\frac{n}{2}}(t)\right).
\end{eqnarray*}
\begin{equation}\label{Nphix}
N_{D_x^1\phi}^{\frac{n}{2}}(t)\leq C\left(\|D_x\phi_0\|_{L^\infty}+(1+F_kK)N_u^{\frac{n}{2}}(t)+F_kKN_\phi^{\frac{n}{2}}(t)\right).
\end{equation}

Moreover, if $K$ is sufficiently small, then we have:
\begin{equation}\label{N_phi_small}
N_\phi^{\frac{n}{2}}(t)\leq C\left(\|\phi_0\|_{L^\infty}+(1+F_kK)N_u^{\frac{n}{2}}(t)\right).
\end{equation}

\paragraph{$L^2$-estimate for $\phi$\\}
We estimate now the function $\phi$ and its derivatives in the $L^2$-norm.
Let us start from the $L^2$ estimate for $\phi$.

By the Duhamel's formula \eqref{solphi2}, follows
\begin{eqnarray*}
\|\phi(t)\|_{L^2}&\leq&\|e^{-bt}\Gamma^p(t)\ast\phi_0\|_{L^2}+\int_0^t \|e^{-b(t-s)}\Gamma^p(t-s)\ast (\alpha u(s)+\bar{f}(u,\phi)(s))\|_{L^2} ds\\
%&\leq& e^{-bt}\|\Gamma^p(t)\|_{L^1}\|\phi_0\|_{L^2}+\int_0^t e^{-b(t-s)}(\|\Gamma^p(t-s)\|_{L^1}\|\alpha u(s)+\bar{f}(u,\phi)(s)\|_{L^2}ds\\
&\leq& C(e^{-bt}\|\phi_0\|_{L^2}+\int_0^t e^{-b(t-s)}(\|u(s)\|_{L^2}+F_kK(\|u(s)\|_{L^2}+\|\phi(s)\|_{L^2})) ds\\
&\leq& C(e^{-bt}\|\phi_0\|_{L^2}+(1+F_kK)M_u^{\frac{n}{4}}(t)\int_0^t  e^{-b(t-s)}\min\{1,s^{-\frac{n}{4}}\}ds \\
&+&F_kKM_\phi^{\frac{n}{4}}(t)\int_0^t  e^{-b(t-s)}\min\{1,s^{-\frac{n}{4}}\}ds).
\end{eqnarray*}
Proceeding as done before, by Lemma \ref{Lemma5.2} we obtain
 \begin{eqnarray}\label{phiL2}
&&\|\phi(t)\|_{L^2}\leq C\left(e^{-bt}\|\phi_0\|_{L^2}+ (1+F_kK) M_u^{\frac{n}{4}}(t)\min\{1,t^{-\frac{n}{4}}\}+F_kK M_{\phi}^{\frac{n}{4}}(t)\min\{1,t^{-\frac{n}{4}}\}\right).
\end{eqnarray}
Then, for the related functional, the following estimate yields
$$
M_{\phi}^{\frac{n}{4}}(t)\leq C\left(\|\phi_0\|_{L^2}+(1+F_kK)M_{u}^{\frac{n}{4}}(t)+F_kKM_{\phi}^{\frac{n}{4}}(t)\right).
$$

Also in this case, if $K$ is sufficiently small, then
\begin{eqnarray}\label{M_phi}
M_{\phi}^{\frac{n}{4}}(t)&\leq& C\left(\|\phi_0\|_{L^2}+(1+F_kK)M_{u}^{\frac{n}{4}}(t)\right)\label{M_phi_s}.
\end{eqnarray}  

\paragraph{$L^2$-estimate for $D_x^{s}\phi$\\}
Let us proceed estimating the $L^2-$norm of the $s$-derivative the function $\phi$. By the Duhamel's formula, we obtain
in a similar way 
\begin{eqnarray*}
\|D_x^{s}\phi(t)\|_{L^2}&\leq& e^{-bt}\|D_x^{s}\Gamma^p(t)\ast\phi_0\|_{L^2}+\int_0^t e^{-b(t-s)}\|D_x^{s}\Gamma^p(t-s)\ast (\alpha u(s)+\bar{f}(u,\phi)(s))\|_{L^2} ds\\
%&\leq& Ce^{-bt}\|D_x^{s}\phi_0\|_{L^2}+\int_0^t Ce^{-b(t-s)}(\|D_x^{s}u(s)\|_{L^2}+\|D_x^s\bar{f}(u,\phi)(s)\|_{L^2}) ds\\
&\leq& Ce^{-bt}\|D_x^{s}\phi_0\|_{L^2}+\int_0^t Ce^{-b(t-s)}(\|D_x^{s}u(s)\|_{L^2}+C_{\bar{f}'} K^{s-1}(\|D_x^s\phi(s)\|_{L^2}+\|D_x^su(s)\|_{L^2})).
\end{eqnarray*}
Using Lemma \ref{Lemma5.2} we deduce:
\begin{eqnarray*}
\|D_x^{s}\phi(t)\|_{L^2}&\leq& C\left(e^{-bt}\|D_x^{s}\phi_0\|_{L^2}+ 
(1+2C_{f'}K^{s-1}) M_{D_x^s u}^{\tilde{\delta}}(t)\min\{1,t^{-\tilde{\delta}}\}\right.\\
&+&\left.2C_{f'}K^{s-1} M_{D_x^s\phi}^{\tilde{\delta}}(t)\min\{1,t^{-\tilde{\delta}}\}\right),
\end{eqnarray*}
where $\tilde{\delta}=\min \left\{\frac{n}{4}+\frac{1}{2}+\frac{s}{2},\frac{n}{2}\right\}$. Then, for the related functional we have
$$
M_{D_x^{s}\phi}^{\tilde{\delta}}(t)\leq C\left(\|D_x^{s}\phi_0\|_{L^2}+(1+C_k) M_{D_x^s u}^{\tilde{\delta}}(t)+C_k M_{D_x^s\phi}^{\tilde{\delta}}(t)\right)
$$
and, if $K$ is sufficiently small, then
\begin{equation}\label{M_phiL2}
M_{D_x^{s}\phi}^{\tilde{\delta}}(t)\leq C\left(\|D_x^{s}\phi_0\|_{L^2}+(1+C_k) M_{D_x^s u}^{\tilde{\delta}}(t)\right).
\end{equation}

\paragraph{$L^2$-estimate for $D_x^{s+1}\phi$\\ }
Finally we estimate the $L^2$-norm of the $s+1$-derivative of $\phi$.
As done before
\begin{eqnarray*}
\|D_x^{s+1}\phi(t)\|_{L^2}&\leq& e^{-bt}\|D_x^{s+1}\Gamma^p(t)\ast\phi_0\|_{L^2}+\int_0^t e^{-b(t-s)}\|D_x^{s+1}\Gamma^p(t-s)\ast (\alpha u(s)+\bar{f}(u,\phi)(s))\|_{L^2} ds\\
&\leq& Ce^{-bt}\|D_x^{s+1}\phi_0\|_{L^2}+\int_0^t Ce^{-b(t-s)}\|D_x^1\Gamma^p(t-s)\|_{L^1}\|D_x^s(\alpha u(s)+\bar{f}(u,\phi)(s))\|_{L^2}ds\\
&\leq& Ce^{-bt}\|D_x^{s+1}\phi_0\|_{L^2}+\int_0^t Ce^{-b(t-s)} (t-s)^{-\frac{1}{2}}\|D_x^{s}u(s)\|_{L^2}\\
&+&\int_0^t 2CC_{f'}K^{s-1}e^{-b(t-s)} (t-s)^{-\frac{1}{2}}(\|D_x^{s}\phi(s)\|_{L^2}+\|D_x^{s}u(s)\|_{L^2}) ds.
\end{eqnarray*}

Thanks to Lemma \ref{Lemma5.2} we deduce,
\begin{eqnarray}\label{D_s_phiL2}
\|D_x^{s+1}\phi(t)\|_{L^2}&\leq& C\left(e^{-bt}\|D_x^{s+1}\phi_0\|_{L^2}+ 
(1+K) M_{D_x^s u}^{\tilde{\delta}}(t)\min\{1,|t-1|^{-\frac{n}{4}}\}\right.\\
&+&\left.K M_{D_x^s\phi}^{\tilde{\delta}}(t)\min\{1,|t-1|^{-\tilde{\delta}}\}\right)\nonumber,
\end{eqnarray}
where $\tilde{\delta}=\min \left\{\frac{n}{4}+\frac{1}{2}+\frac{s}{2},\frac{n}{2}\right\}$.
Then, for the functional we get
\begin{equation}\label{M_Ds+phi}
M_{D_x^{s+1}\phi}^{\tilde{\delta}}(t)\leq C\left(\|D_x^{s+1}\phi_0\|_{L^2}+(1+C_k) M_{D_x^s u}^{\tilde{\delta}}(t)+C_k M_{D_x^s\phi}^{\tilde{\delta}}(t)\right).
\end{equation}
\end{proof}

%%%%%%%%%%%%%%%%%%%%%%%%%STIME PER W%%%%%%%%%%%%%%%%%%%%%%%%%%%%%%%%%%%%%%%%%%%%%%%%%%%%%%
\subsubsection{Decay Estimates for the Conservative and Dissipative Variables}
Now we can prove the existence of global solutions to system \eqref{sistema_H} for suitably small initial data.
\begin{theorem}\label{Theo_Glo_GUMA}
Under the assumptions  $(H_b)$, $(H_f)$,$(H_g)$, and $(H_h)$ there exists an $\epsilon_0>0$ such that, if 
$$
\|u_0\|_{H^s}, \|u_0\|_{L^1}, \|v_0\|_{H^s}, \|v_0\|_{L^1},\|\phi_0\|_{H^{s+1}},\|\phi_0\|_{W^{1,\infty}} \leq \epsilon_0,
$$
then there exists a unique global solution to the Cauchy problem \eqref{sistema_H}-\eqref{dati_cauchy}:
\begin{align*}
u \in C([0,\infty), H^s(\mathbb{R}^n)),\; v \in C([0,\infty), H^s(\mathbb{R}^n)),\;  \phi \in C([0,\infty), H^{s+1}(\mathbb{R}^n)), \quad \textrm{for }  s\geq \left[\frac{n}{2}\right]+1.
\end{align*}
Moreover for the solution $(u,v,\phi)$ the following decay rates  are satisfied 
\begin{equation*}
\begin{array}{llll}
 \|u(t)\|_{L^\infty}\sim t^{-\frac{n}{2}},& \|u(t)\|_{L^2}\sim t^{-\frac{n}{4}}, & \|D_x^k u(t)\|_{L^2}\sim t^{-\delta_k}, &\textrm{ for } k=0,\ldots,s; \\ \\
\|v(t)\|_{L^\infty}\sim t^{-\frac{n}{2}}, & \|v(t)\|_{L^2}\sim t^{-\nu_{0}}, & \|D_x^k v(t)\|_{L^2}\sim t^{-\nu_k}, & \textrm{ for } k=0,\ldots,s;  \\ \\
 \|\phi(t)\|_{L^\infty}\sim t^{-\frac{n}{2}}, &  \|D_x^1 \phi(t)\|_{L^\infty}\sim t^{-\frac{n}{2}},, & & \\\\
 \|\phi(t)\|_{L^2}\sim t^{-\frac{n}{4}}, & \|D_x^{k+1}\phi(t)\|_{L^2}\sim t^{-\delta_k}, & \textrm{ for } k=0,\ldots,s; &\\
\end{array}
\end{equation*}
where $\delta_k=\min\left\{\frac{n}{4}+\frac{1}{2}+\frac{1}{2}\left[\frac{k+1}{2}\right],\frac{n}{4}+\delta_r\right\}$, with $r=\left[\frac{k}{2}\right]$, $\nu_{0}=\min\left\{\frac{n}{2},\frac{n}{4}+\frac{1}{2}\right\}$, and\\
$\nu_k=\min\left\{\frac{n}{4}+1+\frac{1}{2}\left[\frac{k+1}{2}\right],\frac{n}{4}+\delta_r\right\}$, with $r=\left[\frac{k}{2}\right]$.\\
\end{theorem}

\begin{rmk}
We have defined the decay rates of the $s$-order derivative as
\begin{equation}\label{delta}
\delta_s=\min \left\{\frac{n}{4}+\frac{1}{2}+\frac{1}{2}\left[\frac{s+1}{2}\right],\frac{n}{4}+\delta_r\right\}, \quad \textrm{ for } s\geq1, 
\end{equation}
where $r=\left[\frac{s}{2}\right]$. Here we write the explicit form for the lower orders.\\
Set $\delta_0=\frac{n}{4}$. Let be $s=1$, then, by the relation \eqref{delta}, we have $\delta_1=\min \left\{\frac{n}{4}+1,\frac{n}{2}\right\}$. \\
If $s=2$, then $\delta_2=\min \left\{\frac{n}{4}+1,\frac{n}{4}+\delta_1\right\}$. When $s=3$, we get $\delta_3=\min \left\{\frac{n}{4}+\frac{3}{2},\frac{n}{4}+\delta_1\right\}$ and so on.\\

\end{rmk}

\begin{proof}
In order to prove our global existence result, we need to estimate the $H^s$ and $L^\infty$-norm of the solution $(u,v)$ to system \eqref{C-D_uv}.

By the Duhamel's formula that solution $w$ to system \eqref{sistema_H} can be written as
\begin{eqnarray}
w(x,t)= (\Gamma^h(t)\ast w_0)(x)+ \int_0^t \Gamma^h(t-s)\ast (\bar{B}(\phi,\nabla\phi)(s)w(s)+H(\phi,\nabla\phi,w)(s))ds.
\end{eqnarray}
where the function $\Gamma^h(\cdot)$ is the Green function of the dissipative hyperbolic system \eqref{sistema_Hyp}.\\ 
Thus for the first component of $w$, the conservative variable $u$, we have:
\begin{eqnarray}\label{comp_u}
u(x,t)=(\Gamma^h_{1}(t)\ast w_0)(x)+ \int_0^t \Gamma^h_{1}(t-s)\ast (\bar{B}(\phi,\nabla\phi)w(s)+H(\phi,\nabla\phi,w)(s))ds,
\end{eqnarray}
where $\Gamma^h_{1}$ is the first row of the $(n+1)\times (n+1)$ Kernel $\Gamma^h$.

Regarding the generic dissipative component $v_j$, for $j=1,\ldots,n$, we have
\begin{eqnarray}\label{comp_v}
v_j(x,t)=(\Gamma^h_{j+1}(t)\ast w_0)(x)+ \int_0^t \Gamma^h_{j+1}(t-s)\ast (\bar{B}(\phi(s),\nabla\phi(s))w+H(\phi,\nabla\phi,w)(s))ds,
\end{eqnarray}
where  $\Gamma^h_{j+1}$ is the $(j+1)$-th row of $\Gamma^h$. 

We take into account the expression of $\Gamma^h$ and its decay rates, presented in Section \ref{MultiGreen}, in order to obtain decay estimates of the conservative and dissipative variables.

\paragraph{$L^2$-estimate for $u$\\}

We will start our analysis by the $L^2$ estimate for the function $u$.
By equation \eqref{comp_u} follows
\begin{equation}\label{uL2}
\begin{array}{l}
\|u(t)\|_{L^2}\leq \|\Gamma^h_{1}(t)\ast w_0\|_{L^2}+\displaystyle\int_0^t \|\Gamma^h_{1}(t-s)\ast (\bar{B}(\phi(s),\nabla\phi(s))w+H(\phi,\nabla\phi,w))\|_{L^2} ds.
\end{array}
\end{equation}
In Section \ref{MultiGreen} we observed that, it is possible to decompose the Green Kernel, then
\begin{equation*}
\|\Gamma^h_{1}(t)\ast w_0\|_{L^2}\leq  \|K_{1,1}(t) u_0\|_{L^2}+\sum_{i=1}^n\|K_{1,i+1}(t) v_0^i\|_{L^2}+\|\mathcal{K}_{1,1}(t) u_0\|_{L^2}+\sum_{i=1}^n\|\mathcal{K}_{1,i+1}(t) v_0^i\|_{L^2}.
\end{equation*}
By Theorem \ref{Theo_BHN} we deduce
\begin{equation} \nonumber
\begin{array}{lll}
\|\mathcal{K}_{1,1}(t) u_0\|_{L^2}\leq Ce^{-ct}\|u_0\|_{L^2} &\quad & \|K_{1,1}(t) u_0\|_{L^2}\leq C\min\{1,t^{-\frac{n}{4}}\}\|u_0\|_{L^1},\\ & \\
\|\mathcal{K}_{1,i+1}(t) v_0^i\|_{L^2}\leq Ce^{-ct}\|v_0^i\|_{L^2},&\quad &\|K_{1,i+1}(t) v_0^i\|_{L^2}\leq C\min\{1,t^{-\frac{n}{4}-\frac{1}{2}}\}\|v_0^i\|_{L^1}.
\end{array}
\end{equation}

Moreover we can decompose the integral term in \eqref{uL2} as 
\begin{align*}
\int_0^t \| & \Gamma^h_{1}(t-s)\ast (\bar{B}(\phi,\nabla\phi)(s)w(s)+H(\phi,\nabla\phi,w)(s))\|_{L^2}ds\\
\leq&\int_0^t \|\mathcal{K}_{1}(t-s) (\bar{B}(\phi,\nabla\phi)(s)w(s)+H(\phi,\nabla\phi,w)(s))\|_{L^2}ds\\
+&\int_0^t \|K_{1}(t-s) (\bar{B}(\phi,\nabla\phi)(s)w(s)+H(\phi,\nabla\phi,w)(s))\|_{L^2}ds.
\end{align*}
Let us start estimating the first integral.
\begin{align*}
\int_0^t \| & \mathcal{K}_{1}(t-s) (\bar{B}(\phi,\nabla\phi)(s)w(s)+H(\phi,\nabla\phi,w)(s))\|_{L^2}\\
\leq & \int_0^t Ce^{-c(t-s)}(\| \bar{B}(\phi,\nabla\phi)(s)w(s)\|_{L^2}+\|H(\phi,\nabla\phi,w)\|_{L^2}) ds\\
\leq & \int_0^tCe^{-c(t-s)}(B_k(\|\nabla\phi(s)\|_{L^\infty}+\|\phi(s)\|_{L^\infty})\|v(s)\|_{L^2}+H_kG_k(\|\nabla\phi(s)\|_{L^\infty}+\|\phi(s)\|_{L^\infty})\|u(s)\|_{L^2} ds.
\end{align*}
Proceeding as done for the estimates of the function $\phi$, we arrive at
\begin{align*}
\int_0^t \|& \mathcal{K}_{1}(t-s) (\bar{B}(\phi,\nabla\phi)(s)w(s)+H(\phi,\nabla\phi,w)(s))\|_{L^2} \\
\leq&CB_k(N_{D_x^1 \phi}^{\frac{n}{2}}(t)M_{v}^{\nu_{0}}(t)+N_{\phi}^{\frac{n}{2}}(t)M_{v}^{\nu_{0}}(t))\int_0^te^{-c(t-s)}\min\{1,s^{-(\frac{n}{2}+\nu_{0})}\}ds\\
+&CH_kG_k(M_{u}^{\frac{n}{4}}(t)N_{\phi}^{\frac{n}{2}}(t)+M_{u}^{\frac{n}{4}}(t)N_{D_x^1 \phi}^{\frac{n}{2}}(t))\int_0^t e^{-c(t-s)}\min\{1,s^{-\frac{3}{4}n}\}ds,
\end{align*}
where $\nu_{0}=\min\left\{\frac{n}{4}+\frac{1}{2},\frac{n}{2}\right\}$.
Then thanks to Lemma \ref{Lemma5.2} we deduce
\begin{align*}
\int_0^t \| & \mathcal{K}_{1}(t-s) (\bar{B}(\phi,\nabla\phi)(s)w(s)+H(\phi,\nabla\phi,w)(s))\|_{L^2}\leq\\
+&CB_k\min\{1,t^{-(\frac{n}{2}+\nu_{0})}\}(N_{D_x^1 \phi}^{\frac{n}{2}}(t)M_{v}^{\nu_{0}}(t)+N_{\phi}^{\frac{n}{2}}(t)M_{v}^{\nu_{0}}(t))\\
+&CH_kG_k \min\{1,t^{-\frac{3}{4}n}\}(M_{u}^{\frac{n}{4}}(t)N_{\phi}^{\frac{n}{2}}(t)+ M_{u}^{\frac{n}{4}}(t)N_{D_x^1 \phi}^{\frac{n}{2}}(t)).
%\min\{1,t^{-(\frac{n}{4}+\frac{n}{2})}\}+\min\{1,s^{-(\delta_1+\frac{n}{2})}\}M_{u}^{\delta_1}N_{u}^{\frac{n}{2}}).
\end{align*}

To complete our estimate we need to study the contribution of the hyperbolic Green function diffusive part.
\begin{align*}
\int_0^t\|&K_{1}(t-s)(\bar{B}(\phi,\nabla\phi)(s)w(s)+H(\phi,\nabla\phi,w)(s))\|_{L^2}ds\\
%=&\int_0^t\sum_{i=1}^n\|K_{1,i+1}(t-s)\ast (\bar{B}(\phi,\nabla\phi)(s)v_i(s)+H_i(\phi,\nabla\phi,u)(s))\|_{L^2}ds\\
\leq& \int_0^t C\min\{1,(t-s)^{-\frac{n}{4}-\frac{1}{2}}\}(\|\bar{b}(\phi,\nabla\phi)(s)v(s)\|_{L^1}+\|h(\phi,\nabla\phi)g(u)(s)\|_{L^1})ds\\
%\leq& \int_0^tC \min\{1,(t-s)^{-\frac{n}{4}-\frac{1}{2}}\}(\|\bar{b}(\phi,\nabla\phi)(s)\|_{L^2}\|v(s)\|_{L^2}+\|h(\phi,\nabla\phi)(s)\|_{L^2}\|u(s)\|_{L^2})ds\\
\leq& \int_0^tC\min\{1,(t-s)^{-\frac{n}{4}-\frac{1}{2}}\}(B_k\|\phi(s)\|_{L^2}\|v(s)\|_{L^2}+B_k\|\nabla\phi(s)\|_{L^2}\|v(s)\|_{L^2}ds\\
+&\int_0^tC\min\{1,(t-s)^{-\frac{n}{4}-\frac{1}{2}}\}(H_kG_k\|\phi(s)\|_{L^2}\|u(s)\|_{L^2}+H_KG_k\|\nabla\phi(s)\|_{L^2}\|u(s)\|_{L^2})ds.
\end{align*}
Introducing the functionals $M^\delta$, we arrive at
\begin{align*}
\int_0^t\| & K_{1}(t-s) (\bar{B}(\phi,\nabla\phi)(s)w(s)+H(\phi,\nabla\phi,w)(s))\|_{L^2}ds\\
\leq &  B_K(M_{\phi}^{\frac{n}{4}}(t)M_{v}^{\nu_{0}}(t)+M_{D_x^1 \phi}^{\frac{n}{4}}(t)M_{v}^{\nu_{0}}(t))\int_0^t\min\{1,(t-s)^{-\frac{n}{4}-\frac{1}{2}}\}\min\{1,s^{-(\frac{n}{4}+\nu_{0})}\}ds\\
+& H_kG_k (M_{\phi}^{\frac{n}{4}}(t)M_{u}^{\frac{n}{4}}(t)+M_{D_x^1 \phi}^{\frac{n}{4}}(t)M_{u}^{\frac{n}{4}}(t))\int_0^t\min\{1,(t-s)^{-\frac{n}{4}-\frac{1}{2}}\}\min\{1,s^{-\frac{n}{2}}\}ds,
\end{align*}
and by Lemma \ref{Lemma5.2} we deduce
\begin{align}\label{stimaKs}
\int_0^t \|&K_{1}(t-s) (\bar{B}(\phi,\nabla\phi)w(s)+H(\phi,\nabla\phi,w)(s))\|_{L^2}\nonumber\\
\leq & \min\{1,t^{-\nu_{0}}\}(B_KM_{\phi}^{\frac{n}{4}}(t)M_{v}^{\nu_{0}}(t)+B_KM_{D_x^1 \phi}^{\frac{n}{4}}(t)M_{v}^{\nu_{0}}(t))\\
+& \min\{1,t^{-\theta}\}(H_kG_kM_{u}^{\frac{n}{4}}(t)M_{\phi}^{\frac{n}{4}}(t)+H_kG_kM_{u}^{\frac{n}{4}}(t)M_{D_x^1 \phi}^{\frac{n}{4}}(t))\nonumber.
\end{align}
where $\theta=\frac{1}{4}$ if $n=1$ otherwise $\theta=\frac{n}{4}+\frac{1}{2}$.

Then we obtain the $L^2$-norm of the function $u$ summing the previous inequalities.
\begin{eqnarray}\label{UL2a}
\|u(t)\|_{L^2}&\leq& C\left[e^{-ct}(\|u_0\|_{L^2}+\sum_{i=1}^n\|v_0^i\|_{L^2})+\min\{1,t^{-\frac{n}{4}}\}\|u_0\|_{L^1}+\min\{1,t^{-\frac{n}{4}-\frac{1}{2}}\}\sum_{i=1}^n\|v^i_0\|_{L^1}\right]\nonumber\\
&+&C_k\left[\min\{1,t^{-(\frac{n}{2}+\nu_{0})}\}(N_{D_x^1 \phi}^{\frac{n}{2}}(t)M_{v}^{\nu_{0}}(t)+N_{\phi}^{\frac{n}{2}}(t)M_{v}^{\nu_{0}}(t))\right.\nonumber\\
&+& \min\{1,t^{-\frac{3}{4}n}\}(M_{u}^{\frac{n}{4}}(t)N_{\phi}^{\frac{n}{2}}(t)+ M_{u}^{\frac{n}{4}}(t)N_{D_x^1 \phi}^{\frac{n}{2}}(t))\\
&+&\min\{1,t^{-\nu_{0}}\}(M_{\phi}^{\frac{n}{4}}(t)M_{v}^{\nu_{0}}(t)+M_{D_x^1 \phi}^{\frac{n}{4}}(t)M_{v}^{\nu_{0}}(t))\nonumber\\
&+&\left.\min\{1,t^{-\theta}\}(M_{u}^{\frac{n}{4}}(t)M_{\phi}^{\frac{n}{4}}(t)+M_{u}^{\frac{n}{4}}(t)M_{D_x^1 \phi}^{\frac{n}{4}}(t))\right]\nonumber,
\end{eqnarray}
where the constant $C_k$ depends on $K$.

%%%%%%%%%%%%%%%%%%%%%%%%%%%%%%%%%%%%%%%%%%%%%%%%%%%%%%%%%%%%%%%%%%%%%%%%%%%%%%%%%%%%%%%%%%%%%%%%
%%%%%%%%%%%%%%%%%%%%%%%%%%%%%%%%%%%%%%%%%%%%%%%%%%%%%%%%%%%%%%%%%%%%%%%%%%%%%%%%%%%%%%%%%%%%%%%%%%%%%%%%%%%%%%%%%%%%%%%%%%%%%%%%%%%%%%%%%%%%%%%%%%%%%%%%%%%%%%%%%%%%%%%%%%%%%%

\paragraph{$L^2$-estimate for $D_x^s u$\\}
The next step is the estimate of  $s$-order derivative of function $u$.  From the Duhamel's formula, it follows that
\begin{equation}\label{DuL2}
\begin{array}{l}
\|D_x^su(t)\|_{L^2}\leq \|D_x^s\Gamma^h_{1}(t)\ast w_0\|_{L^2}+\displaystyle\int_0^t \|D_x^s\Gamma^h_{1}(t-s)\ast (\bar{B}(\phi,\nabla\phi)(s)w(s)+H(\phi,\nabla\phi,w)(s))\|_{L^2} ds.
\end{array}
\end{equation}
Decomposing the Green Kernel, the first term in the previous inequality can be estimated as
\begin{equation*}
\begin{array}{lll}
\|D_x^s \mathcal{K}_{1,1}(t) u_0\|_{L^2}\leq Ce^{-ct}\|D_x^su_0\|_{L^2}, &\quad & \|D_x^sK_{1,1}(t) u_0\|_{L^2}\leq C\min\{1,t^{-\frac{n}{4}-\frac{s}{2}}\}\|u_0\|_{L^1},\\\\
\|D_x^s\mathcal{K}_{1,i+1}(t) v_0^i\|_{L^2}\leq Ce^{-ct}\|D_x^s v_0^i\|_{L^2}, & \quad &\|D_x^sK_{1,i+1}(t) v_0^i\|_{L^2}\leq C\min\{1,t^{-\frac{n}{4}-\frac{1}{2}-\frac{s}{2}}\}\|v_0^i\|_{L^1}.
\end{array}
\end{equation*}
Moreover we can decompose the integral term in
 \eqref{DuL2} as
\begin{align*}
\int_0^t \| &D_x^s\Gamma^h_{1}(t-s)\ast (\bar{B}(\phi,\nabla\phi)w(s)+H(\phi,\nabla\phi,w)(s))\|_{L^2}ds\\
\leq & \int_0^t \|D_x^s \mathcal{K}_{1}(t-s) (\bar{B}(\phi,\nabla\phi)w(s)+H(\phi,\nabla\phi,w)(s))\|_{L^2}ds\\
+&\int_0^t \|D_x^sK_{1}(t-s)(\bar{B}(\phi,\nabla\phi)w(s)+H(\phi,\nabla\phi,w))\|_{L^2}ds.
\end{align*}
Let us start estimating the first integral.  
\begin{align*}
\int_0^t \| & D_x^s \mathcal{K}_{1}(t-s) (\bar{B}(\phi,\nabla\phi)w(s)+H(\phi,\nabla\phi,w)(s))\|_{L^2}\\
\leq & \int_0^t Ce^{-c(t-s)}\| D_x^s\bar{B}(\phi,\nabla\phi)w(s)\|_{L^2}+\|D_x^s H(\phi,\nabla\phi,w)(s)\|_{L^2}ds\\
\leq & \int_0^t Ce^{-c(t-s)}(\|D_x^s\bar{b}(\phi,\nabla\phi)\|_{L^2}\|v(s)\|_{L^\infty}+\|\bar{b}(\phi,\nabla\phi)\|_{L^2}\|D_x^s v(s)\|_{L^\infty})ds\\
+&\int_0^t Ce^{-c(t-s)}(\|D_x^sh(\phi,\nabla\phi)\|_{L^2}\|g(u)(s)\|_{L^\infty}+\|\bar{b}(\phi,\nabla\phi)\|_{L^\infty})\|D_x^s g(u)(s)\|_{L^\infty})ds\\
\leq& \int_0^t e^{-c(t-s)}2C_{b'}K^{s-1}(\|D_x^{s+1}\phi(s)\|_{L^2}+\|D_x^s \phi(s)\|_{L^2})\|v(s)\|_{L^\infty}ds
\\
+&\int_0^t e^{-c(t-s)}B_k(\|\nabla\phi(s)\|_{L^\infty}+\|\phi(s)\|_{L^\infty})\|D_x^s v(s)\|_{L^2}ds\\
+&\int_0^t e^{-c(t-s)}2C_{h'}K^{s-1}G_k(\|D_x^{s+1}\phi(s)\|_{L^2}+\|D_x^s \phi(s)\|_{L^2})\|u(s)\|_{L^\infty}ds
\\
+&\int_0^t e^{-c(t-s)}H_kG_k(\|\nabla\phi(s)\|_{L^\infty}+\|\phi(s)\|_{L^\infty})\|D_x^s u(s)\|_{L^2}ds.
\end{align*}

Then, by Lemma \ref{Lemma5.2} we get
\begin{align*}
\int_0^t \| & D_x^s \mathcal{K}_{1}(t-s) (\bar{B}(\phi,\nabla\phi)w(s)+H(\phi,\nabla\phi,w)(s))\|_{L^2}\\
\leq& 2C_{b'}K^{s-1}\min\{1,t^{-(\tilde{\delta}+\frac{n}{2})}\}(M_{D_x^s \phi}^{\tilde{\delta}}(t)N_{v}^{\frac{n}{2}}(t)+M_{D_x^{s+1} \phi}^{\tilde{\delta}}(t)N_{v}^{\frac{n}{2}}(t))\\
+&B_k\min\{1,t^{-(\frac{n}{2}+\tilde{\nu})}\}(N_{D_x^1 \phi}^{\frac{n}{2}}(t)M_{D_x^s v}^{\tilde{\nu}}(t)+N_{\phi}^{\frac{n}{2}}(t)M_{D_x^s v}^{\tilde{\nu}}(t))\\
+&G_kC_{h'}K^{s-2}\min\{1,t^{-(\tilde{\delta}+\frac{n}{2})}\}(M_{D_x^s\phi}^{\tilde{\delta}}(t)N_{u}^{\frac{n}{2}}(t)
+M_{D_x^s\phi}^{\tilde{\delta}}(t)N_{u}^{\frac{n}{2}}(t))\\
+&H_kG_k\min\{1,t^{-(\tilde{\delta}+\frac{n}{2})}\}(M_{D_x^s u}^{\tilde{\delta}}(t)N_{\phi}^{\frac{n}{2}}(t)+M_{D_x^s u}^{\tilde{\delta}}(t)N_{D_x^1 \phi}^{\frac{n}{2}}(t)),
\end{align*}
where $\tilde{\delta}=\min\{\frac{n}{4}+\frac{1}{2}+\frac{s}{2},\frac{n}{2}\}$ and $\tilde{\nu}=\min\{\frac{n}{4}+1+\frac{s}{2},\frac{n}{2}\}$.\\
In order to complete our estimate, we need to study the contribution of the hyperbolic Green function diffusive part. Proceeding as before
\begin{align*}
\int_0^t\| & D_x^s K_{1}(t-s) (\bar{B}(\phi,\nabla\phi)(s)w(s)+H(\phi,\nabla\phi,w)(s))\|_{L^2}ds\\
%\leq& \int_0^t\|\sum_{i=1}^n D_x^s K_{1,i+1}(t-s)\ast (\bar{B}(\phi,\nabla\phi)(s)v_i(s)+h_i(\phi,\nabla\phi,u)(s))\|_{L^2}ds\\
\leq &\int_0^t C\min\{1,(t-s)^{-\frac{n}{4}-\frac{1}{2}-\frac{s}{2}}\}\|\bar{B}(\phi,\nabla\phi)(s)v(s)\|_{L^1}+\|h(\phi,\nabla\phi)g(u)(s)\|_{L^1}ds\\
\leq &\int_0^tC\min\{1,(t-s)^{-\frac{n}{4}-\frac{1}{2}-\frac{s}{2}}\}B_K(\|\phi(s)\|_{L^2}\|v(s)\|_{L^2}+\|\nabla\phi(s)\|_{L^2}\|v(s)\|_{L^2})ds\\
+&\int_0^tC\min\{1,(t-s)^{-\frac{n}{4}-\frac{1}{2}-\frac{s}{2}}\} G_kH_k(\|\phi(s)\|_{L^2}\|u(s)\|_{L^2}+\|\nabla\phi(s)\|_{L^2}\|u(s)\|_{L^2})ds.
\end{align*}
Moreover
\begin{align*}
\int_0^t\|&D_x^s K_{1}(t-s) (\bar{B}(\phi,\nabla\phi)(s)w(s)+H(\phi,\nabla\phi,w)(s))\|_{L^2}ds\\
\leq& CB_K(M_{\phi}^{\frac{n}{4}}(t)M_{v}^{\nu_{0}}(t)+M_{D_x^1 \phi}^{\frac{n}{4}}(t)M_{v}^{\nu_{0}}(t))\int_0^t\min\{1,(t-s)^{-\frac{n}{4}-\frac{1}{2}-\frac{s}{2}}\}\min\{1,s^{-(\frac{n}{4}+\nu_{0})}\}ds\\
+&C H_kG_k(M_{\phi}^{\frac{n}{4}}(t)M_{u}^{\frac{n}{4}}(t)+M_{D_x^1 \phi}^{\frac{n}{4}} M_{u}^{\frac{n}{4}}(t))\int_0^t\min\{1,(t-s)^{-\frac{n}{4}-\frac{1}{2}-\frac{s}{2}}\}\min\{1,s^{-\frac{n}{2}}\}ds.
\end{align*}
Then, by Lemma \ref{Lemma5.2} we get
\begin{align*}
\int_0^t \|&D_x^1 K_{1}(t-s)(\bar{B}(\phi(s),\nabla\phi(s))w+H(\phi,\nabla\phi ,w)(s))\|_{L^2}\\
\leq& C_k\min\{1,t^{-\tilde{\delta}}\}\left( M_{\phi}^{\frac{n}{4}}(t)M_{v}^{\nu_{0}}(t)+M_{D_x^1 \phi}^{\frac{n}{4}}(t)M_{v}^{\nu_{0}}(t)+M_{D_x^1 \phi}^{\frac{n}{4}}(t)M_{u}^{\frac{n}{4}}(t)+M_{\phi}^{\frac{n}{4}}(t)M_{u}^{\frac{n}{4}}(t)\right),
\end{align*}

where $\tilde{\delta}=\min\left\{\frac{n}{4}+\frac{1}{2}+\frac{s}{2},\frac{n}{2}\right\}$, and $\tilde{\nu}=\min\left\{\frac{n}{4}+1+\frac{s}{2},\frac{n}{2}\right\}$.

Finally the $L^2$- norm of the $s$-derivative of function $u$, can be estimated as follows:
\begin{eqnarray}\label{DUL2a}
\|D_x^s u(t)\|_{L^2}&\leq& C\left[e^{-ct}(\|D_x^su_0\|_{L^2}+\sum_{i=1}^n\|D_x^sv_0^i\|_{L^2})+\min\{1,t^{-\frac{n}{4}-\frac{1}{2}-\frac{s}{2}}\}\sum_{i=1}^n\|v_0^i\|_{L^1}\right.
\nonumber\\
&+&\left.\min\{1,t^{-\frac{n}{4}-\frac{s}{2}}\}\|u_0\|_{L^1}\right]+C_k\left[ \min\{1,t^{-(\tilde{\delta}+\frac{n}{2})}\}(M_{D_x^s \phi}^{\tilde{\delta}}(t)N_{v}^{\frac{n}{2}}(t)+M_{D_x^{s+1} \phi}^{\tilde{\delta}}(t)N_{v}^{\frac{n}{2}}(t))\right.\nonumber\\
&+&\min\{1,t^{-(\tilde{\nu}+\frac{n}{2})}\}(N_{D_x^1 \phi}^{\frac{n}{2}}(t)M_{D_x^s v}^{\tilde{\nu}}(t)+ N_{\phi}^{\frac{n}{2}}(t)M_{D_x^s v}^{\tilde{\nu}}(t))\\
&+&\min\{1,t^{-(\tilde{\delta}+\frac{n}{2})}\}(M_{D_x^s\phi}^{\tilde{\delta}}(t)N_{u}^{\frac{n}{2}}(t)
+M_{D_x^s\phi}^{\tilde{\delta}}(t)N_{u}^{\frac{n}{2}}(t))+M_{D_x^s u}^{\tilde{\delta}}(t)N_{\phi}^{\frac{n}{2}}(t)+M_{D_x^s u}^{\tilde{\delta}}(t)N_{D_x^1 \phi}^{\frac{n}{2}}(t)\nonumber\\
&+&\left.\min\{1,t^{-\tilde{\delta}}\}( M_{\phi}^{\frac{n}{4}}(t)M_{ v}^{\nu_{0}}(t)+M_{D_x^1 \phi}^{\frac{n}{4}}(t)M_{ v}^{\nu_{0}}(t)+M_{D_x^1 \phi}^{\frac{n}{4}}(t)M_{ u}^{\frac{n}{4}}(t)+M_{\phi}^{\frac{n}{4}}(t)M_{u}^{\frac{n}{4}}(t)) \right]\nonumber.
\end{eqnarray}

%%%%%%%%%%%%%%%%%%%%%%%%%%%%%%%%%%%%%%%%%%%%%%%%%%%%%%%%%%%%%%%%%%%%%%%%%%%%%%%%%%%%%%%%%%%%%%%%%%%%%%%%%%%%%%%%%%%%%%%%%%%%%%%%%%%%%%%%%%%%%%%%%%%%%%%%%%%%%%%%%%%%%%%%%%%%%%%%%%%%%%%%%%%%%%L^\infty estimate for the hyperbolic variables

\paragraph{$L^\infty$- estimate for $u$\\}
Let us focus now on the $L^\infty$-norm of the function $u$.
\begin{equation}\label{uLinf}
\begin{array}{l}
\|u(t)\|_{L^\infty}\leq \|\Gamma^h_{1}(t)\ast w_0\|_{L^\infty}+\displaystyle\int_0^t \|\Gamma^h_{1}(t-s)\ast (\bar{B}(\phi,\nabla\phi)(s)w(s)+H(\phi,\nabla\phi,w)(s))\|_{L^\infty} ds.
\end{array}
\end{equation}
By the decomposition of the Green Kernel, we can estimate the first term in the previous inequality as
\begin{equation}
\begin{array}{lll}
\|\mathcal{K}_{11}(t) u_0\|_{L^\infty}\leq\|\mathcal{K}_{1,1}(t)u_0\|_{H^s}\leq Ce^{-ct}\|u_0\|_{H^s},&\, &\|K_{1,1}(t) u_0\|_{L^\infty}\leq C\min\{1,t^{-\frac{n}{2}}\}\|u_0\|_{L^1},\\\nonumber\\
\|\mathcal{K}_{1,i+1}(t)v_0^i\|_{L^2}\leq Ce^{-ct}\|v_0^i\|_{H^s},& \,& \|K_{1,i+1}(t) v_0^i\|_{L^2}\leq C\min\{1,t^{-\frac{n}{2}-\frac{1}{2}}\}\|v_0^i\|_{L^1}.
\end{array}
\end{equation}
Let us decompose the integral term.
\begin{align*}
\int_0^t \|&\Gamma^h_{1}(t-s)\ast (\bar{B}(\phi,\nabla\phi)(s)w(s)+H(\phi,\nabla\phi,w)(s))\|_{L^\infty}ds\\
\leq& \int_0^t \|\mathcal{K}_{1}(t-s) (\bar{B}(\phi,\nabla\phi)(s)w(s)+H(\phi,\nabla\phi,w)(s))\|_{L^\infty}ds\\
+&\int_0^t \|K_{1}(t-s) (\bar{B}(\phi,\nabla\phi)(s)w(s)+H(\phi,\nabla\phi,w)(s))\|_{L^\infty}ds.
\end{align*}
We can estimate the first integral as 
\begin{align*}
\int_0^t \|&  \mathcal{K}_{1}(t-s)(\bar{B}(\phi,\nabla\phi)(s)w(s)+H(\phi,\nabla\phi,w)(s))\|_{L^\infty}\\ 
\leq & C\int_0^t\|\mathcal{K}_{1}(t-s) (\bar{B}(\phi,\nabla\phi)(s)w(s)+H(\phi,\nabla\phi,w)(s))\|_{L^2}ds\\
+& C\int_0^t \sum_{|\alpha|=s}\|D_x^s \mathcal{K}_{1}(t-s) (\bar{B}(\phi,\nabla\phi)(s)w(s)+H(\phi,\nabla\phi,w)(s))\|_{L^2}ds.
\end{align*}

Then, thanks to the $L^2$-estimate calculated previously, we easily obtain 
\begin{align*}
\int_0^t \|&\mathcal{K}_{1}(t-s) (\bar{B}(\phi,\nabla\phi)(s)w(s)+H(\phi,\nabla\phi,w)(s))\|_{L^\infty}\\
\leq&
CB_k\min\{1,t^{-(\frac{n}{2}+\nu_0)}\}(N_{D_x^1 \phi}^{\frac{n}{2}}(t)M_{v}^{\nu_0}(t)+N_{\phi}^{\frac{n}{2}}(t)M_{v}^{\nu_0}(t))\\
+&CH_kG_k \min\{1,t^{-\frac{3}{4}n}\}(M_{u}^{\frac{n}{4}}(t)N_{\phi}^{\frac{n}{2}}(t)+ M_{u}^{\frac{n}{4}}(t)N_{D_x^1 \phi}^{\frac{n}{2}}(t))\\
+& 2C_{b'}K^{s-1}\min\{1,t^{-(\tilde{\delta}+\frac{n}{2})}\}(M_{D_x^s \phi}^{\tilde{\delta}}(t)N_{v}^{\frac{n}{2}}(t)+M_{D_x^{s+1} \phi}^{\tilde{\delta}}(t)N_{v}^{\frac{n}{2}}(t))\\
+&B_k\min\{1,t^{-(\frac{n}{2}+\tilde{\nu})}\}(N_{D_x^1 \phi}^{\frac{n}{2}}(t)M_{D_x^s v}^{\tilde{\nu}}(t)+N_{\phi}^{\frac{n}{2}}(t)M_{D_x^s v}^{\tilde{\nu}}(t))\\
+&G_kC_{h'}K^{s-2}\min\{1,t^{-(\tilde{\delta}+\frac{n}{2})}\}(M_{D_x^s\phi}^{\tilde{\delta}}(t)N_{u}^{\frac{n}{2}}(t)
+M_{D_x^s\phi}^{\tilde{\delta}}(t)N_{u}^{\frac{n}{2}}(t))\\
+&H_kG_k\min\{1,t^{-(\tilde{\delta}+\frac{n}{2})}\}(M_{D_x^s u}^{\tilde{\delta}}(t)N_{\phi}^{\frac{n}{2}}(t)+M_{D_x^s u}^{\tilde{\delta}}(t)N_{D_x^1 \phi}^{\frac{n}{2}}(t)).
\end{align*}

In order to complete our study on the $L^\infty$- norm of function $u$, we estimate the contribution of the hyperbolic Green function diffusive part.
\begin{align*}
\int_0^t\|&K_{1}(t-s) (\bar{B}(\phi,\nabla\phi)(s)w(s)+H(\phi,\nabla\phi,w)(s)(s))\|_{L^\infty}ds\\
\leq & \int_0^t \min\{1,(t-s)^{-\frac{n}{2}-\frac{1}{2}}\}\|\bar{B}(\phi,\nabla\phi)(s)v(s)\|_{L^1}+\|\bar{h}(\phi,\nabla\phi)g(u)(s)\|_{L^1}ds\\
\leq &\int_0^t\min\{1,(t-s)^{-\frac{n}{2}-\frac{1}{2}}\}B_k(\|\phi(s)\|_{L^2}\|v(s)\|_{L^2}+\|\nabla\phi(s)\|_{L^2}\|v(s)\|_{L^2})ds\\
+&\int_0^t\min\{1,(t-s)^{-\frac{n}{2}-\frac{1}{2}}\}H_kG_k(\|\phi(s)\|_{L^2}\|u(s)\|_{L^2}+\|\nabla\phi(s)\|_{L^2}\|u(s)\|_{L^2})ds\\
\leq & B_K(M_{\phi}^{\frac{n}{4}}(t)M_{v}^{\nu_{0}}(t)+M_{D_x^1 \phi}^{\frac{n}{4}}(t)M_{v}^{\nu_{0}}(t)) \int_0^t\min\{1,(t-s)^{-\frac{n}{2}-\frac{1}{2}}\}\min\{1,s^{-(\frac{n}{4}+\nu_{0})}\}ds\\
+&H_kG_k(M_{\phi}^{\frac{n}{4}}(t)M_{u}^{\frac{n}{4}}(t)+M_{D_x^1 \phi}^{\frac{n}{4}}(t)M_{u}^{\frac{n}{4}}(t))\int_0^t\min\{1,(t-s)^{-\frac{n}{2}-\frac{1}{2}}\}\min\{1,s^{-\frac{n}{2}}\}ds.
\end{align*}
Thanks to Lemma \ref{Lemma5.2} we deduce
\begin{align*}
\int_0^t \|& K_{1}(t-s) (\bar{B}(\phi,\nabla\phi)w(s)+H(\phi,\nabla\phi,w)(s))\|_{L^\infty} \\
\leq & C_k\min\{1,t^{-\frac{n}{2}}\}\left(M_{\phi}^{\frac{n}{4}}(t)M_{v}^{\nu_{0}}(t)+M_{D_x^1 \phi}^{\frac{n}{4}}(t)M_{v}^{\nu_{0}}(t)+M_{\phi}^{\frac{n}{4}}(t)M_{u}^{\frac{n}{4}}(t)+M_{D_x^1 \phi}^{\frac{n}{4}}(t)M_{u}^{\frac{n}{4}}(t)\right).
\end{align*}

We can collect the previous estimates in the following inequality 
\begin{eqnarray}\label{ULinfa}
\|u(t)\|_{L^\infty}&\leq& C\left[ e^{-ct}(\|u_0\|_{H^s}+\sum_{i=1}^n\|v_0^i\|_{H^s})+\min\{1,t^{-\frac{n}{2}}\}\|u^0\|_{L^1}+\min\{1,t^{-\frac{n}{2}-\frac{1}{2}}\}\sum_{i=1}^n\|v_0^i\|_{L^1}\right]\nonumber\\
&+&C_k\left[\min\{1,t^{-(\frac{n}{2}+\nu_{0})}\}(N_{D_x^1 \phi}^{\frac{n}{2}}(t)M_{v}^{\nu_{0}}(t)+N_{\phi}^{\frac{n}{2}}(t)M_{v}^{\nu_{0}}(t))\right.\nonumber\\
&+&\min\{1,t^{-\frac{3}{4}n}\}(M_{u}^{\frac{n}{4}}(t)N_{\phi}^{\frac{n}{2}}(t)+ M_{u}^{\frac{n}{4}}(t)N_{D_x^1 \phi}^{\frac{n}{2}}(t))\nonumber\\
&+& \min\{1,t^{-(\tilde{\delta}+\frac{n}{2})}\}(M_{D_x^s \phi}^{\tilde{\delta}}(t)N_{v}^{\frac{n}{2}}(t)+M_{D_x^{s+1} \phi}^{\tilde{\delta}}(t)N_{v}^{\frac{n}{2}}(t))\\
&+&\min\{1,t^{-(\frac{n}{2}+\tilde{\nu})}\}(N_{D_x^1 \phi}^{\frac{n}{2}}(t)M_{D_x^s v}^{\tilde{\nu}}(t)+N_{\phi}^{\frac{n}{2}}(t)M_{D_x^s v}^{\tilde{\nu}}(t))\nonumber\\
&+&\min\{1,t^{-(\tilde{\delta}+\frac{n}{2})}\}(M_{D_x^s\phi}^{\tilde{\delta}}(t)N_{u}^{\frac{n}{2}}(t)
+M_{D_x^s\phi}^{\tilde{\delta}}(t)N_{u}^{\frac{n}{2}}(t))+M_{D_x^s u}^{\tilde{\delta}}(t)N_{\phi}^{\frac{n}{2}}(t)+M_{D_x^s u}^{\tilde{\delta}}(t)N_{D_x^1 \phi}^{\frac{n}{2}}(t))\nonumber\\
&+&\left.\min\{1,t^{-\frac{n}{2}}\}(M_{\phi}^{\frac{n}{4}}(t)M_{v}^{\nu_{0}}(t)+M_{D_x^1 \phi}^{\frac{n}{4}}(t)M_{v}^{\nu_{0}}(t))+M_{\phi}^{\frac{n}{4}}(t)M_{u}^{\frac{n}{4}}(t)+M_{D_x^1 \phi}^{\frac{n}{4}}(t)M_{u}^{\frac{n}{4}}(t))\right],\nonumber
\end{eqnarray}

where the constant $C_k$ depends on $K$.

%%%%%%%%%%%%%%%%%%%%%%%%%%%%%%%%%%%%%%%%%%%%%%%%%%%%%%%%%%%%%%%%%%%%%%%%%%%%%%%%%%%%%%%%%%%%%%%%%%%%%%%%%%%%%%%%%%%%%%%%%%%%%%%%%%%%%%%%%%%%%%%%%%%%%%%%%%%%%%%%%%%%%%%%%%%%%%%%%%%%%%%%%%%%%%%DISSSIPATIVE VARIABLES

In order to complete our proof we need to estimate, by the same technique, the dissipative variable $v$.

\paragraph{$L^2$-estimate for v\\}
Let us start with the $L^2$-norm of a generic component $v_j$, with $j=1,\ldots,n$.

By the Duhamel's formula \eqref{comp_v} we get
\begin{equation}\label{vL2}
\begin{array}{l}
\|v_j(t)\|_{L^2}\leq \|\Gamma^h_{j+1}(t)\ast w_0\|_{L^2}+ \displaystyle\int_0^t \|\Gamma^h_{j+1}(t-s)\ast (\bar{B}(\phi,\nabla\phi)(s)w(s)+H(\phi,\nabla\phi,w)(s))\|_{L^2}ds.
\end{array}
\end{equation}
Then by the decomposition of the Green kernel and by Theorem \ref{Theo_BHN} we get the following estimates
\begin{equation*}
\begin{array}{lll}
\|\mathcal{K}_{j+1,1}(t) u_0\|_{L^2}\leq Ce^{-ct}\|u_0\|_{L^2}, &\quad &\|K_{j+1,1}(t) u_0\|_{L^2}\leq C\min\{1,t^{-\frac{n}{4}-\frac{1}{2}}\}\|u_0\|_{L^1},\\&&\\
\|\mathcal{K}_{j+1,i+1}(t) v_0^i\|_{L^2}\leq Ce^{-ct}\|v_0^i\|_{L^2},& \quad &\|K_{j+1,i+1}(t) v_0^i\|_{L^2}\leq C\min\{1,t^{-\frac{n}{4}-1}\}\|v_0^i\|_{L^1}.
\end{array}
\end{equation*}

We pass now to estimate the second term in (\ref{vL2}). Decomposing the integral term, we get
\begin{align*}
\int_0^t \|&\Gamma^h_{j+1}(t-s) (\bar{B}(\phi,\nabla\phi)(s)w(s)+H(\phi,\nabla\phi,w)(s))\|_{L^2}ds \\
\leq&\int_0^t \|\mathcal{K}_{j+1}(t-s) (\bar{B}(\phi,\nabla\phi)(s)w(s)+H(\phi,\nabla\phi,s)(s))\|_{L^2}ds\\
+&\int_0^t \|K_{j+1}(t-s)(\bar{B}(\phi(s),\nabla\phi(s))w(s)+H(\phi,\nabla\phi,u)(s))\|_{L^2}ds.
\end{align*}
Let us focus on the first integral on the right-hand side.
We can notice that, since the singular part of the Green Kernel has the same decay rate for both conservative and dissipative variable, we can estimate this term, as done previously in the estimate of function $u$. Then,
\begin{align*}
\int_0^t \|&\mathcal{K}_{i+1}(t-s) (\bar{B}(\phi,\nabla\phi)w+H(\phi,\nabla\phi,w)(s))\|_{L^2}ds\\
\leq&CB_k\min\{1,t^{-(\frac{n}{2}+\nu_{0})}\}\left(N_{D_x^1 \phi}^{\frac{n}{2}}(t)M_{v}^{\nu_{0}}(t)
+N_{\phi}^{\frac{n}{2}}(t)M_{v}^{\nu_{0}}(t)\right)\\
+&CH_kG_k \min\{1,t^{-(\frac{n}{4}+\frac{n}{2})}\}\left(M_{u}^{\frac{n}{4}}(t)N_{\phi}^{\frac{n}{2}}(t)+ M_{u}^{\frac{n}{4}}(t)N_{D_x^1 \phi}^{\frac{n}{2}}(t)\right),
\end{align*}
where $\nu_{0}=\min\left\{\frac{n}{4}+\frac{1}{2},\frac{n}{2}\right\}$.
On the other hand, when estimating the dissipative term of Green Kernel diffusive part, we get a faster decay, with respect to the conservative variable $u$. The dissipative part, being strongly influenced by the dissipation, decays at the rate $t^{-\frac{1}{2}}$ faster of the conservative one.\\ 
Proceeding as done before, 
\begin{align*}
\int_0^t\|& K_{j+1}(t-s) (\bar{B}(\phi,\nabla\phi)(s)w(s)+H(\phi,\nabla\phi,w)(s))\|_{L^2}ds\\
\leq&\int_0^t C\min\{1,(t-s)^{-\frac{n}{4}-1}\}\|(\bar{B}(\phi,\nabla\phi)(s)w(s)+H(\phi,\nabla\phi,w)(s))\|_{L^1}ds\\
\leq& \int_0^tC\min\{1,(t-s)^{-\frac{n}{4}-1}\}B_k(\|\phi(s)\|_{L^2}\|v(s)\|_{L^2}+\|\nabla\phi(s)\|_{L^2}\|v(s)\|_{L^2})ds\\
+& \int_0^tC\min\{1,(t-s)^{-\frac{n}{4}-1}\}H_kG_k(\|\phi(s)\|_{L^2}\|u(s)\|_{L^2}+\|\nabla\phi(s)\|_{L^2}\|u(s)\|_{L^2})ds\\
\leq& B_K(M_{\phi}^{\frac{n}{4}}(t)M_{v}^{\nu_{0}}(t)+M_{D_x^1 \phi}^{\frac{n}{4}}(t)M_{v}^{\nu_{0}}(t))\int_0^tC\min\{1,(t-s)^{-\frac{n}{4}-1}\}\min\{1,s^{-(\frac{n}{4}+\nu_{0})}\}ds\\
+&H_kG_k(M_{\phi}^{\frac{n}{4}}(t)M_{u}^{\frac{n}{4}}(t)+M_{D_x^1 \phi}^{\frac{n}{4}}(t)M_{u}^{\frac{n}{4}}(t))\int_0^tC\min\{1,(t-s)^{-\frac{n}{4}-1}\}\min\{1,s^{-\frac{n}{2}}\}ds.
\end{align*}

Thanks to Lemma \ref{Lemma5.2} we deduce
\begin{align*}
\int_0^t \|& K_{j+1}(t-s) (\bar{B}(\phi,\nabla\phi)(s)w(s)+H(\phi,\nabla\phi,w)(s))\|_{L^2}\\
\leq& \min\{1,t^{-\nu_0}\}C_k(M_{\phi}^{\frac{n}{4}}(t)M_{v}^{\nu_{0}}(t)+M_{D_x^1 \phi}^{\frac{n}{4}}(t)M_{v}^{\nu_{0}}(t)
+M_{\phi}^{\frac{n}{4}}(t)M_{u}^{\frac{n}{4}}(t)+ M_{D_x^1 \phi}^{\frac{n}{4}}(t)M_{u}^{\frac{n}{4}}(t)).
\end{align*}
where $\nu_0=\min \left\{\frac{n}{2},\frac{n}{4}+1\right\}$.
Then, summing the previous inequalities we obtain the $L^2$-norm of the function $v$.
\begin{eqnarray}\label{VL2a}
\|v_j(t)\|_{L^2}&\leq& C\left[e^{-ct}(\|u_0\|_{L^2}+\sum_i\|v_0^i\|_{L^2})+\min\{1,t^{-\frac{n}{4}-\frac{1}{2}}\}\|u^0\|_{L^1}+\min\{1,t^{-\frac{n}{4}-1}\}\sum_i\|v_0^i\|_{L^1}\right]\nonumber\\
&+&C_k\left[\min\{1,t^{-(\frac{n}{2}+\nu_{0})}\}(N_{D_x^1 \phi}^{\frac{n}{2}}(t)M_{v}^{\nu_{0}}(t)+N_{\phi}^{\frac{n}{2}}(t)M_{v}^{\nu_{0}}(t))\right.\\
&+&\min\{1,t^{-\frac{3}{4}n}\}(M_{u}^{\frac{n}{4}}(t)N_{\phi}^{\frac{n}{2}}(t)+M_{u}^{\frac{n}{4}}(t)N_{D_x^1 \phi}^{\frac{n}{2}}(t))\nonumber\\
&+&\left.\min\{1,t^{-\nu_0}\}(M_{\phi}^{\frac{n}{4}}(t)M_{v}^{\nu_{0}}(t)+M_{D_x^1 \phi}^{\frac{n}{4}}(t)M_{v}^{\nu_{0}}(t)+M_{\phi}^{\frac{n}{4}}(t)M_{u}^{\frac{n}{4}}(t)+ M_{D_x^1 \phi}^{\frac{n}{4}}(t)M_{u}^{\frac{n}{4}}(t))\right]\nonumber.
\end{eqnarray}

%%%%%%%%%%%%%%%%%%%%%%%%%%%%%%%%%%%%%%%%%%%%%%%%%%%%%%%%%%%%%%%%%%%%%%%%%%%%%%%%%%%%%%%%%%%%%%%%
%%%%%%%%%%%%%%%%%%%%%%%%%%%%%%%%%%%%%%%%%%%%%%%%%%%%%%%%%%%%%%%%%%%%%%%%%%%%%%%%%%%%%%%%%%%%%%%%%%%%%%%%%%%%%%%%%%%%%%%%%%%%%%%%%%%%%%%%%%%%%%%%%%%%%%%%%%%%%%%%%%%%%%%%%%%%%%
In order to complete our study we need to estimate the $L^2$- norm of the $s$-derivative of function $v$ and its $L^\infty$- norm.
\paragraph{$L^2$-estimate for $D_x^s v$\\}
Regarding the $s$-order estimate for $v_j$, we have
\begin{eqnarray}\label{DVL2a}
\|D_x^s v_j(t)\|_{L^2}&\leq& C\left[e^{-ct}(\|D_x^su_0\|_{L^2}+\sum_{i=1}^n\|D_x^sv_0^i\|_{L^2}+\min\{1,t^{-\frac{n}{4}-\frac{1}{2}-\frac{s}{2}}\}\|u_0\|_{L^1}\right.\nonumber\\
&+&\left.\min\{1,t^{-\frac{n}{4}-1-\frac{s}{2}}\}\sum_{i=1}^n\|v_0^i\|_{L^1})\right]+ C_k\left[\min\{1,t^{-(\tilde{\delta}+\frac{n}{2})}\}(M_{D_x^s \phi}^{\tilde{\delta}}(t)N_{v}^{\frac{n}{2}}(t)+M_{D_x^{s+1} \phi}^{\tilde{\delta}}(t)N_{v}^{\frac{n}{2}}(t))\right.\nonumber\\
&+&\min\{1,t^{-(\frac{n}{2}+\tilde{\nu})}\}(N_{D_x^1 \phi}^{\frac{n}{2}}(t)M_{D_x^s v}^{\tilde{\nu}}(t)+ N_{\phi}^{\frac{n}{2}}(t)M_{D_x^s v}^{\tilde{\nu}}(t))\\
&+&\min\{1,t^{-(\tilde{\delta}+\frac{n}{2})}\}(M_{D_x^s\phi}^{\tilde{\delta}}(t)N_{u}^{\frac{n}{2}}(t)
+M_{D_x^s\phi}^{\tilde{\delta}}(t)N_{u}^{\frac{n}{2}}(t)+M_{D_x^s u}^{\tilde{\delta}}(t)N_{\phi}^{\frac{n}{2}}(t)+M_{D_x^s u}^{\tilde{\delta}}(t)N_{D_x^1 \phi}^{\frac{n}{2}}(t)\nonumber\\
&+&\left.\min\{1,t^{-\tilde{\nu}}\}(M_{\phi}^{\frac{n}{4}}(t)M_{ v}^{\nu_{0}}(t)+M_{D_x^1 \phi}^{\frac{n}{4}}(t)M_{ v}^{\nu_{0}}(t)+M_{\phi}^{\frac{n}{4}}(t)M_{v}^{\nu_{0}}(t)
+M_{D_x^1\phi}^{\frac{n}{4}}(t)M_{v}^{\nu_{0}}(t)\right].\nonumber
\end{eqnarray}
Let us recall that $\tilde{\nu}=\min\left\{\frac{n}{4}+1+ \frac{s}{2},\frac{n}{2}\right\}$.

%%%%%%%%%%%%%%%%%%%%%%%%%%%%%%%%%%%%%%%%%%%%%%%%%%%%%%%%%%%%%%%%%%%%%%%%%%%%%%%%%%%%%%%%%%%%%%%%%%%%%%%%%%%%%%%%%%%%%%%%%%%%%%%%%%%%%%%%%%%%%%%%%%%%%%%%%%%%%%%%%%%%%%%%%%%%%%%%%%%%%%%%%%%%%%%%%%%Dissipative variable L^\infty estimate

\paragraph{$L^\infty$-estimate for $v$\\}
On the other hand, for the $L^\infty$- norm of function $v_j$, we get the following estimates
\begin{eqnarray}\label{VLinfa}
\|v_j(t)\|_{L^\infty}&\leq& C\left[e^{-ct}(\|u_0\|_{H^s}+\sum_{i=1}^n\|v_0^i\|_{H^s})+\min\{1,t^{-\frac{n}{2}-\frac{1}{2}}\}\|u_0\|_{L^1}+\min\{1,t^{-\frac{n}{2}-1}\}\sum_{i=1}^n\|v_0^n\|_{L^1}\right]\nonumber\\
&+&C_k\left[\min\{1,t^{-(\frac{n}{2}+\nu_0)}\}(N_{D_x^1 \phi}^{\frac{n}{2}}(t)M_{v}^{\nu_0}(t)+N_{\phi}^{\frac{n}{2}}(t)M_{v}^{\nu_0}(t))\right.\nonumber\\
&+&\min\{1,t^{-\frac{3}{4}n}\}(M_{u}^{\frac{n}{4}}(t)N_{\phi}^{\frac{n}{2}}(t)+ M_{u}^{\frac{n}{4}}(t)N_{D_x^1 \phi}^{\frac{n}{2}}(t))\nonumber\\
&+& \min\{1,t^{-(\tilde{\delta}+\frac{n}{2})}\}(M_{D_x^s \phi}^{\tilde{\delta}}(t)N_{v}^{\frac{n}{2}}(t)+M_{D_x^{s+1} \phi}^{\tilde{\delta}}(t)N_{v}^{\frac{n}{2}}(t))\\
&+&\min\{1,t^{-(\frac{n}{2}+\tilde{\nu})}\}(N_{D_x^1 \phi}^{\frac{n}{2}}(t)M_{D_x^s v}^{\tilde{\nu}}(t)+N_{\phi}^{\frac{n}{2}}(t)M_{D_x^s v}^{\tilde{\nu}}(t))\nonumber\\
&+&\min\{1,t^{-(\tilde{\delta}+\frac{n}{2})}\}(M_{D_x^s\phi}^{\tilde{\delta}}(t)N_{u}^{\frac{n}{2}}(t)
+M_{D_x^s\phi}^{\tilde{\delta}}(t)N_{u}^{\frac{n}{2}}(t)+M_{D_x^s u}^{\tilde{\delta}}(t)N_{\phi}^{\frac{n}{2}}(t)+M_{D_x^s u}^{\tilde{\delta}}(t)N_{D_x^1 \phi}^{\frac{n}{2}}(t))\nonumber\\
&+&\min\{1,t^{-\frac{n}{2}}\}(M_{\phi}^{\frac{n}{4}}(t)M_{v}^{\nu_{0}}(t)+M_{D_x^1 \phi}^{\frac{n}{4}}(t)M_{v}^{\nu_{0}}(t))+M_{\phi}^{\frac{n}{4}}(t)M_{u}^{\frac{n}{4}}(t)+M_{D_x^1 \phi}^{\frac{n}{4}}(t)M_{u}^{\frac{n}{4}}(t))\left.\right].\nonumber
\end{eqnarray}

Once that, decay rates of variable have been determinated by inequalities \eqref{UL2a}, \eqref{DUL2a}, \eqref{ULinfa}, \eqref{VL2a}, \eqref{DVL2a}, \eqref{VLinfa}, we apply Proposition \ref{phi_estimates} to get the following estimates for the functionals related to the solution $(u,v)$.
For $t>\epsilon>0$,
\begin{equation*}
\begin{array}{lll}\label{MU}
M_{u}^{\frac{n}{4}}(t)&\leq& C_1 \left[ E_0+ D_0+ D_0\left(M_{u}^{\frac{n}{4}}(t)+N_u^{\frac{n}{2}}(t)+M_{v}^{\nu_{0}}(t)+N_{v}^{\frac{n}{2}}(t)\right) +(N_u^{\frac{n}{2}}(t))^2\right. \\\\
&+&\left.M_{u}^{\frac{n}{4}}(t)N_{v}^{\frac{n}{2}}(t)+N_u^{\frac{n}{2}}(t)M_{u}^{\frac{n}{4}}(t)+(M_{u}^{\frac{n}{4}}(t))^2+M_{u}^{\frac{n}{4}}(t)M_{v}^{\nu_{0}}(t) \right],
\end{array}
\end{equation*}
where $\nu_{0}=\min \left\{\frac{n}{2},\frac{n}{4}+\frac{1}{2}\right\}$.
\begin{equation*}
\begin{array}{lll}\label{MDU}
M_{D_x^s u}^{\tilde{\delta}}(t)&\leq& C_2\left[ E_0+ D_0+ D_0\left(M_{u}^{\frac{n}{4}}(t)+N_u^{\frac{n}{2}}(t)+ M_{D_x^s u}^{\tilde{\delta}}(t)+M_{v}^{\nu_{0}}(t)+N_{v}^{\frac{n}{2}}(t)+M_{D_x^s v}^{\tilde{\nu}}(t)\right)\right.\\\\
&+&N_u^{\frac{n}{2}}(t)M_{D_x^s v}^{\tilde{\nu}}(t)+(M_{u}^{\frac{n}{4}}(t))^2+ M_{D_x^s u}^{\tilde{\delta}}(t)N_{v}^{\frac{n}{2}}(t)+ \left.N_u^{\frac{n}{2}}(t)M_{D_x^s u}^{\tilde{\delta}}(t)+M_{u}^{\frac{n}{4}}(t)M_{v}^{\nu_{0}}+(M_{D_x^s u}^{\tilde{\delta}}(t))^2 
\right],
\end{array}
\end{equation*}
where 
$\tilde{\delta}=\min\{\frac{n}{4}+\frac{1}{2}+\frac{s}{2},\frac{n}{2}\}$, and
$\tilde{\nu}=\min\{\frac{n}{4}+1+\frac{s}{2},\frac{n}{2}\}$.

\begin{equation*}
\begin{array}{lll}\label{NU}N_{u}^{\frac{n}{2}}(t)&\leq& C_3\left[ E_0+  D_0+ D_0 \left(M_{u}^{\frac{n}{4}}(t)+N_u^{\frac{n}{2}}(t)+ M_{D_x^s u}^{\tilde{\delta}}(t) +M_{v}^{\nu_{0}}(t)+N_{v}^{\frac{n}{2}}(t)+M_{D_x^s v}^{\tilde{\nu}}(t)\right)\right. \\\\
&+&M_{u}^{\frac{n}{4}}(t)N_{v}^{\frac{n}{2}}(t)+N_u^{\frac{n}{2}}(t)M_{v}^{\nu_{0}}(t)+N_u^{\frac{n}{2}}(t)M_{u}^{\frac{n}{4}}(t)+(M_{u}^{\frac{n}{4}}(t))^2+ M_{D_x^s u}^{\tilde{\delta}}(t)N_{v}^{\frac{n}{2}}(t)\\\\
&+&\left. N_u^{\frac{n}{2}}(t)M_{D_x^s u}^{\tilde{\delta}}(t)+M_{u}^{\frac{n}{4}}(t)M_{v}^{\nu_{0}}(t)+M_{D_x^s v}^{\tilde{\nu}}(t)N_u^{\frac{n}{2}}(t)+(M_{D_x^s u}^{\tilde{\nu}}(t))^2 + M_{D_x^s u}^{\frac{n}{2}}(t)N_u^{\frac{n}{2}}(t)
\right]. \\\\
\end{array}
\end{equation*}

\begin{equation*}
\begin{array}{lll}\label{MV}
M_{v}^{\nu_{0}}(t)&\leq& C_4 \left[ E_0+ D_0+ D_0\left.(M_{u}^{\frac{n}{4}}(t)+N_u^{\frac{n}{2}}(t)+M_{v}^{\nu_{0}}(t)+N_{v}^{\frac{n}{2}}(t)\right) (N_u^{\frac{n}{2}}(t))^2\right. \\\\
&+&\left.M_{u}^{\frac{n}{4}}(t)N_{v}^{\frac{n}{2}}(t)+N_u^{\frac{n}{2}}(t)M_{u}^{\frac{n}{4}}(t)+(M_{u}^{\frac{n}{4}}(t))^2+M_{u}^{\frac{n}{4}}(t)M_{v}^{\nu_{0}}(t) \right].\\\\\\ 
%\end{array}
%\end{equation}

%\begin{equation}
%\begin{array}{lll}\label{MDV}
M_{D_x^s v}^{\tilde{\nu}}(t)&\leq& C_{5} \left[ E_0+ D_0+ D_0\left(M_{u}^{\frac{n}{4}}(t)+N_u^{\frac{n}{2}}(t)+ M_{D_x^s u}^{\tilde{\delta}}(t)+M_{v}^{\nu_{0}}(t)+N_{v}^{\frac{n}{2}}(t)+M_{D_x^s v}^{\tilde{\nu}}(t)\right)\right.\\\\
&+&N_u^{\frac{n}{2}}(t)M_{D_x^s v}^{\tilde{\nu}}(t)+(M_{u}^{\frac{n}{4}}(t))^2+ M_{D_x^s u}^{\tilde{\delta}}(t)N_{v}^{\frac{n}{2}}(t)+ \left.N_u^{\frac{n}{2}}(t)M_{D_x^s u}^{\tilde{\delta}}(t)+M_{u}^{\frac{n}{4}}(t)M_{v}^{\nu_{0}}+(M_{D_x^s u}^{\tilde{\delta}}(t))^2 
\right].\\\\\\

%\begin{equation}
%\begin{array}{lll}\label{NV}
N_{v}^{\frac{n}{2}}(t)&\leq& C_6 \left[ E_0+  D_0+ D_0 \left(M_{u}^{\frac{n}{4}}(t)+N_u^{\frac{n}{2}}(t)+ M_{D_x^s u}^{\tilde{\delta}}(t) +M_{v}^{\nu_{0}}(t)+N_{v}^{\frac{n}{2}}(t)+M_{D_x^s v}^{\tilde{\nu}}(t)\right)\right. \\\\
&+&M_{u}^{\frac{n}{4}}(t)N_{v}^{\frac{n}{2}}(t)+N_u^{\frac{n}{2}}(t)M_{v}^{\nu_{0}}(t)+N_u^{\frac{n}{2}}(t)M_{u}^{\frac{n}{4}}(t)+(M_{u}^{\frac{n}{4}}(t))^2+ M_{D_x^s u}^{\tilde{\delta}}(t)N_{v}^{\frac{n}{2}}(t)\\\\
&+&\left. N_u^{\frac{n}{2}}(t)M_{D_x^s u}^{\tilde{\delta}}(t)+M_{u}^{\frac{n}{4}}(t)M_{v}^{\nu_{0}}(t)+M_{D_x^s v}^{\tilde{\nu}}(t)N_u^{\frac{n}{2}}(t)+(M_{D_x^s u}^{\tilde{\nu}}(t))^2 + M_{D_x^s u}^{\frac{n}{2}}(t)N_u^{\frac{n}{2}}(t)+ (N_u^{\frac{n}{2}}(t))^2
\right].  \end{array}
\end{equation*}

where $D_0=\|\phi_0\|_{H^{s+1}}$, $E_0=\max\{ \|u_0\|_{H^{s}}, \|u_0\|_{L^1}, \|v_0\|_{H^s},\|v_0\|_{L^1}\}$ and the constant $C_i=C_i(F_k,K, C_{b'},C_{h'})$.\\

Let us define 
$$
P(t):=M_{u}^{\frac{n}{4}}(t)+N_u^{\frac{n}{2}}(t)+ M_{D_x^s u}^{\tilde{\delta}}(t)+M_{v}^{\nu_{0}}(t)+N_{v}^{\frac{n}{2}}(t)+M_{D_x^sv}^{\tilde{\nu}}(t).
$$
We can notice that all the previous estimates are linear combinations of 
sums of type: 
 $A_0 F_w^\delta(t)+$ $F_w^\delta(t)F_{w_1}^{\delta^1}(t)$ where  $F_w^\delta, F_{w_1}^{\delta^1}$ are 
 terms of $P(t)$. Then it is possible to estimate each of them with $A_0P(t)+P(t)^2$. 
It follows that if initial data are small, we have

\begin{equation}
C_kP(t)^2-(1-C_{k0})P(t)+C_0 \geq 0,
\end{equation}
where $C_k$ is a positive constant depending on $K$, $C_{k0}$ is a positive constant depending on $K$ and on data, and $C_0$ also is a positive constant depending on data.
For suitably small initial data, this inequality implies that $M_{u}^{\frac{n}{4}}(t)$, $N_u^{\frac{n}{2}}(t)$, $M_{D_x^s u}^{\tilde{\delta}}(t)$, $M_{v}^{\nu_{0}}(t)$, $N_{v}^{\frac{n}{2}}(t)$, $M_{D_x^s v}^{\tilde{\nu}}(t)$ remain
bounded, as far as $\|u,v\|_{L^\infty}\leq K$ and $\|\phi\|_{W^{1,\infty}}\leq K$. When $t>1$ this implies that $\|w(t)\|_{L^\infty}$ does not increase with $t$. Thanks to Proposition \ref{phi_estimates}
the same is true for $N^{\frac{n}{2}}_{\phi}$ and $N^{\frac{n}{2}}_{D_x^1 \phi}$. \\
Since we have obtained that $\|\phi(t)\|_{W^{1,\infty}}$ is bounded, 
from Lemma \ref{boundinf} and the continuation principle we have the global existence of smooth solutions to system (\ref{sistema_GUMA}).

\subsubsection{Optimal Decay Rates}
In order to complete our proof, we need to improve the decay rates of $D_x^s u$, $D_x^s v$, $D_x^s \phi$ in the $L^2$ norm.\\
By the previous estimates, we got that, independently from the derivative order $s$, the decays rates of these function are equal to  $\tilde{\delta}= \min\left\{\frac{n}{4}+\frac{1}{2}+\frac{s}{2},\frac{n}{2}\right\}$, for $u$, $\phi$ and equal to $\tilde{\nu}= \min\left\{\frac{n}{4}+1+\frac{s}{2},\frac{n}{2}\right\}$ for $v$. This implies that, for small $n$, even if the derivative order is high, we get always the decay $t^{-\frac{n}{2}}$.%, 

Looking at inequality \eqref{DUL2a} we notice that these decays come from the estimates related to Green Kernel diffusive part.
Then, we need to adopt a different strategy to estimate these terms and overcome the difficulty, i.e. split the derivatives on both terms. 
We show this procedure by induction on a simple source term $\phi u$.
 
\begin{itemize}
\item Let $s=1$.
Since in this case we cannot split the order of derivative, we proceed as done before keeping the derivative on the Green Kernel. Then,
\begin{align*}
\int_0^t\| D_x^1K_{12}(t-s) (u(s)\phi(s))\|_{L^2}ds\leq &\int_0^t C\min\{1,(t-s)^{-\frac{n}{4}-1}\}\|\phi)(s)u(s)\|_{L^1}ds\\
\leq&\int_0^tC\min\{1,(t-s)^{-\frac{n}{4}-1}\} \|\phi(s)\|_{L^2}\|u(s)\|_{L^2}ds\\
\leq&CM_{\phi}^{\frac{n}{4}}(t) M_{u}^{\frac{n}{4}}(t)\int_0^t\min\{1,(t-s)^{-\frac{n}{4}-1}\}\min\{1,s^{-\frac{n}{2}}\}ds\\
\leq &\min\{1,t^{-\delta_1}\}C M_{\phi}^{\frac{n}{4}}(t)M_{u}^{\frac{n}{4}}(t),
\end{align*}
where $\delta_1=\min\left\{\frac{n}{4}+1,\frac{n}{2}\right\}$. 

\item Let us consider the second order derivative, i.e. $s=2$. Now we split the derivative both on 
 the Green Kernel and the source term, proceeding as follows, 
\begin{align*}
\int_0^t\|  D_x^2 K_{12}(t-s) & (\phi(s)u(s)\|_{L^2}ds= \int_0^t\| D_x^1 K_{12}(t-s) D_x^1(\phi(s)u(s))\|_{L^2}ds\\
\leq&\int_0^tC \min\{1,(t-s)^{-\frac{n}{4}-1}\}\|D_x^1 \phi(s)u(s)\|_{L^1}ds\\
\leq&\int_0^tC\min\{1,(t-s)^{-\frac{n}{4}-1}\}(\|D_x^1\phi\|_{L^2}\|u(s)\|_{L^2}+\|\phi(s)\|_{L^2}\|D_x^1u(s)\|_{L^2})\\
+&C(M_{D_x^1 \phi}^{\delta_1} M_{u}^{\frac{n}{4}}(t)+M_{\phi}^{\frac{n}{4}} M_{D_x^1u}^{\delta}(t))\int_0^t\min\{1,(t-s)^{-\frac{n}{4}-1}\}\min\{1,s^{-\frac{n}{4}-\delta_1}\}ds\\
\leq& \min\{1,t^{-\delta_2}\}C(M_{\phi}^{\frac{n}{4}}(t)M_{D_x^1 u}^{\delta_1}(t)
+M_{D_x^1 \phi}^{\delta_1}(t)M_{u}^{\frac{n}{4}}(t)),
\end{align*}
where $\delta_2=\min\left\{\frac{n}{4}+1,\frac{n}{4}+\delta_1\right\}$.

\item Finally, we iterate the procedure for a generic $s$, splitting the derivatives as  follows.  We left 
$\left[\frac{s+1}{2}\right]$ derivatives on the Green Kernel, and the remaining ones $\left[\frac{s}{2}\right]$, on the source terms.
By this way we get
\begin{align*}
\int_0^t \|D_x^s K_{12}(t-s)(\phi(s)u(s))\|_{L^2}\leq& \int_0^t\| D_x^{\left[\frac{s+1}{2}\right]}K_{12}(t-s) D_x^{\left[\frac{s}{2}\right]}(\phi(s)u(s))\|_{L^2}ds\\
\leq &\min\{1,t^{-\delta_s}\}C( M_{\phi}^{\frac{n}{4}}(t)M_{D_x^r u}^{\delta_{r}}(t)+M_{D_x^r\phi}^{\delta_r}(t)M_{u}^{\frac{n}{4}}(t)),
\end{align*}
where $\delta_s=\min\left\{\frac{n}{4}+\frac{1}{2}+\frac{1}{2}\left[\frac{s+1}{2}\right],\frac{n}{4}+\delta_r\right\}$, with $r=\left[\frac{s}{2}\right]$.\\
Thus, through this simple procedure, we are able to obtain faster decays rates for the $s$ derivative of the functions $u,v$ and $\phi$.
More precisely for the $s-$derivative of function $v$, since the Green Kernel has a faster decay, we get the rate  $\nu_s=\min\left\{\frac{n}{4}+1+\frac{1}{2}\left[\frac{s+1}{2}\right],\frac{n}{4}+\delta_r\right\}$.
\end{itemize}

\end{proof}

\section{The Cauchy Problem : Perturbations of Non-Zero Constant Stationary States}\label{costante}
The aim of this section is to investigate the behavior of non-zero small constant states.
For the sake of simplicity we will consider the system with simpler source terms,
\begin{equation*}
\label{sistema_easy}
\left\{
\begin{array}{l}
\partial_{t} \tilde{u} +\nabla \cdot \tilde{v} = 0, \\\\
\partial_{t}\tilde{v} + \nabla \tilde{u}= -\tilde{v}+\tilde{u}\nabla \tilde{\phi}, \\\\
\partial_{t} \tilde{\phi} =\Delta \tilde{\phi}+ a\tilde{u} -b\tilde{\phi},
\end{array}
\right.
\end{equation*}
where $(\tilde{u},\tilde{v},\tilde{\phi})=(\bar{u}+u,v,\bar{\phi}+\phi)$, $(\bar{u},0, \bar{\phi})$ is a stationary solution with $\bar{\phi}=\frac{a}{b}\bar{u}$, and $(u,v,\phi)$ is a perturbation.
Therefore we can rewrite the previous system as follows
\begin{equation}
\label{sistema_easy2}
\left\{
\begin{array}{l}
\partial_{t} u +\nabla \cdot v = 0, \\\\
\partial_{t} v+ \nabla u= -v+(u +\bar{u})\nabla \phi, \\\\
\partial_{t} \phi =\Delta \phi+ au -b\phi.
\end{array}
\right.
\end{equation}
This system is supplemented by the initial conditions
\begin{equation}\label{dati_stato}
u_0, v_0\in H^{s}(\R^n)\cap L^1(\R^n), \quad \phi_0 \in H^{s+1}(\R^n)\cap L^1(\R^n).
\end{equation}

In order to prove the global existence result and the decay of solutions to \eqref{sistema_easy2} we will proceed along the lines of the previous sections. Then starting from a local solution to \eqref{sistema_easy2}, which is guaranteed by Theorem \ref{locale}, we will get estimates and decay rates of the $H^s$ and $L^\infty$ norm. Then by the continuation principle \ref{continuation} we will obtain our existence result.\\
To get the decay of solutions we need to adapt the technique used in the above proof of stability for the zero constant state, to treat the linear term $\bar{u}\nabla\phi$, which does not present enough polynomial decay.\\ Let us consider a local solution to system \eqref{sistema_easy2}. 
Taking into account the expressions for the Green function, we are going to estimates the norm of solutions. Even if these new estimates are not optimal, they are in suitable spaces.

\begin{theorem}\label{Theo_BHN2}
Consider the linear PDE in the conservative-dissipative form
\begin{equation*}\label{theo_CD}
\partial_t w+\sum_{j=1}^n A_j \partial_{x_j} w= Bw,
\end{equation*}
where $A_j$, $B$ satisfy the assumption (SK), and let $Q_0=R_0L_0$, $Q_{-}=I-Q_0=R_{-}L_{-}$ be the eigenprojectors on the null 
space and the negative definite part of $B$ with  $L_0=R_0^T=[I_{n_1} \; 0 ]$ and $L_-=R_-^T=[0 \; I_{n_2}]$.

Then, for any function $w_0\in   L^1 \cap L^2 (\mathbb{R}^n,\mathbb{R}^{n+1})$  the solution of the linear dissipative system can be decomposed as
$$
w(t)=\Gamma^h(t)\ast w_0=K(t)w_0+\mathcal{K}(t) w_0,
$$  
where  for any multi index $\beta$, the following estimates hold :

\begin{eqnarray*}\label{k_estimates}
\left\|L_0D^\beta K(t)w_0\right\|_{L^2}&\leq& C \min\{1, t^{-(\frac{n}{8}+\frac{\beta}{2})}\}\|L_0w^0\|^\frac{1}{2}_{L^2}\|L_0w_0\|_{L^\infty}^{\frac{1}{2}}\\
&+&C \min\{1, t^{-(\frac{n}{8}+\frac{\beta}{2}+\frac{1}{2})}\}\|L_-w^0\|^\frac{1}{2}_{L^2}\|L_-w_0\|^\frac{1}{2}_{L^\infty}\nonumber\\
\left\|L_-D^\beta K(t)w_0\right\|_{L^2}&\leq& C \min\{1, t^{-(\frac{n}{8}+\frac{\beta}{2}+\frac{1}{2})}\}\|L_0w_0\|^\frac{1}{2}_{L^2}\|L_0w^0\|_{L^\infty}^{\frac{1}{2}}\\&+&C \min\{1, t^{-(\frac{n}{8}+\frac{\beta}{2}+1)}\}\|L_-w^0\|^\frac{1}{2}_{L^2}\|L_-w_0\|^\frac{1}{2}_{L^\infty}\nonumber\\
\left\|L_0D^\beta K(t)w_0\right\|_{L^\infty}&\leq& C \min\{1, t^{-(\frac{n}{4}+\frac{\beta}{2})}\}\|L_0w_0\|_{L^2}+C \min\{1, t^{-(\frac{n}{4}+\frac{\beta}{2}+\frac{1}{2})}\}\|L_-w_0\|_{L^2}\nonumber\\
\left\|L_-D^\beta K(t)w_0\right\|_{L^\infty}&\leq &C \min\{1, t^{-(\frac{n}{4}+\frac{\beta}{2}+\frac{1}{2})}\}\|L_0w_0\|_{L^2}+C \min\{1, t^{-(\frac{n}{4}+\frac{\beta}{2}+1)}\}\|L_-w_0\|_{L^2}\nonumber.
\end{eqnarray*}

\end{theorem}

\begin{proof} \normalsize
As shown in \cite{BiHaNa} by introducing the Fourier transform there exist two constants $c,C>0$ such that it is possible to estimate the decomposition of the Green function as: 
\begin{align*}
\left|L_0\widehat{K(t)}w_0\right|& \leq Ce^{-c|\xi|^2t}\left(|L_0\hat{w_0}|+|\xi||L_-\hat{w_0}|\right)\\
\left|L_-\widehat{K(t)}w_0\right|& \leq Ce^{-c|\xi|^2t}\left(|\xi||L_0\hat{w_0}|+|\xi|^2|L_-\hat{w_0}|\right).
\end{align*}
where the projectors $L_0$ and $L_-$ are given by   $L_0=[I_{n_1} \; 0 ]$ and $L_-=[0 \; I_{n_2}]$.
Using these inequalities we obtain 
\begin{align*}
\left\|L_0D^\beta K(t)w_0\right\|_{L^2}^2&\leq C\int_0^\infty\int_{S^{n-1}} e^{-c|\xi|^2t} (|L_0\hat{w_0}|^2+|\xi|^2|L_-\hat{w_0}|^2)|\xi|^{2\beta} |\xi|^{n-1}d\varsigma d|\xi|\\
&\leq C\left(\int_0^\infty\int_{S^{n-1}} e^{-2c|\xi|^2t} |\xi|^{(n-1)}|\xi|^{4\beta}d\varsigma d|\xi|\right)^{\frac{1}{2}} \left( \int_0^\infty\int_{S^{n-1}}|L_0\hat{w_0}|^4 |\xi|^{(n-1)} d\varsigma  d|\xi|\right)^{\frac{1}{2}}\\
+& C\left(\int_0^\infty\int_{S^{n-1}} e^{-2c|\xi|^2t}|\xi|^4 |\xi|^{(n-1)}|\xi|^{4\beta}d\varsigma d|\xi|\right)^{\frac{1}{2}}\left( \int_0^\infty\int_{S^{n-1}}|L_-\hat{w_0}|^4 |\xi|^{(n-1)}d\varsigma  d|\xi|\right)^{\frac{1}{2}}\\
&\leq C \min\{1, t^{-(\frac{n}{4}+\beta)}\}\|L_0w_0\|^2_{L^4}+C \min\{1, t^{-(\frac{n}{4}+\beta+1)}\}\|L_-w_0\|^2_{L^4},
\end{align*}
and
\begin{align*}
\left\|L_-D^\beta K(t)w_0\right\|_{L^2}^2&\leq C\int_0^\infty\int_{S^{n-1}} e^{-c|\xi|^2t} (|\xi|^2|L_0\hat{w_0}|^2+|\xi|^4|L_-\hat{w_0}|^2)|\xi|^{2\beta} |\xi|^{n-1}d\varsigma d|\xi|\\
&\leq C\left(\int_0^\infty\int_{S^{n-1}} e^{-2c|\xi|^2t} |\xi|^4|\xi|^{(n-1)}|\xi|^{4\beta}d\varsigma d|\xi|\right)^{\frac{1}{2}} \left( \int_0^\infty\int_{S^{n-1}}|L_0\hat{w_0}|^4 |\xi|^{(n-1)}d\varsigma d|\xi|\right)^{\frac{1}{2}}\\
+& C\left(\int_0^\infty\int_{S^{n-1}} e^{-2c|\xi|^2t}|\xi|^8 |\xi|^{(n-1)}|\xi|^{4\beta}d\varsigma d|\xi|\right)^{\frac{1}{2}}\left( \int_0^\infty\int_{S^{n-1}}|L_-\hat{w_0}|^4 |\xi|^{(n-1)} d\varsigma d|\xi|\right)^{\frac{1}{2}}\\
&\leq C \min\{1, t^{-(\frac{n}{4}+\beta+1)}\}\|L_0w_0\|^2_{L^4}+C \min\{1, t^{-(\frac{n}{4}+\beta+2)}\}\|L_-w_0\|^2_{L^4},
\end{align*}
We have that:
\begin{align}
\left\|L_0D^\beta K(t)w_0\right\|_{L^2}\leq C \min\{1, t^{-(\frac{n}{8}+\frac{\beta}{2})}\}\|L_0w_0\|_{L^4}+C \min\{1, t^{-(\frac{n}{8}+\frac{\beta}{2}+\frac{1}{2})}\}\|L_-w_0\|_{L^4}\nonumber\\
\left\|L_-D^\beta K(t)w_0\right\|_{L^2}\leq C \min\{1, t^{-(\frac{n}{8}+\frac{\beta}{2}+\frac{1}{2})}\}\|L_0w_0\|_{L^4}+C \min\{1, t^{-(\frac{n}{8}+\frac{\beta}{2}+1)}\}\|L_-w_0\|_{L^4}\label{L2}.
\end{align}
and by the interpolation of Lebesgue functions we get:
\begin{align}
\left\|L_0D^\beta K(t)w_0\right\|_{L^2}\leq C \min\{1, t^{-(\frac{n}{8}+\frac{\beta}{2})}\}\|L_0w^0\|^\frac{1}{2}_{L^2}\|L_0w_0\|_{L^\infty}^{\frac{1}{2}}+C \min\{1, t^{-(\frac{n}{8}+\frac{\beta}{2}+\frac{1}{2})}\}\|L_-w^0\|^\frac{1}{2}_{L^2}\|L_-w_0\|^\frac{1}{2}_{L^\infty}\nonumber\\
\left\|L_-D^\beta K(t)w_0\right\|_{L^2}\leq C \min\{1, t^{-(\frac{n}{8}+\frac{\beta}{2}+\frac{1}{2})}\}\|L_0w_0\|^\frac{1}{2}_{L^2}\|L_0w^0\|_{L^\infty}^{\frac{1}{2}}+C \min\{1, t^{-(\frac{n}{8}+\frac{\beta}{2}+1)}\}\|L_-w^0\|^\frac{1}{2}_{L^2}\|L_-w_0\|^\frac{1}{2}_{L^\infty}\label{L2}.
\end{align}

We can estimate also the decay in $L^\infty$. We have that :
\begin{align*}
\left\|L_0D^\beta K(t)w_0\right\|_{L^\infty}&\leq C \int_0^\infty\int_{S^{n-1}} e^{-c|\xi|^2t} (|L_0\hat{w_0}|+|\xi||L_-\hat{w^0}|)|\xi|^{\beta} |\xi|^{n-1} d\varsigma d|\xi|\\
&\leq C\left(\int_0^\infty\int_{S^{n-1}} e^{-2c|\xi|^2t} |\xi|^{(n-1)}|\xi|^{2\beta}d\varsigma d\xi\right)^{\frac{1}{2}} \left( \int_0^\infty\int_{S^{n-1}}|L_0\hat{w_0}|^2 |\xi|^{(n-1)} d\varsigma d\xi\right)^{\frac{1}{2}}\\
+& C\left(\int_0^\infty\int_{S^{n-1}} e^{-2c|\xi|^2t}|\xi|^2 |\xi|^{(n-1)}|\xi|^{2\beta}d\varsigma d\xi\right)^{\frac{1}{2}}\left( \int_0^\infty\int_{S^{n-1}}|L_-\hat{w_0}|^2 |\xi|^{(n-1)}d\varsigma d\xi\right)^{\frac{1}{2}}\\
&\leq C \min\{1, t^{-(\frac{n}{4}+\frac{\beta}{2})}\}\|L_0w_0\|_{L^2}+C \min\{1, t^{-(\frac{n}{4}+\frac{\beta}{2}+\frac{1}{2})}\}\|L_-w_0\|_{L^2},
\end{align*}
and
\begin{align*}
\left\|L_-D^\beta K(t)w_0\right\|_{L^\infty}&\leq C \int_0^\infty\int_{S^{n-1}} e^{-c|\xi|^2t} (|\xi||L_0\hat{w_0}|+|\xi|^2|L_-\hat{w^0}|)|\xi|^{\beta} |\xi|^{n-1} d\varsigma d|\xi|\\
&\leq C\left(\int_0^\infty\int_{S^{n-1}} e^{-2c|\xi|^2t} |\xi|^2|\xi|^{(n-1)}|\xi|^{2\beta}d\varsigma d\xi\right)^{\frac{1}{2}} \left( \int_0^\infty\int_{S^{n-1}}|L_0\hat{w_0}|^2 |\xi|^{(n-1)} d\varsigma d\xi\right)^{\frac{1}{2}}\\
+& C\left(\int_0^\infty\int_{S^{n-1}} e^{-2c|\xi|^2t}|\xi|^4 |\xi|^{(n-1)}|\xi|^{2\beta}d\varsigma d\xi\right)^{\frac{1}{2}}\left( \int_0^\infty\int_{S^{n-1}}|L_-\hat{w_0}|^2|\xi|^{(n-1)} d\varsigma d\xi\right)^{\frac{1}{2}}\\
&\leq C \min\{1, t^{-(\frac{n}{4}+\frac{\beta}{2}+\frac{1}{2})}\}\|L_0w_0\|_{L^2}+C \min\{1, t^{-(\frac{n}{4}+\frac{\beta}{2}+1)}\}\|L_-w_0\|_{L^2}.
\end{align*}
Then if $\beta$ is a multi index, we have also the ``$K(t)$ estimates'' in $L^\infty$: 
\begin{align}
\left\|L_0D^\beta K(t)w_0\right\|_{L^\infty}\leq C \min\{1, t^{-(\frac{n}{4}+\frac{\beta}{2})}\}\|L_0w_0\|_{L^2}+C \min\{1, t^{-(\frac{n}{4}+\frac{\beta}{2}+\frac{1}{2})}\}\|L_-w_0\|_{L^2}\nonumber\\
\left\|L_-D^\beta K(t)w_0\right\|_{L^\infty}\leq C \min\{1, t^{-(\frac{n}{4}+\frac{\beta}{2}+\frac{1}{2})}\}\|L_0w_0\|_{L^2}+C \min\{1, t^{-(\frac{n}{4}+\frac{\beta}{2}+1)}\}\|L_-w_0\|_{L^2}\label{Linfy}.
\end{align}
\end{proof}

\subsection{Global Existence and Asymptotic Behavior of Smooth Solutions}
Existence of global solutions to system \eqref{sistema_easy2} is given by the following theorem.
\begin{theorem}
There exists an $\epsilon_0>0$ such that, if 
$$
\|u_0\|_{H^s}, \|u_0\|_{L^1}, \|v_0\|_{H^s}, \|v_0\|_{L^1},\|\phi_0\|_{H^{s+1}}, \|\phi_0\|_{L^1} , \bar{u} \leq \epsilon_0,
$$%\|\phi_0\|_{W^{1,\infty}}
then there exists a unique global solution to the Cauchy problem \eqref{sistema_easy2}-\eqref{dati_stato}
\begin{align}
u \in C([0,\infty); H^s(\mathbb{R}^n)),\; v \in C([0,\infty); H^s(\mathbb{R}^n)),\;  \phi \in C([0,\infty); H^{s+1}(\mathbb{R}^n)), \quad \textrm{for }  s\geq \left[\frac{n}{2}\right]+1.
\end{align}
Moreover, for the solution $(u,v,\phi)$ the following decay rates  are satisfied 
\begin{equation}
\begin{array}{llll}
 \|u(t)\|_{L^\infty}\sim t^{-\delta},& \|u(t)\|_{H^s}\sim t^{-\delta}, \\ \\
\|v(t)\|_{L^\infty}\sim t^{-\delta}, & \|v(t)\|_{H^s}\sim t^{-\delta}, \\ \\
 \|\phi(t)\|_{L^\infty}\sim t^{-\delta}, &  \|D_x^1 \phi(t)\|_{L^\infty}\sim t^{-\delta}, & & \\\\
 \|\phi(t)\|_{H^{s+1}}\sim t^{-\delta}, & \|\phi(t)\|_{H^s}\sim t^{-\delta}.\\
\end{array}
\end{equation}
where $\delta=\min\left\{\frac{n}{4},\frac{n}{8}+1\right\}$.

\end{theorem}

\begin{proof}

Let us notice that for the solution to the linear parabolic equation the estimates of the previous case still hold. 
We can collect the estimate referred to the function $\phi$ in the following proposition.
\begin{proposition}\label{phi_estimates2}
Let $(u,v,\phi)$ be the solution of system \eqref{sistema_easy2}-\eqref{dati_stato}.  Then for $t\in (0,T)$,  
\begin{equation}\label{Prop_Nphix}
\begin{array}{lcl}
N_{D_x^1 \phi}^{\delta}(t)&\leq& C\left(\|D_x\phi_0\|_{L^\infty}+N_u^{\delta}(t)\right),\\\nonumber\\
M_{D_x^{s+1}\phi}^{\delta}(t)&\leq& C\left(\|D_x^{s+1}\phi_0\|_{L^2}+M_{D_x^s u}^{\delta}(t)\right),\\\\
N_\phi^{\delta}(t)&\leq& C\left(\|\phi_0\|_{L^\infty}+N_u^{\delta}(t)\right),\\\\
M_\phi^{\delta}(t)&\leq& C\left(\|\phi_0\|_{L^2}+M_{u}^{\delta}(t)\right),\\\nonumber\\
M_{D_x^{s}\phi}^{\delta}(t)&\leq& C\left(\|D_x^{s}\phi_0\|_{L^2}+ M_{D_x^s u}^{\delta}(t)\right).\nonumber
\end{array}
\end{equation}
where $\delta=\min\left\{\frac{n}{4},\frac{n}{8}+1\right\}$.
\end{proposition}

\subsubsection{Decay Estimates for the Conservative and Dissipative Variables}
As before we proceed by estimating the norm of the conservative and dissipative variables of the hyperbolic part, starting from the function $u$.
\paragraph{$L^2$-estimate for $u$\\}
By the Duhamel's formula we can write this solution as
\begin{eqnarray}\label{comp_u2}
u(x,t)=(\Gamma^h_{1}(t)\ast w_0)(x)+ \int_0^t \Gamma^h_{1}(t-s)\ast [0, (u+\bar{u})\nabla\phi (s)]ds,
\end{eqnarray}
where $\Gamma^h_{1}$ is the first row of the $(n+1)\times (n+1)$ matrix $\Gamma^h$.

Then
\begin{equation}\label{uL2b}
\begin{array}{l}
\|u(t)\|_{L^2}\leq \|\Gamma^h_{1}(t)\ast w_0\|_{L^2}+\displaystyle\int_0^t \|\Gamma^h_{1}(t-s)\ast [0, (u+\bar{u})\nabla\phi (s)]\|_{L^2} ds,
\end{array}
\end{equation}
The first term can be estimated as in the previous section, while the integral term can be decomposed as
\begin{eqnarray*}
\int_0^t \|\Gamma^h_{1}(t-s)\ast [0, (u(s)+\bar{u})\nabla\phi (s)]\|_{L^2}ds&\leq&\int_0^t \|\mathcal{K}_{1}(t-s) [0, (u(s)+\bar{u})\nabla\phi (s)]\|_{L^2}ds\\
&+&\int_0^t \|K_{1}(t-s) [0, (u(s)+\bar{u})\nabla\phi (s)]\|_{L^2}ds.
\end{eqnarray*}
Let us start with the first integral of the previous inequality.
\begin{eqnarray*}
\int_0^t \|\mathcal{K}_{1}(t-s) ([0, (u(s)+\bar{u})\nabla\phi (s)])\|_{L^2}&\leq& \int_0^t ce^{-c(t-s)}\| (u(s)+\bar{u})\nabla\phi (s)]\|_{L^2} ds\\
&\leq& \int_0^t ce^{-c(t-s)}(\|\nabla\phi(s)\|_{L^2}\|u(s)\|_{L^\infty}+\bar{u}\|\nabla\phi(s)\|_{L^2}ds,
\end{eqnarray*}
then
\begin{eqnarray*}
\int_0^t\|\mathcal{K}_{1}(t-s) ([0, (u+\bar{u})\nabla\phi (s)])\|_{L^2}&\leq& CM_{D_x^1 \phi}^{\delta}(t)N_{u}^{\delta}(t)\int_0^t e^{-c(t-s)}\min\{1,s^{-2\delta}\}ds \\
&+&C\bar{u}M_{D_x^1 \phi}^{\delta}(t)\int_0^t e^{-c(t-s)} \min\{1,s^{-\delta}\}ds,
\end{eqnarray*}
where $\delta=\min\left\{\frac{n}{4},\frac{n}{8}+1\right\}$. \\
By Lemma \ref{Lemma5.2} we deduce that
\begin{equation}\label{B}
\int_0^t \|\mathcal{K}_{1}(t-s) ([0, (u(s)+\bar{u})\nabla\phi (s)])ds\|_{L^2}\leq C\left[\min\{1,t^{-2\delta}\}M_{D_x^1 \phi}^{\delta}(t)N_{u}^{\delta}(t)+\min\{1,t^{-\delta}\}\bar{u}M_{D_x^1 \phi}^{\delta}(t)\right].
\end{equation}
To complete our estimate we need to study the dissipative part. Due to the presence of the linear term $\bar{u}D_x^1 \phi$, we do not have enough polynomial decay. In order to overcome this difficulty we apply the derivative of the linear term to the Green function, getting a faster decay.\\  
Thanks to this modification, we are able to estimate this term as follows using Theorem \ref{Theo_BHN2}:
\begin{align*}
\int_0^t\|&K_{1}(t-s) ([0, (u(s)+\bar{u})\nabla\phi (s)])\|_{L^2}ds\leq\int_0^t\sum_{i=1}^n\|K_{1,i+1}(t-s) ((u(s)+\bar{u})\partial_{x_i}\phi (s))\|_{L^2}ds\\
\leq& \int_0^t\sum_{i=1}^n\|K_{1,i+1}(t-s) (u\partial_{x_i}\phi (s))\|_{L^2}ds+\int_0^t\sum_{i=1}^n\|D_{x_i}^1K_{1,i+1}(t-s) ( \bar{u}\phi(s))\|_{L^2}ds\\
%\leq& C\int_0^t \min\{1,(t-s)^{-\frac{n}{4}-\frac{1}{2}}\}(\|u(s)\nabla\phi(s)\|_{L^1}ds+C\int_0^t \min\{1,(t-s)^{-\frac{n}{4}-1}\}(\|\bar{u}\phi(s)\|_{L^1}ds\\
\leq& C\int_0^t\min\{1,(t-s)^{-\frac{n}{4}-\frac{1}{2}}\}\|\nabla\phi(s)\|_{L^2}\|u(s)\|_{L^2}ds+C\int_0^t\hspace{-0.2cm} \min\{1,(t-s)^{-\frac{n}{8}-1}\}\bar{u}\|\phi(s)\|_{L^2}^{\frac{1}{2}}\|\phi(s)\|_{L^\infty}^{\frac{1}{2}}ds.
\end{align*}
Then we obtain the estimate
\begin{eqnarray*}
\int_0^t\|K_{1}(t-s)([0, (u+\bar{u})\nabla\phi (s)])\|_{L^2}ds
&\leq& CM_{D_x^1 \phi}^{\delta}(t)M_{u}^{\delta}(t)\int_0^t\min\{1,(t-s)^{-\frac{n}{4}-\frac{1}{2}}\}\min\{1,s^{-2\delta}\}ds\\
&+&C\bar{u}(M_{\phi}^{\delta}(t))^{\frac{1}{2}}(N_{\phi}^{\delta}(t))^{\frac{1}{2}}\int_0^t\min\{1,(t-s)^{-\frac{n}{8}-1}\}\min\{1,s^{-\delta}\}ds\\
&\leq& C\left(\min\{1,t^{-\nu}\}M_{D_x^1 \phi}^{\delta}(t)M_{u}^{\delta}(t)+\bar{u}\min\{1,t^{-\delta}\}(M_{\phi}^{\delta}(t))^{\frac{1}{2}}(N_{\phi}^{\delta}(t))^{\frac{1}{2}}\right).
\end{eqnarray*}

where $\nu=\frac{1}{4}$ if $n=1$, otherwise $\nu=\min\{\frac{n}{4}+\frac{1}{2},2\delta\}$.
Summing the last inequality and %\eqref{A}, 
\eqref{B} we obtain
\begin{eqnarray}\label{CUL2}
\|u(t)\|_{L^2}&\leq& C\left(e^{-ct}(\|u_0\|_{L^2}+\sum\limits_{i=1}^n\|v_0^i\|_{L^2}+\min\{1,t^{-\frac{n}{4}}\}\|u_0\|_{L^1}+\min\{1,t^{-\frac{n}{4}-\frac{1}{2}}\}\sum\limits_{i=1}^n\|v_0^i\|_{L^1}\right.\nonumber\\
&+&\min\{1,t^{-2\delta}\}(M_{D_x^1 \phi}^{\delta}(t)N_{u}^{\delta}(t)+\min\{1,t^{-\delta}\}\bar{u}(M_{\phi}^{\delta}(t))\nonumber\\
&+&\left. \min\{1,t^{-\nu}\}M_{D_x^1 \phi}^{\delta}(t)M_{u}^{\delta}(t)+\bar{u}\min\{1,t^{-\delta}\}(M_{\phi}^{\delta}(t))^{\frac{1}{2}}(N_{\phi}^{\delta}(t))^{\frac{1}{2}}\right),
\end{eqnarray}
where $\delta=\min\left\{\frac{n}{4},\frac{n}{8}+1\right\}$.

%%%%%%%%%%%%%%%%%%%%%%%%%%%%%%%%%%%%%%%%%%%%%%%%%%%%%%%%%%%%%%%%%%%%%%%%%%%%%%%%%%%%%%%%%%%%%%%%
%%%%%%%%%%%%%%%%%%%%%%%%%%%%%%%%%%%%%%%%%%%%%%%%%%%%%%%%%%%%%%%%%%%%%%%%%%%%%%%%%%%%%%%%%%%%%%%%%%%%%%%%%%%%%%%%%%%%%%%%%%%%%%%%%%%%%%%%%%%%%%%%%%%%%%%%%%%%%%%%%%%%%%%%%%%%%%

\paragraph{$L^2$- estimate for $D_x^s u$\\}
In a similar way it is possible obtain the $s$-order estimate for the conservative variable. From the Duhamel's formula we know that
\begin{equation}\label{DuL2b}
\begin{array}{l}
\|D_x^su(t)\|_{L^2}\leq \|D_x^s\Gamma^h_{1}(t)\ast w_0\|_{L^2}+\displaystyle\int_0^t \|D_x^s\Gamma^h_{1}(t-s)\ast ([0, (u(s)+\bar{u})\nabla\phi (s)])\|_{L^2} ds.
\end{array}
\end{equation}

Let us focus now on the integral term that we can decompose as
\begin{eqnarray*}
\int_0^t \|D_x^s\Gamma^h_{1}(t-s)\ast ([0, (u+\bar{u})\nabla\phi (s)])\|_{L^2}ds&\leq& \int_0^t \|D_x^s \mathcal{K}_{1}(t-s) ([0, (u(s)+\bar{u})\nabla\phi (s)])\|_{L^2}ds\\
&+&\int_0^t \|D_x^sK_{1}(t-s) ([0, (u(s)+\bar{u})\nabla\phi (s)])\|_{L^2}ds.
\end{eqnarray*}
We estimate the first integral as 
\begin{align*}
\int_0^t \|&D_x^s \mathcal{K}_{1}(t-s) ([0, (u(s)+\bar{u})\nabla\phi (s)])\|_{L^2}\leq \int_0^t Ce^{-c(t-s)}\| D_x^s[(u(s)+\bar{u})\nabla\phi(s)]\|_{L^2}\\
%\leq& \int_0^t Ce^{-c(t-s)}(\| \bar{u}D_x^s (\nabla\phi(s))\|_{L^2}+\|D_x^s (u \nabla\phi(s))\|_{L^2})\\
\leq& \int_0^t
Ce^{-c(t-s)}(\bar{u}\|D_x^{s+1}\phi(s)\|_{L^2}+\|u(s)\|_{L^\infty}\|D_x^{s+1}\phi(s)\|_{L^2}+\|\nabla\phi(s)\|_{L^\infty}\|D_x^{s}u(s)\|_{L^2}) ds.
\end{align*}

Thanks to Lemma \ref{Lemma5.2} we deduce that:
\begin{eqnarray*}
\int_0^t \|D_x^s \mathcal{K}_{1}(t-s) ([0, (u+\bar{u})\nabla\phi (s)])\|_{L^2}&\leq& C\bar{u}\min\{1,t^{-\delta}\}M_{D_x^{s+1} \phi}^{\delta}(t)\\
&+&C\min\{1,t^{-2\delta}\}(N_{u}^{\delta}(t)M_{D_x^{s+1}\phi}^{\delta}(t)+N_{D_x^1 \phi}^{\delta}(t)M_{D_x^{s}u}^{\delta}(t)).
\end{eqnarray*}
To complete our estimate, we need to study the dissipative part,
\begin{align*}
\int_0^t\|&D_x^s K_{1}(t-s) ([0, (u+\bar{u})\nabla\phi (s)])\|_{L^2}ds\leq\int_0^t\sum_{i=1}^n \| D_x^s K_{1,i+1}(t-s) ([0, (u+\bar{u})\partial_{x_i}\phi (s)])\|_{L^2}ds\\
\leq & \int_0^t\sum_{i=1}^n \| D_x^{s}\mathcal{K}_{1,i+1}(t-s) u(s)\partial_{x_i}\phi (s)\|_{L^2}ds+\int_0^t\sum_{i=1}^n \|D_x^{s+1}K_{1,i+1}(t-s) \bar{u}\phi (s)\|_{L^2}ds. 
\end{align*}
We proceed as done before and by Lemma \ref{Lemma5.2} we obtain
\begin{align*}
\int_0^t \|D_x^{s}K_{1}(t-s) ([0, (u(s)+\bar{u})\nabla\phi (s)])\|_{L^2}\leq& C\left[ \min\{1,t^{-\kappa_s}\}(M_{D_x^1 \phi}^{\delta}(t)M_{u}^{\delta}(t)+M_{u}^{\delta}(t)M_{D_x^1 \phi}^{\delta}(t))\right.\\
+& \left.\bar{u}\min\{1,t^{-\delta}\}(M_{\phi}^{\delta}(t))^{\frac{1}{2}}(N_{\phi}^{\delta}(t))^{\frac{1}{2}}\right].
\end{align*}
where $\kappa_s=\min\left\{\frac{n}{4}+\frac{1}{2}+\frac{s}{2},2\delta\right\}$ and $\delta=\min\left\{\frac{n}{4},\frac{n}{8}+1\right\}$.

Then we can write the estimate in the $L^2$-norm of the $s-$derivative of the conservative variable $u$ as
\begin{eqnarray}\label{CDUL2}
\|D_x^s u(t)\|_{L^2}&\leq& C\left[e^{-ct}(\|D_x^su_0\|_{L^2}+\|D_x^sv_0\|_{L^2}+\min\{1,t^{-\frac{n}{4}-\frac{s}{2}}\}\|u_0\|_{L^1}
+\min\{1,t^{-\frac{n}{4}-\frac{1}{2}-\frac{s}{2}}\}\|v_0\|_{L^1}))\right.\nonumber\\
&+&\bar{u}\min\{1,t^{-\delta}\}(M_{D_x^{s+1} \phi}^{\delta}(t)+\min\{1,t^{-2\delta}\}(N_{u}^{\delta}(t)M_{D_x^{s+1} \phi}^{\delta}(t)+N_{D_x^1 \phi}^{\delta}(t)M_{D_x^{s}u}^{\delta}(t))\nonumber\\
&+&\left.\min\{1,t^{-\kappa_s}\}(M_{D_x^1 \phi}^{\delta}(t)M_{u}^{\delta}(t)+M_{u}^{\delta}(t)M_{D_x^1 \phi}^{\delta}(t))+\bar{u}t^{-\delta}(M_{\phi}^{\delta}(t))^{\frac{1}{2}}(N_{\phi}^{\delta}(t))^{\frac{1}{2}}\right]\nonumber.
\end{eqnarray}

%%%%%%%%%%%%%%%%%%%%%%%%%%%%%%%%%%%%%%%%%%%%%%%%%%%%%%%%%%%%%%%%%%%%%%%%%%%%%%%%%%%%%%%%%%%%%%%%%%%%%%%%%%%%%%%%%%%%%%%%%%%%%%%%%%%%%%%%%%%%%%%%%%%%%%%%%%%%%%%%%%%%%%%%%%%%%%%%%%%%%%%%%%%%%%L^\infty estimate for the hyperbolic variables

\paragraph{$L^\infty$-estimate for $u$\\}
Finally with the same approach, we estimate the $L^\infty$-norm of the function $u$. By the Duhamel's formula we know that
\begin{equation}\label{uLinfb}
\begin{array}{l}
\|u(t)\|_{L^\infty}\leq \|\Gamma^h_{1}(t)\ast w_0\|_{L^\infty}+\displaystyle\int_0^t \|\Gamma^h_{1}(t-s)\ast ([0, (u(s)+\bar{u})\nabla\phi (s)])\|_{L^\infty} ds.
\end{array}
\end{equation}

We can decompose the integral term as,
\begin{eqnarray*}
\int_0^t \|\Gamma^h_{1}(t-s)\ast ([0, (u(s)+\bar{u})\nabla\phi (s)])\|_{L^\infty}ds&\leq& \int_0^t \sum_{i=1}^n \|\mathcal{K}_{1,i+1}(t-s) ( (u+\bar{u})\partial_{x_i}\phi (s))\|_{L^\infty}ds\\
&+&\int_0^t \sum_{i=1}^n \|K_{1,i+1}(t-s)  ( (u(s)+\bar{u})\partial_{x_i}\phi (s))\|_{L^\infty}ds.
\end{eqnarray*}
Let us estimate the first term in the previous inequality,
\begin{align*}
\int_0^t \sum_{i=1}^n \|\mathcal{K}_{1,i+1}(t-s)  &( (u(s)+\bar{u})\partial_{x_i}\phi (s))\|_{L^\infty}ds\leq \int_0^t \sum_{i=1}^n C\|\mathcal{K}_{1,i+1}(t-s) ( (u+\bar{u})\partial_{x_i}\phi (s))\|_{L^2}ds\\
+&C\int_0^t \sum_{|\alpha|=s}\|D_x^s \sum_{i=1}^n \mathcal{K}_{1,i+1}(t-s)((u+\bar{u})\partial_{x_i}\phi (s))\|_{L^2}ds.
\end{align*}
Then, by the estimates of the function $u$ and its derivatives in the $L^2$-norm, we have 
\begin{align*}
\int_0^t &\sum_{i=1}^n\| \mathcal{K}_{1,i+1}(t-s)( (u+\bar{u})\partial_{x_i}\phi (s))\|_{L^\infty}ds\\
\leq& C\left(\min\{1,t^{-2\delta}\}(M_{D_x^1 \phi}^{\delta}(t)N_{u}^{\delta}(t)+\min\{1,t^{-\delta}\}\bar{u}M_{D_x^1 \phi}^{\delta}(t)\right.\\
+&\left.\bar{u}\min\{1,t^{-\delta}\}(M_{D_x^{s+1} \phi}^{\delta}(t)+\min\{1,t^{-2\delta}\}(N_{u}^{\delta}(t)M_{D_x^{s+1} \phi}^{\delta}(t)+N_{D_x^1 \phi}^{\delta}(t)M_{D_x^{s}u}^{\delta}(t))\right).
\end{align*}
As the final step we need to estimate the dissipative part: 
\begin{align*}
\int_0^t\| K_{1}(t-s) ([0, (u+\bar{u})\nabla\phi (s)])\|_{L^\infty}ds\leq&% \int_0^tC \min\{1,(t-s)^{-\frac{n}{2}-\frac{1}{2}}\}\|(u+\bar{u})\nabla\phi (s)])\|_{L^1}ds\\
%\leq& 
\int_0^tC\min\{1,(t-s)^{-\frac{n}{2}-\frac{1}{2}}\}\|\nabla\phi(s)\|_{L^2}\|u(s)\|_{L^2}ds\\
+&\int_0^tC\min\{1,(t-s)^{-\frac{n}{4}-1}\}\bar{u}\|\phi(s)\|_{L^2}ds\\
\leq& C\left[\min\{1,t^{-2 \delta}\}M_{D_x^1 \phi}^{\delta}(t)M_{u}^{\delta}(t)+\bar{u}t^{-\delta}M_{\phi}^{\delta}\right].
\end{align*}
Thus we can estimate the $L^\infty$-norm of the function $u$ as follows.
%Summing all these terms yields
\begin{eqnarray}\label{CULinf}
\|u(t)\|_{L^\infty}&\leq& C\left[ e^{-ct}(\|u_0\|_{H^s}+\sum_{i=1}^n\|v_0^i\|_{H^s})+\min\{1,t^{-\frac{n}{2}}\}\|u_0\|_{L^1}+\min\{1,t^{-\frac{n}{2}-\frac{1}{2}}\}\sum_{i=1}^n\|v_0^i\|_{L^1}\right.\nonumber\\
&+&\min\{1,t^{-2\delta}\}(M_{D_x^1 \phi}^{\delta}(t)N_{u}^{\delta}(t)+\min\{1,t^{-\delta}\}\bar{u}M_{D_x^1 \phi}^{\delta}(t)\nonumber\\
&+&\bar{u}\min\{1,t^{-\delta}\}M_{D_x^{s+1} \phi}^{\delta}(t)+\min\{1,t^{-2\delta}\}(N_{u}^{\delta}(t)M_{D_x^{s+1} \phi}^{\delta}(t)+N_{D_x^1 \phi}^{\delta}(t)M_{D_x^{s}u}^{\delta}(t))\nonumber\\
&+&\left.\min\{1,t^{-2\delta}\}M_{D_x^1 \phi}^{\delta}(t)M_{u}^{\delta}(t)+\bar{u}t^{-\delta}M_{\phi}^{\delta}\right].
\end{eqnarray}

%%%%%%%%%%%%%%%%%%%%%%%%%%%%%%%%%%%%%%%%%%%%%%%%%%%%%%%%%%%%%%%%%%%%%%%%%%%%%%%%%%%%%%%%%%%%%%%%%%%%%%%%%%%%%%%%%%%%%%%%%%%%%%%%%%%%%%%%%%%%%%%%%%%%%%%%%%%%%%%%%%%%%%%%%%%%%%%%%%%%%%%%%%%%%%%DISSSIPATIVE VARIABLES

%\subsection{Decay estimates for the dissipative variable}
Next paragraphs are devoted to the estimates of the $L^2$ and $L^\infty$-norms of the function $v$.
\paragraph{$L^2$-estimate for v\\}
By the Duhamel's formula we can write the generic component $v_j$, with $j=1,\ldots,n$, as
\begin{eqnarray*}\label{comp_v4}
v_j(x,t)=& (\Gamma^h_{j+1}(t)\ast w_0)(x)+ \displaystyle\int_0^t \Gamma^h_{j+1}(t-s)\ast ([0, (u(s)+\bar{u})\nabla\phi (s)])ds,
\end{eqnarray*}
then
\begin{equation}\label{vL2_stato}
\begin{array}{l}
\|v_j(t)\|_{L^2}\leq \|\Gamma^h_{j+1}(t)\ast w_0\|_{L^2}+ \displaystyle\int_0^t \|\Gamma^h_{j+1}(t-s)\ast ([0, (u+\bar{u})\nabla\phi (s)])\|_{L^2}ds.
\end{array}
\end{equation}
By the decomposition of the Green kernel and thanks to Theorem \ref{Theo_BHN} we deduce 
\begin{equation}
\begin{array}{lll}
\|\mathcal{K}_{j+1,1}(t) u_0\|_{L^2}\leq Ce^{-ct}\|u_0\|_{L^2}, &\quad& \|K_{j+1,1}(t) u_0\|_{L^2}\leq C\min\{1,t^{-\frac{n}{4}-\frac{1}{2}}\}\|u_0\|_{L^1},\\\\
\|\mathcal{K}_{j+1,i+1}(t) v_0^i\|_{L^2}\leq Ce^{-ct}\|v_0^i\|_{L^2},&\quad& \|K_{j+1,i+1}(t)v_0^i\|_{L^2}\leq C\min\{1,t^{-\frac{n}{4}-1}\}\|v_0^i\|_{L^1} \quad \textrm{for }i=1,\ldots,n.
\end{array}\label{Ap}
\end{equation}
As done before we can decompose the integral term in \eqref{vL2_stato} as
\begin{eqnarray*}
\int_0^t \|\Gamma^h_{j+1}(t-s)\ast ([0, (u(s)+\bar{u})\nabla\phi (s)])\|_{L^2}ds &\leq&\int_0^t \|\mathcal{K}_{j+1}(t-s) ([0, (u+\bar{u})\nabla\phi (s)])\|_{L^2}ds\\
&+&\int_0^t \|K_{j+1}(t-s) ([0, (u+\bar{u})\nabla\phi (s)])\|_{L^2}ds.
\end{eqnarray*}
Let us start estimating the first integral 
\begin{eqnarray*}
\int_0^t \|\mathcal{K}_{j+1}(t-s) ([0, (u(s)+\bar{u})\nabla\phi (s)])\|_{L^2}ds&\leq&\int_0^t Ce^{-c(t-s)}\| (u(s)+\bar{u})\nabla\phi (s)\|_{L^2}ds\\
&\leq& \int_0^t Ce^{-c(t-s)}(\|\nabla\phi(s)\|_{L^2}\|u(s)\|_{L^\infty}+\bar{u}\|\nabla\phi(s)\|_{L^2}ds\\
&+& CM_{D_x^1 \phi}^{\delta}(t)N_{u}^{\delta}(t)\int_0^te^{-c(t-s)}\min\{1,s^{-2\delta}\}ds\\
&+&C\bar{u}M_{D_x^1 \phi}^{\delta}(t)\int_0^t e^{-c(t-s)}\min\{1,s^{-\delta}\}ds.
\end{eqnarray*}

Thanks to Lemma \ref{Lemma5.2} we obtain:
\begin{equation}\label{Bp}
\int_0^t \|\mathcal{K}_{j+1}(t-s) ([0, (u+\bar{u})\nabla\phi (s)])\|_{L^2}\leq C\min\{1,t^{-2\delta}\}(M_{D_x^1 \phi}^{\delta}(t)N_{u}^{\delta}(t)+\bar{u}\min\{1,t^{-\delta}\}M_{D_x^1 \phi}^{\delta}(t).
\end{equation}
In order to complete our estimate we need to study the dissipative part, then
\begin{align*}
\int_0^t\| K_{j+1}(t-s) ([0, (u(s)+\bar{u})\nabla\phi (s)])\|_{L^2}ds\leq& \int_0^tC\min\{1,(t-s)^{-\frac{n}{4}-1}\}\|\nabla\phi(s)\|_{L^2}\|u(s)\|_{L^2}ds\\
+&\int_0^tC\min\{1,(t-s)^{-\frac{n}{8}-1}\} \bar{u}\|\phi(s)\|_{L^2}^{\frac{1}{2}}\|\phi(s)\|_{L^\infty}^{\frac{1}{2}}ds\\
\leq& C(\min\{1,t^{-\nu}\}M_{D_x^1 \phi}^{\delta}(t)M_{u}^{\delta}(t)+\bar{u}\min\{1,t^{-\delta}\}(M_{\phi}^{\delta}(t))^\frac{1}{2}(N_{\phi}^{\delta}(t))^\frac{1}{2}.
\end{align*}
where $\nu=\min\left\{\frac{n}{4}+1,2\delta\right\}$.
Finally if we sum the last inequality and \eqref{Ap}, \eqref{Bp} we get the $L^2$-norm of the function $v$
\begin{eqnarray}\label{CVL2}
\|v(t)\|_{L^2}&\leq& C\left[e^{-ct}(\|u_0\|_{L^2}+\sum_{i=1}^n\|v_0^i\|_{L^2})+\min\{1,t^{-\frac{n}{4}-\frac{1}{2}}\}\|u_0\|_{L^1}+\min\{1,t^{-\frac{n}{4}-1}\}\sum_{i=1}^n\|v_0^i\|_{L^1}\right.\nonumber\\
&+&\min\{1,t^{-2\delta}\}(M_{D_x^1 \phi}^{\delta}(t)N_{u}^{\delta}(t)+\min\{1,t^{-\delta}\} \bar{u}M_{D_x^1 \phi}^{\delta}(t)+\min\{1,t^{-\nu}M_{D_x^1 \phi}^{\delta}(t)M_{u}^{\delta}(t)\nonumber\\
&+& \left.\bar{u}\min\{1,t^{-\delta}\}(M_{\phi}^{\delta}(t))^\frac{1}{2}(N_{\phi}^{\delta}(t))^\frac{1}{2}\right].
\end{eqnarray}

%%%%%%%%%%%%%%%%%%%%%%%%%%%%%%%%%%%%%%%%%%%%%%%%%%%%%%%%%%%%%%%%%%%%%%%%%%%%%%%%%%%%%%%%%%%%%%%%
%%%%%%%%%%%%%%%%%%%%%%%%%%%%%%%%%%%%%%%%%%%%%%%%%%%%%%%%%%%%%%%%%%%%%%%%%%%%%%%%%%%%%%%%%%%%%%%%%%%%%%%%%%%%%%%%%%%%%%%%%%%%%%%%%%%%%%%%%%%%%%%%%%%%%%%%%%%%%%%%%%%%%%%%%%%%%%

\paragraph{$L^2$-estimate for $D_x^s v$\\}
Proceeding along the lines of the conservative variable estimates, we get the estimate of the $s-$derivative of $v$ in $L^2$,
\begin{eqnarray}\label{CDVL2}
\|D_x^s v_j(t)\|_{L^2}&\leq& C\left[e^{-ct}(\|D_x^su_0\|_{L^2}+\sum_{i=1}^n\|D_x^sv_0^i\|_{L^2}+\min\{1,t^{-\frac{n}{4}-\frac{1}{2}-\frac{s}{2}}\}\|u_0\|_{L^1}\right.\nonumber\\
&+&\min\{1,t^{-\frac{n}{4}-1-\frac{s}{2}}\}\sum_{i=1}^n\|v_0^i\|_{L^1})+\bar{u}\min\{1,t^{-\delta}\}(M_{D_x^{s+1} \phi}^{\delta}(t)\nonumber\\
&+&\min\{1,t^{-2\delta}\}(N_{u}^{\delta}(t)M_{D_x^{s+1} \phi}^{\delta}(t)+N_{D_x^1 \phi}^{\delta}(t)M_{D_x^{s}u}^{\delta}(t))\\
&+&\min\{1,t^{-\nu_s}\}(M_{D_x^1 \phi}^{\delta}(t)M_{u}^{\frac{n}{4}}(t)+M_{u}^{\delta}(t)M_{D_x^1\phi}^{\delta}(t))\nonumber\\
&+&\left.\bar{u}t^{-\delta}(M_{\phi}^{\delta}(t))^{\frac{1}{2}}(N_{\phi}^{\delta}(t))^{\frac{1}{2}}\right]\nonumber,
\end{eqnarray}
where $\nu_s=\min\left\{\frac{n}{4}+1+\frac{s}{2},2\delta\right\}$.

%%%%%%%%%%%%%%%%%%%%%%%%%%%%%%%%%%%%%%%%%%%%%%%%%%%%%%%%%%%%%%%%%%%%%%%%%%%%%%%%%%%%%%%%%%%%%%%%%%%%%%%%%%%%%%%%%%%%%%%%%%%%%%%%%%%%%%%%%%%%%%%%%%%%%%%%%%%%%%%%%%%%%%%%%%%%%%%%%%%%%%%%%%%%%%%%%%%Dissipative variable L^\infty estimate

\paragraph{$L^\infty$ estimates for $v$\\}
In a similar way we obtain the estimate of the $L^\infty$ norm of $v_j$, 
\begin{eqnarray}\label{CVLinf}
\|v_j(t)\|_{L^\infty}&\leq& C\left[e^{-ct}(\|u_0\|_{H^s}+\sum_{i=1}^n\|v_0^i\|_{H^s})+\min\{1,t^{-\frac{n}{2}-\frac{1}{2}}\}\|u_0\|_{L^1}+\min\{1,t^{-\frac{n}{2}-1}\}\sum_{i=1}^n\|v_0^n\|_{L^1}\right.\nonumber\\
&+&\min\{1,t^{-2\delta}\}M_{D_x^1 \phi}^{\delta}(t)N_{u}^{\delta}(t)+\bar{u}\min\{1,t^{-\delta}\}M_{D_x^1 \phi}^{\delta}(t)+\min\{1,t^{-2\delta}\}M_{D_x^1 \phi}^{\delta}(t)M_{u}^{\delta}(t)\nonumber\\
&+&\left.\bar{u}\min\{1,t^{-\delta}\}M_{\phi}^{\delta}(t)\displaystyle\right].
%&+&\left.\min\{1,t^{-\frac{n}{4}+\frac{n}{4}}\}(N_{u}^{\frac{n}{4}}(t)M_{D_x^{s+1}\phi}^{\frac{n}{4}}(t)+N_{D_x^1 \phi}^{\frac{n}{4}}(t)M_{D_x^{s}u}^{\frac{n}{4}}(t)%)\right]\nonumber
\end{eqnarray}

\subsubsection{Decay rates of variables}
%We are interested in to determinate the decay rates of the different variables involved in system \eqref{sistema_H}.
Thanks to Proposition \eqref{phi_estimates} and  inequalities in  \eqref{CUL2}, \eqref{CDUL2}, \eqref{CULinf}, \eqref{CVL2}, \eqref{CDVL2}, \eqref{CVLinf}, we obtain, for $t>\epsilon>0$ the following estimates for functionals:

\begin{equation*}
\begin{array}{lll}\label{MU2}
M_{u}^{\delta}(t)&\leq&  C\bar{u}(E_0+ D_0)+ C_1 \left[\bar{u}D_0\left((M_{u}^{\delta}(t))^{\frac{1}{2}}+(N_u^{\delta}(t))^{\frac{1}{2}}\right)+\bar{u}D_0\left(M_{u}^{\delta}(t)+N_u^{\delta}(t)\right) +(N_u^{\delta}(t))^2\right. \\\\
&+&\left.N_u^{\delta}(t)M_{u}^{\delta}(t)+(M_{u}^{\delta}(t))^2 \right]\\ \\
%\end{array}
%\end{equation}

%%\begin{equation}
%%\begin{array}{lll}\label{MDU2}
M_{D_x^s u}^{\delta}(t)&\leq& C\bar{u}(E_0+ D_0)+ C_2\left[\bar{u}D_0\left((M_{u}^{\delta}(t))^{\frac{1}{2}}+(N_u^{\delta}(t))^{\frac{1}{2}}\right)+\bar{u}D_0\left(M_{u}^{\delta}(t)+N_u^{\delta}(t)+ M_{D_x^s u}^{\delta}(t)\right)\right.+(M_{u}^{\delta}(t))^2\\\\
&+&\left.N_u^{\delta}(t)M_{D_x^s u}^{\delta}(t)+(M_{D_x^s u}^{\delta}(t))^2 \right].\\ \\
%%\end{array}

N_{u}^{\delta}(t)&\leq& C\bar{u}(E_0+ D_0)+ C_3\left[\bar{u}D_0 \left(M_{u}^{\delta}(t)+N_u^{\delta}(t)+ M_{D_x^s u}^{\delta}(t)\right)\right. +N_u^{\delta}(t)M_{u}^{\delta}(t)+(M_{u}^{\delta}(t))^2\\\\
&+&\left. N_u^{\delta}(t)M_{D_x^s u}^{\delta}(t)+(M_{D_x^s u}^{\delta}(t))^2 + M_{D_x^s u}^{\delta}(t)N_u^{\delta}(t)+ (N_u^{\delta}(t))^2\right].\\ \\
%\end{array}
%\end{equation}
%%
%%
%\begin{equation}
%\begin{array}{lll}\label{MV2}
M_{v}^{\delta}(t)&\leq&  C\bar{u}(E_0+ D_0)+ C_4 \left[\bar{u}D_0\left((M_{u}^{\delta}(t))^{\frac{1}{2}}+(N_u^{\delta}(t))^{\frac{1}{2}}\right)+\bar{u}D_0\left(M_{u}^{\delta}(t)+N_u^{\delta}(t)\right) +(N_u^{\delta}(t))^2\right.\\\\
&+&\left.N_u^{\delta}(t)M_{u}^{\delta}(t)+(M_{u}^{\delta}(t))^2 \right]. \\\\
\end{array}
\end{equation*}
\begin{equation*}
\begin{array}{lll}\label{MDV2}
M_{D_x^s v}^{\delta}(t)&\leq& C\bar{u}(E_0+ D_0)+ C_{5} \left[\bar{u}D_0\left((M_{u}^{\delta}(t))^{\frac{1}{2}}+(N_u^{\delta}(t))^{\frac{1}{2}}\right)+\bar{u}D_0\left(M_{u}^{\delta}(t)+N_u^{\delta}(t)+ M_{D_x^s u}^{\delta}(t)\right)\right.+(M_{u}^{\delta}(t))^2  \\\\
&+& \left.N_u^{\delta}(t)M_{D_x^s u}^{\delta}(t)+(M_{D_x^s u}^{\delta}(t))^2 \right].\\\\
%%\end{array}
%%\end{equation}
%%\begin{equation}
%%\begin{array}{lll}\label{NV2}
N_{v}^{\delta}(t)&\leq& C\bar{u}(E_0+ D_0)+ C_6\left[\bar{u}D_0 \left(M_{u}^{\delta}(t)+N_u^{\delta}(t)+ M_{D_x^s u}^{\delta}(t)\right)\right. +N_u^{\delta}(t)M_{u}^{\delta}(t)+(M_{u}^{\delta}(t))^2\\\\
&+&\left. N_u^{\delta}(t)M_{D_x^s u}^{\delta}(t)+(M_{D_x^s u}^{\delta}(t))^2 + M_{D_x^s u}^{\delta}(t)N_u^{\delta}(t)+ (N_u^{\delta}(t))^2\right].	\\  

\end{array}
\end{equation*}

%where $\delta_{s}=\min \left\{\delta+\frac{1}{2}+\frac{s}{2},\frac{n}{2}\right\}$. 
Moreover $D_0=\max\{\|\phi_0\|_{H^{s+1}},\|\phi_0\|_{L^1}\}$
$E_0=\max\{\|w_0\|_{H^{s}},\|w_0\|_{L^{1}}\}$, while the constant $C_i=C_i(F_k,K,C_{b'},C_{h'})$ for $i=1,\ldots,6$.\\
Let us proceed as in the previous section setting
 
$$
P(t):=M_{u}^{\delta}(t)+N_u^{\delta}(t)+ M_{D_x^s u}^{\delta}(t)+M_{v}^{\delta}(t)+N_{v}^{\delta}(t)+M_{D_x^sv}^{\delta}(t).
$$
It follows that, if initial data and the constant state are small, we have 
\begin{equation}
CP(t)^2-(1-C_{k0})P(t)+C_1P(t)^{\frac{1}{2}}+C_0 \geq 0,
\end{equation}
where $C_0$, $C_1$ and $C_{k0}$ are positive constants depending on initial data and constant state %and $C_{k0}$ is a positive constant depending on data, constant state and $K$, 
and $C$ is a positive constant depending on estimates of Green function.
For suitably small data, this inequality implies that $M_{u}^{\delta}(t)$, $N_u^{\delta}(t)$, 
$M_{D_x^s u}^{\delta}(t)$, $M_{v}^{\delta}(t)$, $N_{v}^{\delta}(t)$, $M_{D_x^s v}^{\delta}(t)$ remain
bounded. %  as far as $\|u,v,\phi,\nabla\phi\|_{L^\infty}\leq K$.
On the other hand, when $t>1$, this implies that $L^\infty$-norm of solution $(u,v)$ does not increase with $t$. 
Thanks to the Proposition \ref{phi_estimates}
the same holds for $N^{\delta}_{\phi}$ and $N^{\delta}_{D_x^1 \phi}$.
Then by Lemma \ref{boundinf} and the continuation principle we get the global existence of solution.

\end{proof}

\section{Comparison with the Patlak-Keller-Segel Model}
Hyperbolic and parabolic models are expected to have the same behavior for large times. In this section we investigate this aspect by studying the  comparison with the analogous Patlak-Keller-Segel (PKS) model. For the sake of simplicity we consider a simplified version of system \eqref{sistema_GUMA}, namely

\begin{equation}
\label{sistema_S}
\left\{
\begin{array}{l}
\partial_{t} u +\nabla \cdot v = 0, \\\\
\partial_{t} v + \nabla u= -\beta v+h(\phi,\nabla \phi)g(u),\\\\
\partial_{t} \phi =\Delta \phi +f(u,\phi).
\end{array}
\right.
\end{equation}
Thus, assuming  $b(\phi,\nabla\phi)\equiv \beta$ and formally disregarding the term $\partial_t v$ in the second equation of \eqref{sistema_S}, we get 
$v=\frac{1}{\beta}(h(\phi,\nabla \phi)g(u)-\nabla u)$, then the system reduces to the PKS parabolic system:
\begin{equation*}
\label{sistema_KS}
\left\{
\begin{array}{l}
\beta\partial_{t} \tilde{u} - \Delta  \tilde{u}+ \nabla \cdot( h(\tilde{\phi},\nabla \tilde{\phi}),g(\tilde{u})) = 0, \\\\
\partial_{t} \tilde{\phi} =\Delta \tilde{\phi} +f(\tilde{u},\tilde{\phi}),
\end{array}
\right.
\end{equation*} 
where the functions $f,g,h$ satisfy the assumptions $(H_g)$, $(H_f)$, $(H_h)$. Then we are led to consider the system

\begin{equation}
\label{sistema_KS2}
\left\{
\begin{array}{l}
\beta\partial_{t} \tilde{u} - \Delta  \tilde{u}+  \nabla \cdot( h(\tilde{\phi},\nabla \tilde{\phi})g(\tilde{u})) = 0, \\\\
\partial_{t} \tilde{\phi} =\Delta \tilde{\phi} + a\tilde{u} -b\tilde{\phi}+ \bar{f}(\tilde{u},\tilde{\phi}),
\end{array}
\right.
\end{equation}
with initial condition
\begin{equation}\label{dati_cauchy_KS}
\tilde{u}(x,0)=\tilde{u}_0(x), \quad  \tilde{\phi}(x,0)=\tilde{\phi}_0(x).
\end{equation}

It is known that, for small initial data the solution of the above problem decay in time in $L^2$-norm in the same way as the solutions to problem 
\eqref{sistema_S} \cite{KoSu}.

Let us recall that it is possible to give a more precise expansion of the diffusive part $K(x,t)$ of the Green Kernel of the dissipative hyperbolic system. As a matter of fact, in \cite{BiHaNa} it is shown that in the linearized isentropic Euler equations with damping for a generic $n$, $K(x,t)$ can be decomposed as:
\begin{equation}\label{GammaP}
K(x,t)= \left[
\begin{array}{cc}
\Gamma^p & (\nabla\Gamma^p)^T\\
\nabla\Gamma^p & \nabla^2\Gamma^p
\end{array}\right]+ R_1(x,t),
\end{equation}
where $\Gamma^p$ is the heat kernel for $u_t=\Delta u$, and the rest term $R_1(x,t)$ satisfies the bound

\begin{equation*}
R_1(x,t)= \frac{e^{-c|x|^2/t}}{(1+t)^{\frac{n}{2}+\frac{1}{2}}}\left[
\begin{array}{cc}
O(1) & O(1)(1+t)^{-\frac{1}{2}}\\
O(1)(1+t)^{-\frac{1}{2}} & O(1)(1+t)^{-1}
\end{array}\right].
\end{equation*}

Our aim is to show that, under the assumption of small initial data, if
\begin{equation}\label{dati_keller}
u_0(x)=\tilde{u}_0(x), \quad \phi_0(x)=\tilde{\phi}_0(x),
\end{equation}
then $\|u(t)-\tilde{u}(t)\|_{L^2}$, and $\|\phi(t)-\tilde{\phi}(t)\|_{L^2}$, for large $t$, approach zero faster than $\|u(t)\|_{L^2}$, $\|\tilde{u}(t)\|_{L^2}$, $\|\phi(t)\|_{L^2}$ and $\|\tilde{\phi}(t)\|_{L^2}$.

\begin{theorem}\label{confrontoKS}
Let $(u,v,\phi)$ and $(\tilde{u},\tilde{\phi})$ be the global solutions respectively to system \eqref{sistema_S} and \eqref{sistema_KS2} under the assumptions $(H_f)$, $(H_g)$, $(H_h)$ and \eqref{dati_keller}. Then there exist $\epsilon_0,L>0$ such that, if 
$$
\|u_0\|_{H^s}, \|u_0\|_{L^1}, \|v_0\|_{H^s}, \|v_0\|_{L^1},\|\phi_0\|_{H^{s+1}},\|\phi_0\|_{W^{1,\infty}} \leq \epsilon_0
$$
then, for all $t>0$,
\begin{equation*}
\sup_{(0,t)}\left\{\max\{1,s^{\delta} \}\|u(s)-\tilde{u}(s)\|_{L^2}\right\}\leq L, \quad
\sup_{(0,t)}\left\{\max\{1,s^{\delta} \}\|\phi(s)-\tilde{\phi}(s)\|_{L^2}\right\} \leq L,
\end{equation*}
where $\delta=\min\{\frac{n}{4}+\frac{1}{2},\frac{n}{2}\}$.
\end{theorem}

\begin{proof}
Let $K>0$ such that $\|u,v,\phi,\nabla\phi,\tilde{u},\tilde{v},\tilde{\phi},\nabla\tilde{\phi}\|_{L^\infty(\R^n\times(0,\infty))}\leq K$. 
The difference between $u$ and $\tilde{u}$ can be expressed as follows
\begin{eqnarray*}
|u-\tilde{u}|&\leq& |(\Gamma^h_{11}(t)-\Gamma^\bb(t))\ast u_0|+\left|\sum_{i=1}^n\Gamma^h_{1,i+1}(t)\ast v_0^i\right|\\
&+& \left|\int_0^t \Gamma^h_{1}(t-s)\ast (\bar{B}(\phi,\nabla\phi)v(s)+H(\phi,\nabla\phi,u)(s))ds\right. \\
&-&\left.\frac{1}{\beta} \int_0^t \nabla \Gamma^\bb(t-s)\ast(H(\tilde{\phi},\nabla\tilde{\phi},\tilde{u}))ds \right| .
\end{eqnarray*}
By equation \eqref{GammaP}, we have
\begin{eqnarray*}
|u-\tilde{u}|&\leq& \left|(\mathcal{K}_{11}(t)+R_{11}(t)) u_0\right|+\left|\sum_{i=1}^n\Gamma^h_{1,i+1}(t)\ast v_0^i\right|\\
&+&\left|\frac{1}{\beta} \int_0^t  \nabla \Gamma^\bb(t-s)\ast(H(\phi,\nabla\phi,u)-H(\tilde{\phi},\nabla\tilde{\phi},\tilde{u}))ds \right|\\
&+& \left|\int_0^t (\mathcal{K}_{1}(t-s)+R_1(t-s))H(\phi,\nabla\phi,u)ds\right| \\
&+&\left| \int_0^t \Gamma^h_{1}(t-s)\ast \bar{B}(\phi,\nabla\phi)v(s) ds \right| .
\end{eqnarray*}
Proceeding as in the proof of Theorem \ref{Theo_Glo_GUMA}, we are able to estimate $\|u-\tilde{u}\|_{L^2}$ for large $t$:
\begin{eqnarray}\label{diff_uL2}
\|u(t)-\tilde{u}(t)\|_{L^2}&\leq& \|(\mathcal{K}_{11}(t)+R_{11}(t)) u_0\|_{L^2}+\|\sum_{i=1}^n\Gamma^h_{1,i+1}(t)\ast v_0^i\|_{L^2}\nonumber\\
&+&\frac{1}{\beta} \int_0^{t} \|\nabla\Gamma^\bb(t-s)\ast(H(\phi,\nabla\phi,u)(s)-H(\tilde{\phi},\nabla\tilde{\phi},\tilde{u})(s)) \|_{L^2}ds\nonumber\\
&+&\int_0^t \|(\mathcal{K}_{1}(t-s)+R_1(t-s)) H(\phi,\nabla\phi,u)(s)\|_{L^2}ds \\
&+&\int_0^t \|\Gamma^h_{1}(t-s)\ast \bar{B}(\phi,\nabla\phi)v(s) \|_{L^2}ds. \nonumber\\
%&+&\frac{1}{\beta} \int_{t-1}^t \|\nabla \Gamma^\bb(t-s)\ast(H(\phi,\nabla\phi,u)(s)-H(\tilde{\phi},\nabla\tilde{\phi},\tilde{u})(s)) \|_{L^2}ds\nonumber.
\end{eqnarray}
For the first two terms on the right hand side we have,
\begin{eqnarray*}
&&\|(\mathcal{K}_{11}(t)+R_{11}(t)) u_0\|_{L^2}  \leq e^{-ct}\|u_0\|_{L^2}+t^{-(\frac{n}{4}+\frac{1}{2})}\|u_0\|_{L^1}\\\\
&&\|\sum_{i=1}^n \Gamma^h_{1,i+1}(t)\ast v_0^i\|_{L^2}\leq  e^{-ct}\sum_{i=1}^n\|v_0^i\|_{L^2}+\min\{1,t^{-(\frac{n}{4}+\frac{1}{2})}\}\sum_{i=1}^n\|v_0^i\|_{L^1}.
\end{eqnarray*}
Let us now estimate the first integral as
\begin{align*}
\frac{1}{\beta} \int_0^{t} \|&\nabla \Gamma^\bb(t-s)\ast(H(\phi,\nabla\phi,u)(s)-H(\tilde{\phi},\nabla\tilde{\phi},\tilde{u})(s)) \|_{L^2}ds\\
 %\leq&\frac{1}{\beta} \int_0^{t-1} \|\nabla \Gamma^\bb(t-s)\|_{L^2}\|h(\phi,\nabla\phi)(s)g(u)(s)-h(\tilde{\phi},\nabla\tilde{\phi})(s)g(\tilde{u})(s) \|_{L^1}ds\\
\leq&\frac{1}{\beta}\int_0^{t}C \min\{1,{(t-s)}^{-(\frac{n}{4}+\frac{1}{2})}\}\|h(\phi,\nabla\phi)(s)g(u)(s)-h(\tilde{\phi},\nabla\tilde{\phi})(s)g(\tilde{u})(s) \|_{L^1}ds\\
%&\frac{1}{\beta} \int_0^{t-1} \|\nabla \Gamma^\bb(t-s)\ast(H(\phi,\nabla\phi,u)-H(\tilde{\phi},\nabla\tilde{\phi},\tilde{u})) 
%+&H_kG_k(M_{\phi}^{\frac{n}{4}}(t))\frac{1}{\beta} \int_0^{t-1}\min\{1,{t-s}^{-(\frac{n}{4}+\frac{1}{2})}\}\min\{1,s^{-\frac{n}{2})}\}ds\\
\leq&CH_kG_k(M_{D_x^1\phi}^{\frac{n}{4}}(t)+M_{\phi}^{\frac{n}{4}}(t))M_{u-\tilde{u}}^{\delta}\frac{1}{\beta} \int_0^{t}\min\{1,{(t-s)}^{-(\frac{n}{4}+\frac{1}{2})}\}\min\{1,s^{-\frac{n}{4}-\delta)}\}ds\\
+&CH_kG_k M_{\tilde{u}}^{\frac{n}{4}}(t)M_{\phi-\tilde{\phi}}^{\delta} \frac{1}{\beta} \int_0^{t}\min\{1,{(t-s)}^{-(\frac{n}{4}+\frac{1}{2})}\}\min\{1,s^{-\frac{n}{4}-\delta)}\}ds\\
+&CH_kG_k M_{\tilde{u}}^{\frac{n}{4}}(t)M_{D_x^1\phi-D_x^1\tilde{\phi}}^{\delta} \frac{1}{\beta} \int_0^{t}\min\{1,{(t-s)}^{-(\frac{n}{4}+\frac{1}{2})}\}\min\{1,s^{-\frac{n}{4}-\delta)}\}ds,
\end{align*}
where $\delta=\min\{ \frac{n}{4}+\frac{1}{2},\frac{n}{2}\}$.

Then, thanks to Lemma \ref{Lemma5.2} we deduce
\begin{align*}
\frac{1}{\beta} \int_0^{t} \|&\nabla\Gamma^\bb(t-s)\ast(H(\phi,\nabla\phi,u)(s)-H(\tilde{\phi},\nabla\tilde{\phi},\tilde{u})(s)) \|_{L^2}ds\\
\leq& C_1 (\min\{1,t^{-\theta_1}\}(H_kG_k( M_{\phi}^{\frac{n}{4}}(t)+M_{D_x^1 \phi}^{\frac{n}{4}}(t))M_{u-\tilde{u}}^{\delta}+M_{\tilde{u}}^{\frac{n}{4}}(t)M_{\phi-\tilde{\phi}}^{\delta}+M_{\tilde{u}}^{\frac{n}{4}}(t)M_{D_x^1\phi-D_x^1\tilde{\phi}}^{\delta}).
\end{align*}
where $\theta_1=\min\left\{\frac{n}{4}+\frac{1}{2},\frac{n}{4}+\delta, \frac{n}{2}+\delta-\frac{1}{2}\right\}$.

We estimate now the fourth term in \eqref{diff_uL2} as,
\begin{align*}
\int_0^t \|&(\mathcal{K}_{1}(t-s)+R_1(t-s)) H(\phi,\nabla\phi,u)(s)\|_{L^2}ds \\
\leq & \int_0^t Ce^{-c(t-s)}\|h(\phi,\nabla\phi)g(u)\|_{L^2}+C\min\{1,(t-s)^{-(\frac{n}{4}+1)}\}\|h(\phi,\nabla\phi)g(u)(s)\|_{L^1}ds.
\end{align*}

On the other hand the first term can be estimated as 
\begin{eqnarray*}
\int_0^t Ce^{-c(t-s)}\|h(\phi,\nabla\phi)g(u)\|_{L^2}ds&\leq& \int_0^t Ce^{-c(t-s)}H_kG_k(\|\phi(s)\|_{L^2}\|u(s)\|_{L^\infty}+\|\nabla{\phi}(s)\|_{L^2}\|u(s)\|_{L^\infty})ds\\
&\leq& C\min\{1,t^{-\frac{3}{4}n}\}(G_kH_kM_{D_x^1\phi}^{\frac{n}{4}}(t)N_{u}^{\frac{n}{2}}(t)+M_{\phi}^{\frac{n}{4}}(t)N_{u}^{\frac{n}{2}}(t)).
\end{eqnarray*}
While the second term is estimate by
\begin{align*}
\int_0^t C\min\{1,(t-s)^{-(\frac{n}{4}+1)}&\}\|h(\phi,\nabla\phi)(s)g(u)(s)\|_{L^1}ds\\
\leq &\int_0^tC\min\{1,(t-s)^{-(\frac{n}{4}+1)}\}G_kH_k(\|\phi(s)\|_{L^2}+\|\nabla{\phi}(s)\|_{L^2})\|u(s)\|_{L^2})ds\\
\leq & C_2\min\{1,t^{-\theta_2}\}H_kG_k(M_u^{\frac{n}{4}}(t)M_{\phi}^{\frac{n}{4}}(t)+M_u^{\frac{n}{4}}(t)M_{D_x^1\phi}^{\frac{n}{4}}(t)),
\end{align*}
where $\theta_2=\min\{\frac{n}{4}+1, \frac{n}{2}\}$.
In order to complete our estimate, we need to study the fifth integral term in \eqref{diff_uL2}, then proceeding as done before,
\begin{eqnarray*}
\int_0^t \|\Gamma^h_{1}(t-s)\ast \bar{B}(\phi,\nabla\phi)w(s) \|_{L^2}ds%&\leq& \int_0^t \|K_{1}(t-s)\ast\bar{b}(\phi,\nabla\phi)v(s) \|_{L^2}ds\\
%&+&\int_0^t\|\mathcal{K}_{1}(t-s)\ast\bar{b}(\phi,\nabla\phi)v(s) \|_{L^2}ds\\
&\leq& B_k\min\{1,t^{-(\frac{n}{4}+\frac{n}{2})}\}(M_\phi^{\frac{n}{4}}(t)N_{v}^{\frac{n}{2}}(t)+M_{D_x^1 \phi}^{\frac{n}{4}}(t)N_{v}^{\frac{n}{2}}(t))\\
&+&B_k\min\{1,t^{-\theta_3}\}B_K(M_{\phi}^{\frac{n}{4}}(t)M_{v}^{\nu_0}(t)+M_{D_x^1 \phi}^{\frac{n}{4}}(t)M_{v}^{\nu_0}(t)),
\end{eqnarray*}
where $\theta_3=\min\{\frac{n}{4}+\frac{1}{2},\frac{n}{2}\}$.

If we sum all the previous estimates, we get the following estimate for the difference of function $u$ and function $\tilde{u}$ in the $L^2$-norm.
\begin{eqnarray*}
\|u(t)-\tilde{u}(t)\|_{L^2}&\leq& C\left[e^{-ct}\|u_0\|_{L^2}+t^{-(\frac{n}{4}+\frac{1}{2})}\|u_0\|_{L^1}+e^{-ct}\sum_{i=1}^n\|v_0^i\|_{L^2}+\min\{1,t^{-(\frac{n}{4}+\frac{1}{2})}\}\sum_{i=1}^n\|v_0^i\|_{L^1}\right]\\
&+&C_k\left[\min\{1,t^{-\theta_1}\}((M_{\phi}^{\frac{n}{4}}(t)+M_{D_x^1 \phi}^{\frac{n}{4}}(t))M_{u-\tilde{u}}^{\delta}(t)+M_{\tilde{u}}^{\frac{n}{4}}(t)(M_{\phi-\tilde{\phi}}^{\delta}(t)+M_{D_x^1\phi-D_x^1\tilde{\phi}}^{\delta}(t))\right.\\
%&+&\min\{1,t^{-(\frac{n}{4}+\delta)}\}(M_{u-\tilde{u}}^{\delta}(t)N_{D_x^1\tilde{\phi}}^{\frac{n}{2}}(t)+M_{u-\tilde{u}}^{\delta}(t)N_{\tilde{\phi}}^{\frac{n}{2}}(t))\\
%&+&\min\{1,t^{-(\frac{n}{4}+\delta)}\}(M_{\phi-\tilde{\phi}}^{\delta}(t)N_{u}^{\frac{n}{2}}(t)+M_{D_x^1\phi-D_x^1\tilde{\phi}}^{\delta}(t)N_{u}^{\frac{n}{2}}(t))\\
&+&\min\{1,t^{-\frac{3}{4}n}\}(M_{\phi}^{\frac{n}{4}}(t)N_{u}^{\frac{n}{2}}(t)+M_{D_x^1\phi}^{\frac{n}{4}}(t)N_{u}^{\frac{n}{2}}(t))\\
&+&\min\{1,t^{-\theta_2}\}(M_u^{\frac{n}{4}}M_{\phi}^{\frac{n}{4}}+M_u^{\frac{n}{4}}M_{D_x^1\phi}^{\frac{n}{4}} )
+\left.\min\{1,t^{-\theta_3}\}(M_{\phi}^{\frac{n}{4}}M_{v}^{\nu_0}(t)+M_{D_x^1 \phi}^{\frac{n}{4}}M_{v}^{\nu_0}(t))\right.,
\end{eqnarray*}
where $\delta=\min\{\frac{n}{4}+\frac{1}{2},\frac{n}{2}\}$, $\theta_1=\min\{\frac{n}{4}+\frac{1}{2},\frac{n}{4}+\delta,\frac{n}{2}-\frac{1}{2}+\delta\}$, $\theta_2=\min\{\frac{n}{4}+1,\frac{n}{2}\}$, and $\theta_3=\min\{\frac{n}{4}+\frac{1}{2},\frac{n}{2}\}$. \\
%%%%%%%%%%%%%%%%%%%%%%%%%%%%%%%%%%%%%%%%%%%STIME PER LA DIFFERENZA DI \PHI
Let us now focus on the function $\phi$. Arguing as in Proposition \ref{phi_estimates}, it is easy to show that 
the difference of the second variables is given by 
\begin{eqnarray*}
\|\phi(t)-\tilde{\phi}(t)\|_{L^2}&\leq& \int_0^t \|e^{-b(t-s)}\Gamma^p(t-s)\ast(au(s)-a\tilde{u}(s)+\bar{f}(u,\phi)(s)-\bar{f}(\tilde{u},\tilde{\phi})(s))\|_{L^2}ds\\
%&\leq& \int_0^t Ce^{-b(t-s)}(\|au(s)-a\tilde{u}(s)\|_{L^2}+\|\bar{f}(u,\phi)(s)-\bar{f}(\tilde{u},\tilde{\phi})(s)\|_{L^2}ds\\
&\leq& \int_0^t e^{-b(t-s)}C_k(\|u(s)-\tilde{u}(s)\|_{L^2}+\|\phi(s)-\tilde{\phi}(s)\|_{L^2})ds\\
&\leq& C_k (M_{u-\tilde{u}}^{\delta}(t)+M_{\phi-\tilde{\phi}}^{\delta}(t))\int_0^t e^{-b(t-s)}\min\{1,s^{-\delta}\}ds\\
&\leq& C_k \min\{1,t^{-\delta}\}(M_{u-\tilde{u}}^{\delta}(t)+M_{\phi-\tilde{\phi}}^{\delta}(t)).
\end{eqnarray*}

Then, for small initial data we have
\begin{equation}\label{diff_funz_phiL2_a}
M_{\phi-\tilde{\phi}}^{\delta}(t)\leq C_{k}M_{u-\tilde{u}}^{\delta}(t),
\end{equation}
where $C_k$ is a constant depending on $K$.
Proceeding in a similar way we get also 
\begin{eqnarray}\label{diff_funz_DphiL2_b}
M_{D_x^1\phi-D_x^1\tilde{\phi}}^{\delta}\leq C_{2K} M_{u-\tilde{u}}^{\delta}.
\end{eqnarray}
Then by using the known decays of the $L^2$ and $L^\infty$-norms of $u,\tilde{u}$, $\phi,\tilde{\phi}$, $\nabla\phi, \nabla\tilde{\phi}$, from inequalities in \eqref{diff_funz_phiL2_a} and \eqref{diff_funz_DphiL2_b} we obtain
\begin{eqnarray*}
M_{u-\tilde{u}}^{\delta}(t)&\leq& C_0\left(\|u_0\|_{L^2}+\sum_{i=1}^n\|v_0^i\|_{L^2}+\|u_0\|_{L^1}+\sum_i\|v_0^i\|_{L^1}\right)\\
&+&C_{1k}\left[M_{u-\tilde{u}}^{\delta}(t)(M_{\phi}^{\frac{n}{4}}(t)+M_{D_x^1\phi}^{\frac{n}{4}}(t)+ M_{\tilde{u}}^{\frac{n}{4}}(t))\right]\\
&+&C_{2k}\left[M_{\phi}^{\frac{n}{4}}(t)N_{\tilde{\phi}}^{\frac{n}{2}}(t)+M_{u}^{\frac{n}{4}}(t)N_{u}^{\frac{n}{2}}(t)\right.\\
&+&M_u^{\frac{n}{4}}(t)N_{u}^{\frac{n}{2}}(t)+M_{\phi}^{\frac{n}{4}}(t)N_{u}^{\frac{n}{2}}(t)+M_{D_x^1\phi}^{\frac{n}{4}}(t)N_{u}^{\frac{n}{2}}(t)+M_{\phi}^{\frac{n}{4}}(t)M_{u}^{\frac{n}{4}}(t)\\
&+&\left.M_{D_x^1 \phi}^{\frac{n}{4}}(t)M_{u}^{\frac{n}{4}}(t)+M_{\phi}^{\frac{n}{4}}(t)M_{v}^{\nu_0}(t)+M_{D_x^1\phi}^{\frac{n}{4}}(t)M_{v}^{\nu_0}(t)\right],
\end{eqnarray*}

where $C_{1k}$ and $C_{2k}$ are positive constant depending on $K$.\\
Now for small $M_{\phi}^{\frac{n}{4}}(t)$, $M_{D_x^1\phi}^{\frac{n}{4}}(t)$ and $M_{\tilde{u}}^{\frac{n}{4}}(t)$, or $K$, i.e. for small initial data, we have a global bound for $M_{u-\tilde{u}}^{\delta}(t)$, with $\delta=\min\{\frac{n}{4}+\frac{1}{2}, \frac{n}{2}\}$, and of course for the functional
$M_{\phi-\tilde{\phi}}^{\delta}(t)$.   
\end{proof}

\section*{Acknowledgments} 

The author would like to thank R. Natalini for several helpful discussions, for spending time reading and revising the paper, and thus considerably improving the presentation. This work has been partially supported by ANR project MONUMENTALG (ANR-10-JCJC 0103) and by ANR project KIBORD (ANR-13-BS01-0004).

\bibliographystyle{plain}

\bibliography{bibliografia}	

\end{document}